\documentclass[a4paper,11pt]{article}
\usepackage{amsmath}
\usepackage{amssymb}
\usepackage{amsthm}
\usepackage{MnSymbol}
\usepackage{mathrsfs}

\usepackage{hyperref,url}
\usepackage{bm} 
\usepackage{color}
\usepackage[margin=30mm]{geometry}

\newcommand{\N}{\mathbb{N}}

\newcommand{\R}{\mathbb{R}}

\newcommand{\Optime}{\tau_{*}}

\renewcommand{\div}{\, \mathrm{div} \, }
\newcommand{\PP}{\mathcal{P}}
\renewcommand{\AA}{\mathcal{A}}
\newcommand{\CC}{\mathcal{C}}
\newcommand{\BB}{\mathcal{B}}
\newcommand{\Soln}{\mathcal{S}}

\newcommand{\dx}{\, \mathrm{dx}}
\newcommand{\ds}{\, \mathrm{ds}}
\newcommand{\dt}{\, \mathrm{dt}}
\newcommand{\dd}{\, \mathrm{d}}
\newcommand{\dz}{\, \mathrm{dz}}
\newcommand{\dr}{\, \mathrm{dr}}

\newcommand{\der}{\mathrm{D}}
\newcommand{\pd}{\partial}
\newcommand{\pdnu}{\partial_{\bm{\nu}}}

\newcommand{\abs}[1]{\left| #1 \right|}
\newcommand{\norm}[1]{\| #1 \|}
\newcommand{\bignorm}[1]{\left\| #1 \right\|}
\newcommand{\inner}[2]{\langle #1 , #2 \rangle}

\newcommand{\Laplace}{\Delta}

\theoremstyle{plain}
\newtheorem{thm}{Theorem}[section]

\newtheorem{lemma}[thm]{Lemma}

\theoremstyle{plain}
\newtheorem{remark}{Remark}[section]

\newtheorem{assump}{Assumption}[section]
\numberwithin{equation}{section}

\title{Optimal control of treatment time in a diffuse interface model of
tumor growth}

\date{\today}
\author{Harald Garcke \footnotemark[1] \and Kei Fong Lam\footnotemark[1] \and Elisabetta Rocca \footnotemark[2]}

\begin{document}

\maketitle

\renewcommand{\thefootnote}{\fnsymbol{footnote}}
\footnotetext[1]{Fakult\"at f\"ur Mathematik, Universit\"at Regensburg, 93040 Regensburg, Germany
({\tt \{Harald.Garcke, Kei-Fong.Lam\}@mathematik.uni-regensburg.de}).}
\footnotetext[2]{
Universit\`{a} degli Studi di  Pavia, Dipartimento di Matematica, Via Ferrata 5, 27100, Pavia, Italy
({\tt elisabetta.rocca@unipv.it}).}

\begin{abstract}
We consider an optimal control problem for a diffuse interface model of tumor growth.  The state equations couples a Cahn-Hilliard equation and a reaction-diffusion equation, which models the growth of a tumor in the presence of a nutrient and surrounded by host tissue.  The introduction of cytotoxic drugs into the system serves to eliminate the tumor cells and in this setting the concentration of the cytotoxic drugs will act as the control variable.  Furthermore, we allow the objective functional to depend on a free time variable, which represents the unknown treatment time to be optimized.  As a result, we obtain first order necessary optimality conditions for both the cytotoxic concentration and the treatment time. 
\end{abstract}



\section{Introduction}
There has been a recent surge in the development of phase field models for tumor growth.  These models aim to describe the evolution of a tumor colony surrounded by healthy tissues which experience biological mechanisms such as proliferation via nutrient consumption, apoptosis, chemotaxis and active transport of specific chemical species.  For the case of a young tumor, before the development of quiescent cells, the phase field models often consist of a Cahn--Hilliard equation coupled with a reaction-diffusion equation for the nutrient \cite{article:CristiniLiLowengrubWise09,article:GarckeLamSitkaStyles,article:HawkinsPrudhommevanderZeeOden13,
article:HawkinsZeeOden12,article:OdenHawkinsPrudhomme10}.  One may also treat the tumor cells and the healthy cells as inertia-less fluids, and include the effects of fluid flow into the evolution of the tumor, leading to the development of a Cahn--Hilliard--Darcy system \cite{book:CristiniLowengrub,article:DFRSS,article:GarckeLamSitkaStyles,article:WiseLowengrubFrieboesCristini08}.

Current treatments for cancer include surgery, immunotherapy (strengthening the immune system), radiotherapy (using radiation to kill cancer cells) and chemotherapy (using drugs to kill cancer cells).  The latter three treatments are typically conducted in cycles.  A cycle is a period of treatment followed by a (longer) period of rest, so that the patient's body can build new healthy cells.  The goal of these therapeutic treatments is to shrink the tumor into a more manageable size for which surgery can be applied.  Further therapeutic treatments may be necessary in order to destroy the cancer cells that may remain after the surgery. 

In this work, we consider an optimal control problem involving a cancer treatment with cytotoxic drugs.  It is well-known that while cytotoxic drugs mainly target and damage rapidly dividing cells such as tumor cells, the drugs can also accumulate in the body and cause adverse side-effects to the immune system and various vital organs such as the kidneys and the liver.  In a worst case scenario, too much cytotoxic drugs may allow tumor cells to mutate and become resistant to the treatment.  Thus, from the viewpoint of the patient, the shortest treatment time in which the objectives of the chemotherapy are achieved is the most ideal.   Therefore, the optimal control problem we study involves finding the optimal drug distribution and the optimal treatment time.

For $T > 0$, in a bounded domain $\Omega \subset \R^{3}$ with $C^{3}$-boundary $\Gamma$, we consider the following Cahn--Hilliard model for tumor growth,
\begin{subequations}\label{eq:state}
\begin{alignat}{3}
\pd_{t}\varphi & = \Laplace \mu + (\PP \sigma - \AA - \alpha u) h(\varphi) && \text{ in } \Omega \times (0,T) =: Q, \label{state:varphi} \\
\mu & = A \Psi'(\varphi) - B \Laplace \varphi && \text{ in } Q, \label{state:mu} \\
\pd_{t} \sigma & = \Laplace \sigma - \CC \sigma h(\varphi) + \BB (\sigma_{S} - \sigma) && \text{ in } Q, \label{state:sigma} \\
\pdnu \varphi  &  = \pdnu \mu = \pdnu \sigma  = 0 && \text{ on } \Gamma \times (0,T), \\
\varphi(0) & = \varphi_{0}, \; \sigma(0) = \sigma_{0} && \text{ in } \Omega.
\end{alignat}
\end{subequations}
Here, $\alpha$ is a positive constant, $\varphi$ denotes the difference in volume fraction, where $\varphi = 1$ represents the tumor phase and $\varphi = -1$ represents the healthy tissue phase.  The function $\mu$ is a chemical potential associated to $\varphi$, $\Psi'(\varphi)$ is the derivative of a potential $\Psi(\varphi)$ with equal minima at $\varphi = \pm 1$, $\sigma$ denotes the concentration of an unspecified chemical species acting as nutrient for the tumor cells, while $u$ denotes the concentration of cytotoxic drugs.  

The function $h(\varphi)$ is an interpolation function such that $h(-1) = 0$ and $h(1) = 1$, and the parameters $\PP$, $\AA$, $\CC$, and $\BB$ denote the constant proliferation rate, apoptosis rate, nutrient consumption rate, and nutrient supply rate, respectively.  The positive constants $A$ and $B$ are related to the thickness of the interfacial layer and the surface tension, while $\pdnu f = \nabla f \cdot \bm{\nu}$ denotes the normal derivative of $f$ where $\bm{\nu}$ is the unit outward normal of $\Gamma$.  

The term $h(\varphi) \PP \sigma$ models the proliferation of tumor cells which is proportional to the concentration of the nutrient, the term $h(\varphi) \AA$ models the apoptosis of tumor cells, and $\CC h(\varphi) \sigma$ models the consumption of the nutrient only by the tumor cells.  The term $\alpha u h(\varphi)$ models the elimination of the tumor cells by the cytotoxic drugs at a constant rate $\alpha$.  Meanwhile, $\sigma_{S}$ denotes the nutrient concentration in a pre-existing vasculature, and $\BB(\sigma_{S} - \sigma)$ models the supply of nutrient from the blood vessels if $\sigma_{S} > \sigma$ and the transport of nutrient away from the domain $\Omega$ if $\sigma_{S} < \sigma$.

In comparison with the models of \cite{article:GarckeLamSitkaStyles}, we have neglected the effects of chemotaxis and active transport, but the new feature of \eqref{eq:state} is the inclusion of the effects of cytotoxic drugs via the term $\alpha u h(\varphi)$, and in this work the function $u$ will act as our control.  For realistic applications the control $u:[0,T] \to [0,1]$ should be spatially constant, where $u = 1$ represents a full dosage and $u = 0$ represents no dosage.  However, in the subsequent analysis, we allow for spatial dependence (see Assumption \ref{assump:main} below). 

For positive constants $r$, $\beta_{u}$ and $\beta_{T}$, and nonnegative constants $\beta_{Q}$, $\beta_{\Omega}$, and $\beta_{S}$, we consider the objective functional $J_{r}$ given as
\begin{equation}\label{eq:cost}
\begin{aligned}
J_{r}(\varphi, u, \tau) & = \frac{\beta_{Q}}{2} \int_{0}^{\tau} \int_{\Omega} \abs{\varphi - \varphi_{Q}}^{2} \dx \dt + \frac{\beta_{\Omega}}{2} \frac{1}{r} \int_{\tau-r}^{\tau} \int_{\Omega} \abs{\varphi - \varphi_{\Omega}}^{2} \dx  \dt \\
& + \frac{\beta_{S}}{2} \frac{1}{r} \int_{\tau-r}^{\tau} \int_{\Omega} 1 + \varphi \dx \dt + \frac{\beta_{u}}{2} \int_{0}^{T} \int_{\Omega} \abs{u}^{2} \dx \dt + \beta_{T} \tau.
\end{aligned}
\end{equation}
In particular, \eqref{eq:cost} can be seen as the relaxation of the following more natural objective functional
\begin{equation}\label{eq:originalcost}
\begin{aligned}
J(\varphi, u, \tau) & = \frac{\beta_{Q}}{2} \int_{0}^{\tau} \int_{\Omega} \abs{\varphi - \varphi_{Q}}^{2} \dx \dt + \frac{\beta_{\Omega}}{2} \int_{\Omega} \abs{\varphi(\tau) - \varphi_{\Omega}}^{2} \dx  \\
& + \frac{\beta_{S}}{2} \int_{\Omega} 1 + \varphi(\tau) \dx + \frac{\beta_{u}}{2} \int_{0}^{\tau} \int_{\Omega} \abs{u}^{2} \dx \dt + \beta_{T} \tau.
\end{aligned}
\end{equation}
Here, $\tau \in (0,T]$ represents the treatment time, $\varphi_{Q}$ represents a desired evolution for the tumor cells while $\varphi_{\Omega}$ represents a desired final distribution.  The first two terms of $J$ are of standard tracking type as often considered in the literature of parabolic optimal control, and the third term of $J$ measures the size of the tumor at the end of the treatment.  The fourth term penalizes large concentrations of the cytotoxic drugs, and the fifth term of $J$ penalizes long treatment times.  

Let us make the following comments:
\begin{enumerate}
\item A large value of $\abs{\varphi - \varphi_{Q}}^{2}$ would mean that the patient suffers from the growth of the tumor, and a large value of $\abs{u}^{2}$ would mean that the patient suffers from high toxicity of the drug.  
\item The function $\varphi_{\Omega}$ can be a stable configuration of the system, so that the tumor does not grow again once the treatment is completed.  One can also choose $\varphi_{\Omega}$ as a configuration which is suitable for surgery.
\item The variable $\tau$ can be regarded as the treatment time of one cycle, i.e., the amount of time the drug is applied to the patient before the period of rest, or the treatment time before surgery.
\item It is possible to replace $\beta_{T} \tau$ by a more general function $f(\tau)$ where $f: \R_{\geq 0} \to \R_{\geq 0}$ is continuously differentiable and increasing.  
\item We consider $T \in (0,\infty)$ as a fixed maximal time in which the patient is allowed to undergo a treatment obtained from this optimal control problem.
\end{enumerate} 

For technical reasons highlighted below, we consider an optimal control problem with the relaxed objective functional \eqref{eq:cost} and the state equations \eqref{eq:state}.  We denote the space of admissible controls as $\mathcal{U}_{\mathrm{ad}}$ (see Assumption \ref{assump:main} below) and the optimal control problem we study in this work can be expressed as follows,
\begin{equation}\tag{$\mathrm{P}$}\label{OCProblem}
\begin{aligned}
\text{ minimise } J_{r}(\varphi, u, \tau) \text{ subject to } \eqref{eq:state}, \; u \in \mathcal{U}_{\mathrm{ad}}, \; \tau \in (0,T).
\end{aligned}
\end{equation}
The optimal control problem \eqref{OCProblem} is a problem involving a free terminal time, and we say that $(u_{*}, \Optime)$ is a minimizer of \eqref{OCProblem} if
\begin{align*}
J_{r}(\varphi_{*}, u_{*}, \Optime) & = \inf J_{r}(\phi, w, s),
\end{align*}
where the infimum is taken over triplets $(\phi, w, s)$ such that $w \in \mathcal{U}_{\mathrm{ad}}$, $s \in [0,T]$ and $\phi$ solves \eqref{eq:state} with datum $w$.  In ODE constrained optimal control where the cost functional depends on the free terminal time, the necessary optimality condition can be derived with the help of the corresponding Hamiltonian function, see for instance \cite[Chapter 20]{book:LenhartWorkman} and \cite{article:HartlSethi,article:JangKwonLee,article:PalankiKravarisWang}.  One may use the notion of Hamiltonian functional to derive the optimality condition for the free terminal time when the state equations are partial differential equations, see in particular \cite{article:AradaRaymond,article:RaymondZidani,article:RaymondZidani2} for semilinear parabolic state equations. 

Below we illustrate with an example the optimality conditions obtained with the Hamiltonian from ODE theory and with the Lagrangian method for PDE-constrainted optimization, see for instance \cite[\S 2.10]{book:Troeltzsch} and \cite[\S 1.6.4]{book:Hinze}.  Suppose the objective functional is of the form
\begin{align*}
\int_{0}^{\tau} \int_{\Omega} F(t, \varphi(t,x), u(t,x)) \dx \dt + \int_{\Omega} L(\tau, \varphi(\tau,x)) \dx,
\end{align*}
and $\varphi$ satisfies for example
\begin{align*}
\pd_{t}\varphi = \Laplace \varphi + f(t,\varphi, u).
\end{align*}
Let $u_{*}$ denote an optimal control with corresponding state $\varphi_{*}$.  The Hamiltonian $H$ is defined as
\begin{align*}
H(t,\varphi, u, p) := \int_{\Omega} F(t,\varphi, u) \dx + \int_{\Omega} p \Laplace \varphi + p f(t,\varphi, u) \dx,
\end{align*}
where $p$ act as the adjoint variable to $\varphi_{*}$.  From the works of \cite{article:AradaRaymond,article:RaymondZidani,article:RaymondZidani2} and also from the theory of ODE-constraint optimal control, the optimality condition for the optimal time $\Optime$ is
\begin{equation}\label{HamiltonianOC}
\begin{aligned}
0 & = H(\Optime, \varphi_{*}(\Optime), u_{*}(\Optime), p(\Optime)) + \int_{\Omega} \frac{\pd L}{\pd t}(\Optime, \varphi_{*}(\Optime)) \dx.
\end{aligned}
\end{equation}
Now, let us define the Lagrangian
\begin{align*}
\mathcal{L} := \int_{0}^{\tau} \int_{\Omega} F(t,\varphi, u) \dx + \int_{\Omega} L(\tau,\varphi(\tau)) \dx - \int_{0}^{\tau} \int_{\Omega} p \left ( \pd_{t}\varphi  - \Laplace \varphi - f(t,\varphi, u) \right ) \dx,
\end{align*}
then one obtains from formally differentiating $\mathcal{L}$ with respect to $\tau$ the optimality condition for $\Optime$, which is
\begin{align*}
\frac{\pd \mathcal{L}}{\pd \tau}(\Optime, \varphi_{*}, u_{*}) & = \int_{\Omega} F(\Optime, \varphi_{*}(\Optime), u_{*}(\Optime)) +  \frac{\pd L}{\pd t}(\Optime, \varphi_{*}(\Optime)) + \frac{\pd L}{\pd \varphi}(\Optime, \varphi_{*}(\Optime)) \pd_{t}\varphi_{*}(\Optime) \dx \\
& - \int_{\Omega} p(\Optime) (\pd_{t}\varphi_{*}(\Optime) - \Laplace \varphi_{*}(\Optime) - f(\Optime, \varphi_{*}(\Optime), u_{*}(\Optime))) \dx = 0.
\end{align*}
The adjoint equation for $p$ is a terminal time boundary value problem:
\begin{align*}
-\pd_{t}p = \Laplace p + \frac{\pd f}{\pd \varphi}p + \frac{\pd F}{\pd \varphi}, \quad p(\Optime) = \frac{\pd L}{\pd \varphi}(\Optime, \varphi_{*}(\Optime)).
\end{align*}
Using the terminal condition for $p$ in the expression for $\frac{\pd \mathcal{L}}{\pd \tau}(\Optime, \varphi_{*}, u_{*})$ we see that 
\begin{align*}
\frac{\pd L}{\pd \varphi}(\Optime,\varphi_{*}(\Optime)) \pd_{t}\varphi_{*}(\Optime) - p(\Optime) \pd_{t}\varphi_{*}(\Optime) = 0,
\end{align*}
and thus $\frac{\pd \mathcal{L}}{\pd \tau}(\Optime, \varphi_{*}, u_{*}) = 0$ is equivalent to \eqref{HamiltonianOC}.  That is, the optimality conditions for the free terminal time obtain from the Hamiltonian formulation and the Lagrangian formulation coincide.

Let us briefly explain the issues with the objective functional \eqref{eq:originalcost}.  Formally differentiating \eqref{eq:originalcost} with respect to $\tau$, we obtain
\begin{equation}\label{formalpdJpdtau}
\begin{aligned}
\frac{\pd J}{\pd \tau}(\varphi_{*}, u_{*}, \Optime) &= \frac{\beta_{Q}}{2} \int_{\Omega} \abs{\varphi_{*}(\Optime) - \varphi_{Q}(\Optime)}^{2} \dx + \beta_{\Omega} \int_{\Omega} (\varphi_{*}(\Optime) - \varphi_{\Omega})\pd_{t}\varphi_{*}(\Optime) \dx \\
& + \frac{\beta_{S}}{2} \int_{\Omega} \pd_{t}\varphi_{*}(\Optime) \dx + \frac{\beta_{u}}{2} \int_{\Omega} \abs{u_{*}(\Optime)}^{2} \dx + \beta_{T},
\end{aligned}
\end{equation}
and in order for the terms in \eqref{formalpdJpdtau} to be well-defined, we need that 
\begin{align*}
u_{*}, \varphi_{Q}, \varphi_{*}, \pd_{t}\varphi_{*} \in C^{0}([0,T];L^{2}(\Omega)). 
\end{align*}
Furthermore, to rigorously establish the Fr\'{e}chet differentiability of $J$ with respect to $\tau$, it turns out that we require 
\begin{align*}
\pd_{tt}\varphi_{*} \in L^{2}(0,T;L^{2}(\Omega)), \text{ and } u_{*}, \varphi_{Q} \in H^{1}(0,T;L^{2}(\Omega)).
\end{align*}
Thus, the main mathematical difficulties arise from establishing high temporal regularity for the state variables.  A preliminary analysis shows that it is possible to derive such regularity but only under rather strong assumptions such as $\varphi_{0} \in H^{5}(\Omega)$, $\sigma_{0} \in H^{3}(\Omega)$ and $\norm{\pd_{t}u}_{L^{2}(0,T;L^{2}(\Omega))} \leq K$ for some fixed $K > 0$.  The assumption on the a priori boundedness of $\pd_{t}u$ is not meaningful as in applications it will be hard to verify this condition.  Furthermore, using the Lagrangian method, one can compute that the terminal condition for the adjoint variable $p$ to $\varphi_{*}$ is $p(\Optime) = \beta_{\Omega}(\varphi_{*}(\Optime) - \varphi_{\Omega}) + \frac{\beta_{S}}{2}$, and so we can write \eqref{formalpdJpdtau} more compactly as
\begin{align*}
\frac{\pd J}{\pd \tau}(\varphi_{*}, u_{*},\Optime) =  \int_{\Omega} \frac{\beta_{Q}}{2}\abs{\varphi_{*}(\Optime) - \varphi_{Q}(\Optime)}^{2} + \frac{\beta_{u}}{2} \abs{u_{*}(\Optime)}^{2} \dx + \int_{\Omega} p(\Optime) \pd_{t}\varphi(\Optime) \dx + \beta_{T}.
\end{align*}
But this would mean that we require the weak formulation for the equation of $\varphi_{*}$ to be satisfied pointwise in $[0,T]$, that is,
\begin{align*}
\int_{\Omega} p(t) \pd_{t}\varphi_{*}(t) + \nabla \mu_{*}(t) \cdot \nabla p(t) - h(\varphi_{*}(t))(\PP \sigma_{*}(t) - \AA - \alpha u_{*}(t)) p(t) \dx = 0
\end{align*}
holds for all $t \in [0,T]$.  This in turn implies that we need 
\begin{align*}
p \in C^{0}([0,T];H^{1}(\Omega)), \; \pd_{t}\varphi_{*} \in C^{0}([0,T];L^{2}(\Omega)), \; \mu_{*} \in C^{0}([0,T];H^{1}(\Omega)).
\end{align*}
These difficulties motivates the current study with the relaxed objective functional $J_{r}$ \eqref{eq:cost}.

There have been many recent contributions regarding the well-posedness and asymptotic behaviour for phase field type tumor models, see for example \cite{article:ColliGilardiHilhorst15,article:ColliGilardiRoccaSprekels1,article:ColliGilardiRoccaSprekels2,article:FrigeriGrasselliRocca15,article:GarckeLamNeumann,article:GarckeLamDirichlet}
for the Cahn--Hilliard variant, and \cite{preprint:BosiaContiGrasselli14,article:DFRSS,article:GarckeLamCHND,preprint:JiangWuZheng14,article:LowengrubTitiZhao13} for the Cahn--Hilliard--Darcy variant.  From the aspect of optimal control, we mention the works of \cite{article:ColliFarshbafGilardiSprekels,article:ColliFarshbafGilardiSprekels2ndorder,article:ColliGilardiSprekels,article:ColliGilardiSprekelsViscous,article:HintermullerWegner12,article:ZhaoLiuViscous1D} for the Cahn--Hilliard equation, \cite{article:RoccaSprekels,article:ZhaoDuan6thorderCH,article:ZhaoLiuConvec1D,article:ZhaoLiuConvec2D} for the convective Cahn--Hilliard equation and \cite{article:FrigeriRoccaSprekels,article:HintermullerKeilWegner,article:HintermullerWegnerGL,article:HintermullerWegnerCHNS} for the Cahn--Hilliard--Navier--Stokes system.  In the context of PDE constraint optimal control for diffuse interface tumor models, we have the recent work of \cite{article:CGRS}, where the objective functional \eqref{eq:originalcost} with $\beta_{S} = \beta_{T} = 0$ and no dependence of on $\tau$ is studied with state equations given by the model of \cite{article:HawkinsZeeOden12} and the control enters the nutrient equation as a source term, similar to the term $\BB (\sigma_{S} - \sigma)$ in \eqref{state:sigma}.  With this work we aim to provide a contribution to the theory of free terminal time optimal control in the context of diffuse interface tumor models.

Let us provide some future directions of research motivated by this study:
\begin{enumerate}
\item An optimal control $u$ that is periodic in time, reflecting the cyclic nature of therapeutic treatments.
\item A feedback mechanism taking into account the patient's response to the therapy, and the  tumor's resistance to the drug.
\item Analysis and identification of stable equilibria for diffuse interface models of tumor growth.
\end{enumerate}

\paragraph{Plan of the paper.} The paper is organised as follows.  In Section \ref{sec:mainresults} the general assumptions are outlined and the main results are stated.  The well-posedness of the state equations \eqref{eq:state} is established in Section \ref{sec:stateequations}.  The existence of a minimizer to \eqref{OCProblem} is proved in Section \ref{sec:minimizer}, while the unique solvability of the linearized state equations and the Fr\'{e}chet differentiability of the control-to-state mapping and of the functional $J_{r}$ are contained in Section \ref{sec:Frechetdiff}.  In Section \ref{sec:FONC}, the unique solvability of the adjoint equations is studied and the first order necessary optimality conditions are derived. 

\section{General assumptions and main results}\label{sec:mainresults}

\paragraph{Notation.}
For convenience, we will often use the notation $L^{p} := L^{p}(\Omega)$ and $W^{k,p} := W^{k,p}(\Omega)$ for any $p \in [1,\infty]$, $k > 0$ to denote the standard Lebesgue spaces and Sobolev spaces equipped with the norms $\norm{\cdot}_{L^{p}}$ and $\norm{\cdot}_{W^{k,p}}$.  In the case $p = 2$ we use $H^{k} := W^{k,2}$ and the norm $\norm{\cdot}_{H^{k}}$.  Moreover, the dual space of a Banach space $X$ will be denoted by $X^{*}$, and the duality pairing between $X$ and $X^{*}$ is denoted by $\inner{\cdot}{\cdot}_{X}$.  The space-time cylinder $\Omega \times (0,T)$ will be denoted by $Q$, and we use the notation $L^{p}(Q)$ to denote the spaces $L^{p}(\Omega \times (0,T))$ for $1 \leq p \leq \infty$.  Using Fubini's theorem we have the isometric isomorphism $L^{p}(0,T;L^{p}) \cong L^{p}(Q)$ for $p \in [1,\infty)$.  We point out that $L^{\infty}(0,T;L^{\infty}) \subset L^{\infty}(Q)$, but the converse inclusion is not true in general due to measurability issues (see for instance \cite[Example 1.4.2]{book:Roubieck}).

\paragraph{Useful preliminaries.}
The following Gronwall inequality in integral form will often be used (see \cite[Lemma 3.1]{article:GarckeLamNeumann} for a proof).  For $W$, $X$, $Y$, $Z$ real-valued functions defined on $[0,T]$ such that $W$ is integrable, $X$ is nonnegative and continuous, $Y$ is continuous, $Z$ is nonnegative and integrable.  If $Y$ and $Z$ satisfy the integral inequality
\begin{align*}
Y(s) + \int_{0}^{s} Z(t) \dt \leq W(s) + \int_{0}^{s} X(t) Y(t) \dt \text{ for } s \in (0,T],
\end{align*}
then it holds that
\begin{align}\label{Gronwall}
Y(s) + \int_{0}^{s} Z(t) \dt \leq W(s) + \int_{0}^{s} X(t) W(t) \exp \left ( \int_{0}^{t} X(r) \dr \right ) \dt.
\end{align}

The following Taylor's theorem with integral remainder will be used to show the Fr\'{e}chet differentiability of the control-to-state mapping.  For $f \in C^{2}(\R)$ and $a,x \in \R$, it holds that
\begin{align}\label{Taylor}
f(x) = f(a) + f'(a)(x-a) + (x-a)^{2} \int_{0}^{1} f''(a + z(x-a)) (1-z) \dz.
\end{align}

The Gagliardo--Nirenberg interpolation inequality in dimension $d$ is also useful (see \cite[Theorem 10.1, p. 27]{book:Friedman}):  Let $\Omega$ be a bounded domain with $C^{m}$ boundary, and $f \in W^{m,r}(\Omega) \cap L^{q}(\Omega)$, $1 \leq q,r \leq \infty$.  For any integer $j$, $0 \leq j < m$, suppose there is an $\alpha \in \R$ such that
\begin{align*}
\frac{1}{p} = \frac{j}{d} + \left ( \frac{1}{r} - \frac{m}{d} \right ) \alpha + \frac{1-\alpha}{q}, \quad \frac{j}{m} \leq \alpha \leq 1.
\end{align*}
There exists a positive constant $C$ depending only on $\Omega$, $m$, $j$, $q$, $r$, and $\alpha$ such that
\begin{align}
\label{GagNirenIneq}
\norm{D^{j} f}_{L^{p}(\Omega)} \leq C \norm{f}_{W^{m,r}(\Omega)}^{\alpha} \norm{f}_{L^{q}(\Omega)}^{1-\alpha} .
\end{align}

We consider the following assumptions.
\begin{assump}\label{assump:main}
\
\begin{itemize}
\item[$(\mathrm{A1})$] The initial conditions satisfy $\varphi_{0} \in H^{3}$ with compatibility condition $\pdnu \varphi_{0} = 0$ on $\Gamma$, $\sigma_{0} \in H^{1}$, with $0 \leq \sigma_{0} \leq 1$ a.e. in $\Omega$, while the target functions satisfy $\varphi_{Q}, \varphi_{\Omega} \in L^{2}(Q)$.  The vasculature nutrient concentration $\sigma_{S}$ satisfies $0 \leq \sigma_{S} \leq 1$ a.e. in $Q$.
\item[$(\mathrm{A2})$] The interpolation function $h : \R \to [0,1]$ is twice continuously differentiable and Lipschitz continuous (with Lipschitz constant $L_{h}$).  The parameters $\PP$, $\AA$, $\CC$, and $\BB$ are nonnegative constants, and $\alpha$ is a positive constant.  
\item[$(\mathrm{A3})$] The space of admissible controls is given as 
\begin{align*}
\mathcal{U}_{\mathrm{ad}} = \left \{ u \in L^{\infty}(0,T;L^{\infty}) : 0 \leq u \leq 1 \text{ a.e. in } Q \right \}.
\end{align*}
\item[$(\mathrm{A4})$] The potential $\Psi: \R \to \R_{\geq 0}$ is three times continuously differentiable and satisfies for some positive constants $\{k_{j}\}_{j=0}^{5}$,
\begin{align}
\abs{\Psi'(s)} & \leq k_{0} \Psi(s) + k_{1},\label{Psi'Psi}\\
\Psi(s) & \geq k_{2} \abs{s} - k_{3}, \label{PsiLower} \\
\abs{\Psi''(s)} & \leq k_{4}(1 + \abs{s}^{2}), \label{Psi:growth} \\
\abs{\Psi'(s) - \Psi'(t)} & \leq k_{5}(1 + \abs{s}^{2} + \abs{t}^{2})\abs{s - t}, \label{Psi'diff}
\end{align}
for all $s, t \in \R$.
\end{itemize}
\end{assump}
We point out that as $\Omega$ is a bounded domain, there exists an open set $\mathcal{U} \subset L^{2}(Q)$ such that $\mathcal{U}_{\mathrm{ad}} \subset \mathcal{U}$.  In this work, we consider quartic potentials $\Psi$ for the state equations, for which the classical double-well potential $\Psi(s) = \frac{1}{4}(1-s^{2})^{2}$ is one example.  The well-posedness of the state equations with higher polynomial growth for $\Psi$ is also possible, see for instance the procedure in \cite[Proof of Theorem 1]{article:FrigeriGrasselliRocca15}, but we restrict our current analysis to that of quartic potentials to simplify the computations. 

\begin{thm}[Well-posedness]\label{thm:state}
For every $T \in (0,\infty)$ and given data $(\varphi_{0}, \sigma_{0}, u)$, under Assumption \ref{assump:main} there exists a unique triplet of solutions $(\varphi, \mu, \sigma)$ with
\begin{align*}
\varphi & \in L^{\infty}(0,T;H^{2}) \cap L^{2}(0,T;H^{3}) \cap H^{1}(0,T;L^{2}) \cap C^{0}(\overline{Q}), \\
\mu & \in L^{2}(0,T;H^{2}) \cap L^{\infty}(0,T;L^{2}), \\
\sigma & \in L^{\infty}(0,T;H^{1}) \cap L^{2}(0,T;H^{2}) \cap H^{1}(0,T;L^{2}), \quad 0 \leq \sigma \leq 1 \text{ a.e. in } Q,
\end{align*}
such that $\varphi(0) = \varphi_{0}$, $\sigma(0) = \sigma_{0}$, and for a.e. $t \in (0,T)$ and for all $\zeta \in H^{1}$,
\begin{subequations}\label{weak:state}
\begin{align}
0 & = \int_{\Omega} \pd_{t}\varphi \zeta + \nabla \mu \cdot \nabla \zeta - (\PP \sigma - \AA  - \alpha u) h(\varphi) \zeta \dx, \label{weak:varphi} \\
0 & = \int_{\Omega} \mu \zeta - A \Psi'(\varphi) \zeta - B \nabla \varphi \cdot \nabla \zeta \dx, \label{weak:mu} \\
0 & = \int_{\Omega} \pd_{t} \sigma \zeta + \nabla \sigma \cdot \nabla \zeta + (\CC h(\varphi) + \BB)\sigma \zeta - \BB \sigma_{S} \zeta \dx. \label{weak:sigma}
\end{align}
\end{subequations}
Furthermore, it holds that
\begin{align*}
& \norm{\varphi}_{L^{\infty}(0,T;H^{2}) \cap L^{2}(0,T;H^{3}) \cap H^{1}(0,T;L^{2})} + \norm{\mu}_{L^{2}(0,T;H^{2}) \cap L^{\infty}(0,T;L^{2})} \\
& \quad  + \norm{\sigma}_{L^{2}(0,T;H^{2}) \cap L^{\infty}(0,T;H^{1}) \cap H^{1}(0,T;L^{2})} \leq \overline{C},
\end{align*}  
for some positive constant $\overline{C}$ not depending on $(\varphi, \mu, \sigma, u)$.  Let $(\varphi_{i}, \mu_{i}, \sigma_{i})_{i = 1,2}$ denote two weak solutions to \eqref{eq:state} satisfying \eqref{weak:state} corresponding to $\{u_{i}\}_{i=1,2}$ with the same initial data $\varphi_{0}$ and $\sigma_{0}$.  Then, there exists a positive constant $C_{\mathrm{cts}}$ depending only on $\norm{\varphi_{i}}_{L^{\infty}(0,T;L^{\infty})}$, $A$, $B$, $\PP$, $\AA$, $\CC$, $\alpha$, $k_{5}$, $T$ and the Lipschitz constant $L_{h}$, such that for all $s \in (0,T]$,
\begin{equation}\label{solution:ctsdep:s}
\begin{aligned}
& \left ( \norm{\varphi_{1}(s) - \varphi_{2}(s)}_{L^{2}}^{2} + \norm{\sigma_{1}(s) - \sigma_{2}(s)}_{H^{1}}^{2} \right ) + \norm{\mu_{1} - \mu_{2}}_{L^{2}(0,s;L^{2})}^{2} \\
& \quad + \norm{\pd_{t}\sigma_{1} - \pd_{t}\sigma_{2}}_{L^{2}(0,s;L^{2})}^{2} + \norm{\varphi_{1} - \varphi_{2}}_{L^{2}(0,s;H^{2})}^{2} 
\leq C_{\mathrm{cts}} \norm{u_{1} - u_{2}}_{L^{2}(0,s;L^{2})}^{2}.
\end{aligned}
\end{equation}
\end{thm}

The existence of solutions to the state equations \eqref{eq:state} is proved via a fixed point argument, see also \cite{article:DFRSS} for a similar argument applied to a multispecies tumor model.  One may also use a Galerkin approximation, which has been applied to similar systems in \cite{article:ColliGilardiHilhorst15, article:FrigeriGrasselliRocca15, article:GarckeLamNeumann,article:GarckeLamDirichlet,article:GarckeLamCHND, preprint:JiangWuZheng14, article:LowengrubTitiZhao13}.  The key difference here are that we establish boundedness of the nutrient concentration $\sigma$, which comes from the application of a weak comparison principle.  Here we also point out that the gradient $\nabla \varphi$ is continuous on the boundary up to initial time by the embedding $\varphi \in L^{\infty}(0,T;H^{2}) \cap H^{1}(0,T;L^{2}) \subset \subset C^{0}([0,T];H^{\beta})$ for $\beta < 2$ and the trace theorem.  Hence, the initial condition $\varphi_{0}$ needs to fulfill the boundary conditions. 

The unique solvability of the state equations \eqref{eq:state} allows us to define a solution operator $\Soln$ given as
\begin{align*}
\Soln(u) := (\varphi, \mu, \sigma),
\end{align*}
where the triplet $(\varphi, \mu, \sigma)$ is the unique weak solution to \eqref{eq:state} with data $(\varphi_{0}, \sigma_{0}, u)$ over the time interval $[0,T]$.  We use the notation $\varphi = \Soln_{1}(u)$ for the first component of $\Soln(u)$.  Then, we deduce the existence of a minimizer to the \eqref{OCProblem}.

\begin{thm}[Existence of minimizer]\label{thm:minimizer}
Under Assumption \ref{assump:main}, there exists at least one minimizer $(\varphi_{*}, u_{*}, \Optime)$ to \eqref{OCProblem}.  That is, $\varphi_{*} = \Soln_{1}(u_{*})$ with
\begin{align*}
J_{r}(\varphi_{*}, u_{*}, \Optime) & = \inf_{\substack{(w, s) \; \in \; \mathcal{U}_{\mathrm{ad}} \times [0,T]  \\
 \text{ s.t. } \phi \; = \; \Soln_{1}(w)}} J_{r}(\phi, w, s).
\end{align*}
\end{thm}
Note that we cannot exclude the trivial cases where $\Optime = 0$ or $\Optime = T$.  To establish the Fr\'{e}chet differentiability of the solution operator with respect to the control $u$, we first investigate the linearized state equations.  For arbitrary but fixed $\overline{u} \in \mathcal{U}_{\mathrm{ad}}$, let $(\overline{\varphi}, \overline{\mu}, \overline{\sigma}) = \Soln(\overline{u})$ denote the unique solution triplet to \eqref{eq:state} from Theorem \ref{thm:state}.  For $w \in L^{2}(Q)$, we consider the following linearized state equations,
\begin{subequations}\label{eq:linearstate}
\begin{alignat}{3}
\pd_{t} \Phi & = \Laplace \Xi + h(\overline{\varphi})(\PP \Sigma - \alpha w) + h'(\overline{\varphi}) \Phi (\PP \overline{\sigma} - \AA - \alpha \overline{u}) && \text{ in } Q, \\
\Xi & = A \Psi''(\overline{\varphi}) \Phi - B \Laplace \Phi && \text{ in } Q, \\
\pd_{t} \Sigma & = \Laplace \Sigma - \BB \Sigma - \CC (h(\overline{\varphi}) \Sigma + h'(\overline{\varphi}) \Phi \overline{\sigma}) && \text{ in } Q, \\
0 & = \pdnu \Phi = \pdnu \Xi = \pdnu \Sigma  && \text{ on } \Gamma \times (0,T), \\
0 & = \Phi(0) = \Sigma(0) && \text{ in } \Omega.
\end{alignat}
\end{subequations}
The unique solvability of \eqref{eq:linearstate} is obtained via a Galerkin procedure.
\begin{thm}[Unique solvability of the linearized state equations]
For any $w \in L^{2}(Q)$, there exists a unique triplet $(\Phi, \Xi, \Sigma)$ with
\begin{align*}
\Phi & \in L^{\infty}(0,T;H^{1}) \cap L^{2}(0,T;H^{3}) \cap H^{1}(0,T;(H^{1})^{*}), \\
\Xi & \in L^{2}(0,T;H^{1}), \\
\Sigma & \in L^{\infty}(0,T;H^{1}) \cap H^{1}(0,T;L^{2}) \cap L^{2}(0,T;H^{2}),
\end{align*}
such that for a.e $t \in (0,T)$ and for all $\zeta \in H^{1}$, 
\begin{subequations}\label{weak:Linearized}
\begin{align}
\label{weak:Linearized:varphi} 0 & = \inner{ \pd_{t} \Phi}{\zeta}_{H^{1}} + \int_{\Omega} \nabla \Xi \cdot \nabla \zeta - \left ( h(\overline{\varphi})(\PP \Sigma - \alpha w) + h'(\overline{\varphi}) (\PP \overline{\sigma} - \AA - \alpha \overline{u}) \Phi \right ) \zeta \dx, \\
\label{weak:Linearized:mu}  0 & = \int_{\Omega} \Xi \zeta- A \Psi''(\overline{\varphi}) \Phi \zeta - B \nabla \Phi \cdot \nabla \zeta \dx , \\
\label{weak:Linearized:sigma} 0 & = \int_{\Omega} \pd_{t} \Sigma \zeta+ \nabla \Sigma \cdot \nabla \zeta + \BB \Sigma \zeta +  \CC (h(\overline{\varphi}) \Sigma + h'(\overline{\varphi}) \Phi \overline{\sigma}) \zeta \dx.
\end{align}
\end{subequations}
Furthermore, there exists a constant $C > 0$ not depending $(\Phi, \Xi, \Sigma, w)$ such that
\begin{align*}
& \norm{\Phi}_{L^{\infty}(0,T;H^{1}) \cap L^{2}(0,T;H^{3}) \cap H^{1}(0,T;(H^{1})^{*})} + \norm{\Xi}_{L^{2}(0,T;H^{1})} \\
&  \quad + \norm{\Sigma}_{L^{\infty}(0,T;H^{1}) \cap H^{1}(0,T;L^{2}) \cap L^{2}(0,T;H^{2})} \leq C \norm{w}_{L^{2}(0,T;L^{2})}.
\end{align*} 
\end{thm}
The expectation is as follows.  Let $\overline{u}, \hat{u} \in \mathcal{U}_{\mathrm{ad}} \subset \mathcal{U}$ be arbitrary, with $(\overline{\varphi}, \overline{\mu}, \overline{\sigma}) = \Soln(\overline{u})$ and $(\hat{\varphi}, \hat{\mu}, \hat{\sigma}) = \Soln(\hat{u})$ denote the unique solution triplets to \eqref{eq:state} corresponding to $\overline{u}$ and $\hat{u}$, respectively.  Denote by $w := \hat{u} - \overline{u} \in L^{2}(Q)$ and let $(\Phi^{w}, \Xi^{w}, \Sigma^{w})$ denote the unique solution to the linearized state equations \eqref{eq:linearstate} associated to $w$.  We define the remainders to be
%
%
\begin{subequations}\label{Fdiff:remainders}
\begin{alignat}{3}
\theta^{w} & := \hat{\varphi} - \overline{\varphi} - \Phi^{w} && \in L^{2}(0,T;H^{3}) \cap L^{\infty}(0,T;H^{1}) \cap H^{1}(0,T;(H^{1})^{*}), \\ 
\rho^{w} & := \hat{\mu} - \overline{\mu} - \Xi^{w} && \in L^{2}(0,T;H^{1}), \\
\xi^{w} & := \hat{\sigma} - \overline{\sigma} - \Sigma^{w} && \in L^{\infty}(0,T;H^{1}) \cap L^{2}(0,T;H^{2}) \cap H^{1}(0,T;L^{2}).
\end{alignat}
\end{subequations}
If, for a suitable Banach space $\mathcal{Y}$ yet to be identified, we have
\begin{align*}
\frac{\norm{(\theta^{w}, \rho^{w}, \xi^{w})}_{\mathcal{Y}}}{\norm{w}_{L^{2}(Q)}} \to 0 \text{ as } \norm{w}_{L^{2}(Q)} \to 0,
\end{align*}
then, it holds that the solution operator $\Soln$ is Fr\'{e}chet differentiable at $\overline{u}$, the Fr\'{e}chet derivative with respect to the control $u$, denoted as $\der_{u} \Soln$, belongs to $\mathcal{L}(L^{2}(Q), \mathcal{Y})$ and satisfies
\begin{align*}
\der_{u} \Soln(\overline{u}) w = (\Phi^{w}, \Xi^{w}, \Sigma^{w}).
\end{align*}

With the unique solvability of the linearized state equations, we have the following result on the Fr\'{e}chet differentiability of the solution operator.
\begin{thm}[Fr\'{e}chet differentiability with respect to the control]\label{thm:Frechet:u}
Under Assumption \ref{assump:main}, the solution operator $\Soln$ is Fr\'{e}chet differentiable in $\mathcal{U}$ as a mapping from $L^{2}(Q)$ to the product Banach space
\begin{align*}
\mathcal{Y} := & \left [ L^{2}(0,T;H^{2}) \cap H^{1}(0,T;(H^{2})^{*}) \cap C^{0}([0,T];L^{2}) \right ] \\
& \quad \times L^{2}(Q) \times \left [ L^{\infty}(0,T;H^{1}) \cap H^{1}(0,T;L^{2}) \right ].
\end{align*}
That is, for any $\hat{u}, \overline{u} \in \mathcal{U}_{\mathrm{ad}} \subset \mathcal{U}$ with $w = \hat{u} - \overline{u} \in L^{2}(Q)$, there exists a positive constant $C_{\mathrm{diff},u}$ not depending on $\hat{u}, \overline{u}$ and $w$ such that
\begin{align}\label{Fdiff:u:thm:estimate}
\norm{(\theta^{w}, \rho^{w}, \xi^{w})}_{\mathcal{Y}}^{2} \leq C_{\mathrm{diff},u} \norm{w}_{L^{2}(Q)}^{4},
\end{align}
where $(\theta^{w}, \rho^{w}, \xi^{w})$ are defined as in \eqref{Fdiff:remainders}.  
\end{thm}

We now define a reduced functional 
\begin{align*}
\mathcal{J}(u, \tau) := J_{r}(S_{1}(u), u, \tau).
\end{align*}
For any $u \in \mathcal{U}_{\mathrm{ad}} \subset \mathcal{U}$, set $w = u - u_{*} \in L^{2}(Q)$ and let $(\Phi^{w}, \Xi^{w}, \Sigma^{w})$ be the unique solution to \eqref{eq:linearstate} corresponding to $w$, the optimal control $u_{*}$ and the corresponding state variables $(\varphi_{*}, \mu_{*}, \sigma_{*})$.  By Theorem \ref{thm:Frechet:u}, $\mathcal{J}$ is Fr\'{e}chet differentiable with respect to the control with 
\begin{equation}\label{Fdiff:J:wrt:u}
\begin{aligned}
\left (\der_{u} \mathcal{J}(u_{*},\Optime) \right ) w & = \beta_{Q} \int_{0}^{\Optime} \int_{\Omega} (\varphi_{*} - \varphi_{Q}) \Phi^{w} \dx \dt + \frac{\beta_{\Omega}}{r} \int_{\Optime-r}^{\Optime} \int_{\Omega} (\varphi_{*} - \varphi_{\Omega}) \Phi^{w}\dx \dt\\
& + \frac{\beta_{S}}{2r} \int_{\Optime-r}^{\Optime} \int_{\Omega} \Phi^{w} \dx \dt + \beta_{u} \int_{0}^{T} \int_{\Omega} u_{*} w \dx \dt,
\end{aligned}
\end{equation}
Next, we make use of the following adjoint system to eliminate the presence of the linearized state variable $\Phi^{w}$ in \eqref{Fdiff:J:wrt:u},
\begin{subequations}\label{eq:adjoint}
\begin{alignat}{3}
-\pd_{t}p + B \Laplace q & =  A\Psi''(\varphi_{*}) q - \CC h'(\varphi_{*}) \sigma_{*} r + h'(\varphi_{*}) (\PP \sigma_{*} - \AA - \alpha u_{*})p  &&  \\
\notag & + \beta_{Q}(\varphi_{*} - \varphi_{Q}) + \tfrac{1}{2r} \chi_{(\Optime - r, \Optime)}(t) \left ( 2\beta_{\Omega} (\varphi_{*} - \varphi_{\Omega}) + \beta_{S} \right ), && \text{ in } \Omega \times (0,\Optime), \\
q & = \Laplace p && \text{ in } \Omega \times (0,\Optime), \\
-\pd_{t}r & = \Laplace r - \BB r - \CC h(\varphi_{*})r + \PP h(\varphi_{*})p && \text{ in } \Omega \times (0,\Optime), \\
0 & = \pdnu p = \pdnu q = \pdnu r && \text{ on } \Gamma \times (0,\Optime), \\
r(\Optime) & = 0, \quad p(\Optime) = 0 && \text{ in } \Omega.
\end{alignat}
\end{subequations}
Note that the adjoint system is supplemented with terminal conditions at the optimal treatment time $\Optime$.  We now state the unique solvability result.
\begin{thm}[Unique solvability of the adjoint equations]\label{thm:adjoint}
Under Assumption \ref{assump:main}, for any $u \in \mathcal{U}_{\mathrm{ad}}$ there exists a unique triplet $(p,q,r)$ associated to $\Soln(u) = (\varphi,\mu, \sigma)$ with
\begin{align*}
p & \in L^{2}(0,\Optime;H^{2}) \cap H^{1}(0,\Optime;(H^{2})^{*}) \cap C^{0}([0,\Optime];L^{2}), \\
q & \in L^{2}(0,\Optime;L^{2}), \\
r & \in L^{2}(0,\Optime;H^{2}) \cap L^{\infty}(0,\Optime;H^{1}) \cap H^{1}(0,\Optime;L^{2}) \cap C^{0}([0,\Optime];L^{2}),
\end{align*}
satisfying
\begin{subequations}\label{adjoint:weak}
\begin{align}
\label{adjoint:weak:p} 0 & = \inner{-\pd_{t}p}{\zeta}_{H^{2}} + \int_{\Omega} B q \Laplace \zeta - A \Psi''(\varphi) q \zeta +  h'(\varphi) \left ( \CC \sigma r - (\PP \sigma - \AA - \alpha u)p \right ) \zeta \dx \\
\notag & - \int_{\Omega} \left ( \beta_{Q}(\varphi - \varphi_{Q}) + \tfrac{1}{2r} \chi_{(\Optime - r, \Optime)}(t) \left ( 2 \beta_{\Omega} (\varphi - \varphi_{\Omega}) + \beta_{S} \right ) \right ) \zeta \dx, \\
\label{adjoint:weak:q} 0 & = \int_{\Omega} q \eta + \nabla p \cdot \nabla \zeta \dx, \\
\label{adjoint:weak:r} 0 & = \int_{\Omega} - \pd_{t}r \eta + \nabla r \cdot \nabla \eta + \BB r \eta + \CC h(\varphi) r \eta - \PP h(\varphi) p \eta \dx
\end{align}
\end{subequations}
for a.e. $t \in (0,\Optime)$ and for all $\eta \in H^{1}$ and $\zeta \in H^{2}$.
\end{thm}

The first order necessary optimality conditions for the minimizer $(u_{*}, \Optime)$ of Theorem \ref{thm:minimizer} also requires the Fr\'{e}chet derivative of $\mathcal{J}$ with respect to $\tau$, and for this we make the additional assumption on the target functions $\varphi_{Q}$ and $\varphi_{\Omega}$.
\begin{assump}\label{assump:regularityU}
We now assume that $\varphi_{Q} \in H^{1}(0,T;L^{2})$ and $\varphi_{\Omega}  \in H^{1}(-r,T;L^{2})$.
\end{assump} 
Furthermore, we extend $\varphi$ to negative times using the initial condition, i.e., $\varphi(t) = \varphi_{0}$ for $t < 0$.

\begin{thm}[Fr\'{e}chet differentiability of the reduced functional with respect to time]
Under Assumptions \ref{assump:main} and \ref{assump:regularityU}, let $u \in \mathcal{U}_{\mathrm{ad}}$ be arbitrary with corresponding state variables $\Soln(u) = (\varphi, \mu, \sigma)$.  The reduced functional $\mathcal{J}(u, \tau)$ is Fr\'{e}chet differentiable with respect to $\tau$ with
\begin{align*}
 \der_{\tau} \mathcal{J}(u, \tau) & = \beta_{T} + \frac{\beta_{Q}}{2} \norm{\varphi(\tau) - \varphi_{Q}(\tau)}_{L^{2}}^{2} + \frac{\beta_{u}}{2} \norm{u(\tau)}_{L^{2}}^{2}  + \frac{\beta_{S}}{2r} \int_{\Omega} \varphi(\tau) - \varphi(\tau - r) \dx \\
& + \frac{\beta_{\Omega}}{2r} \left ( \norm{(\varphi - \varphi_{\Omega})(\tau)}_{L^{2}}^{2} - \norm{(\varphi - \varphi_{\Omega})(\tau - r)}_{L^{2}}^{2} \right ).
\end{align*}
\end{thm}

The first order necessary optimality conditions to \eqref{OCProblem} for the minimizer $(u_{*}, \Optime)$ are given as follows.
\begin{thm}[First order necessary optimality conditions]\label{thm:FONC}
Under Assumptions \ref{assump:main} and \ref{assump:regularityU}, let $(u_{*}, \Optime) \in \mathcal{U}_{\mathrm{ad}} \times [0,T]$ denote a minimizer to \eqref{OCProblem} with corresponding state variables $\Soln(u_{*}) = (\varphi_{*}, \mu_{*}, \sigma_{*})$ and associated adjoint variables $(p,q,r)$.  Then, it holds that
\begin{equation}\label{FONC:u}
\begin{aligned}
\int_{0}^{T} \int_{\Omega} \beta_{u} u_{*} (u - u_{*}) \dx \dt - \int_{0}^{\Optime} \int_{\Omega} h(\varphi_{*}) \alpha p (u - u_{*}) \dx \dt \geq 0 \quad \forall u \in \mathcal{U}_{\mathrm{ad}},
\end{aligned}
\end{equation}
and
\begin{equation}\label{FONC:tau}
\begin{aligned}
\beta_{T} & + \frac{\beta_{Q}}{2} \norm{(\varphi_{*} - \varphi_{Q})(\Optime)}_{L^{2}}^{2} + \frac{\beta_{S}}{2r} \int_{\Omega} \varphi_{*}(\Optime) - \varphi_{*}(\Optime - r) \dx \\
&  + \frac{\beta_{\Omega}}{2r} \left ( \norm{(\varphi_{*} - \varphi_{\Omega})(\Optime)}_{L^{2}}^{2} - \norm{(\varphi_{*} - \varphi_{\Omega})(\Optime - r)}_{L^{2}}^{2} \right ) \begin{cases}
\geq 0 & \text{ if } \Optime = 0, \\
= 0 & \text{ if } \Optime \in (0,T), \\
\leq 0 & \text{ if } \Optime = T.
\end{cases}
\end{aligned}
\end{equation}
\end{thm}

\begin{remark}
If we extend $p$ by zero to $(\Optime, T]$, then we can express \eqref{FONC:u} as
\begin{align*}
\int_{0}^{T} \int_{\Omega} (\beta_{u} u_{*} - h(\varphi_{*}) \alpha p)(u - u_{*}) \dx \dt \geq 0 \quad \forall u \in \mathcal{U}_{\mathrm{ad}},
\end{align*}
which allows for the interpretation that the optimal control $u_{*}$ is the $L^{2}(Q)$-projection of $\beta_{u}^{-1} h(\varphi_{*}) \alpha p$ onto $\mathcal{U}_{\mathrm{ad}}$.
\end{remark}

\section{Results on the state equations}\label{sec:stateequations}
We show the existence of weak solutions to the state equations \eqref{eq:state} by means of a fixed point argument.  The idea is to consider the following two auxiliary problems.  Let $\phi$ be given, we define the solution mapping $\mathcal{M}_{1}$ by $\sigma = \mathcal{M}_{1}(\phi)$, where $\sigma$ is the unique solution to
\begin{equation}\tag{$\mathrm{AP}$1}\label{auxiliary:nutrient}
\begin{alignedat}{3}
\pd_{t}\sigma& = \Laplace \sigma - \CC h(\phi) \sigma + \BB (\sigma_{S} - \sigma) && \text{ in } Q,
\end{alignedat}
\end{equation}
with homogeneous Neumann boundary condition and initial condition $\sigma_{0}$.  Then, we define the solution mapping $\mathcal{M}$ by $\varphi = \mathcal{M}(\phi)$, where $\varphi$ is the unique solution to
\begin{equation}\tag{$\mathrm{AP}$2}\label{auxiliary:CH}
\begin{alignedat}{3}
\pd_{t}\varphi & = \Laplace \mu + h(\varphi)( \PP \mathcal{M}_{1}(\phi) - \AA - \alpha u) && \text{ in } Q, \\
\mu & = A \Psi'(\varphi) - B \Laplace \varphi && \text{ in } Q, 
\end{alignedat}
\end{equation}
with homogeneous Neumann boundary conditions and initial condition $\varphi_{0}$.  If $\tilde{\varphi}$ is a fixed point for $\mathcal{M}$, with $\tilde{\sigma} = \mathcal{M}_{1}(\tilde{\varphi})$ and $\tilde{\mu} = A \Psi'(\tilde{\varphi}) - B \Laplace \tilde{\varphi}$, then the triplet $(\tilde{\varphi}, \tilde{\mu}, \tilde{\sigma})$ is a solution to \eqref{eq:state}.

\subsection{Auxiliary problems}
\begin{lemma}\label{lem:auxProb1:nutrient}
Let $\phi \in L^{2}(Q)$ be given.  Under Assumption \ref{assump:main}, there exists a unique solution
\begin{align*}
\sigma \in L^{2}(0,T;H^{2}) \cap L^{\infty}(0,T;H^{1}) \cap H^{1}(0,T;L^{2})
\end{align*}
to \eqref{auxiliary:nutrient} such that $\sigma(0) = \sigma_{0}$ and $0 \leq \sigma \leq 1$ a.e. in $Q$.  Furthermore there exists a positive constant $C_{\mathrm{AP}1}$ not depending on $\phi$ such that
\begin{align}\label{nutrientestimates}
\norm{\sigma}_{L^{2}(0,T;H^{2}) \cap L^{\infty}(0,T;H^{1}) \cap H^{1}(0,T;L^{2})} \leq C_{\mathrm{AP}1}.
\end{align}
\end{lemma}

\begin{proof}
As \eqref{auxiliary:nutrient} is a linear parabolic equation in $\sigma$, the existence of weak solutions can be shown using a Galerkin approximation, and we will only present the derivation of a priori estimates here.  The weak formulation of \eqref{auxiliary:nutrient} is
\begin{align}\label{nutrient:comparison}
\int_{\Omega} \pd_{t} \sigma\zeta +\nabla \sigma \cdot \nabla \zeta + \CC h(\phi) \sigma \zeta + \BB \sigma \zeta - \BB \sigma_{S} \zeta \dx = 0,
\end{align}
for a.e. $t \in (0,T)$ and for all $\zeta \in H^{1}$.  

\paragraph{First estimate.} Substituting $\zeta = \sigma$ in \eqref{nutrient:comparison} yields
\begin{align*}
\frac{1}{2} \frac{\dd}{\dt} \norm{\sigma}_{L^{2}}^{2} + \norm{\nabla \sigma}_{L^{2}}^{2} + \int_{\Omega} \CC h(\phi) \abs{\sigma}^{2} + \BB \abs{\sigma}^{2} \dx \leq \frac{\BB^{2}}{2} \norm{\sigma_{S}}_{L^{2}}^{2} + \frac{1}{2} \norm{\sigma}_{L^{2}}^{2}.
\end{align*}
Neglecting the nonnegative term $\CC h(\phi)\abs{\sigma}^{2} + \BB \abs{\sigma}^{2}$, and the application of the Gronwall inequality leads to
\begin{align}\label{apriori:nutrient:LinftyL2}
\norm{\sigma}_{L^{\infty}(0,T;L^{2})}^{2} + \norm{\nabla \sigma}_{L^{2}(0,T;L^{2})}^{2} \leq C \left (T, \norm{\sigma_{S}}_{L^{2}(Q)}, \norm{\sigma_{0}}_{L^{2}} \right ).
\end{align}

\paragraph{Second estimate.} Substituting $\zeta = \pd_{t} \sigma$ in \eqref{nutrient:comparison} yields
\begin{align*}
\norm{\pd_{t} \sigma}_{L^{2}}^{2} + \frac{1}{2} \frac{\dd}{\dt} \norm{\nabla \sigma}_{L^{2}}^{2} \leq \frac{1}{2} \left ((\CC + \BB) \norm{\sigma}_{L^{2}} + \BB \norm{\sigma_{S}}_{L^{2}} \right )^{2} + \frac{1}{2} \norm{\pd_{t}\sigma}_{L^{2}}^{2},
\end{align*}
where we used the boundedness of $h$.  Integrating in time and using that $\sigma_{0} \in H^{1}$, $\sigma_{S} \in L^{2}(Q)$ we have
\begin{align}\label{apriori:nutrient:H1L2}
\norm{\pd_{t} \sigma}_{L^{2}(Q)}^{2} + \norm{\nabla \sigma}_{L^{\infty}(0,T;L^{2})}^{2} \leq C \left (T, \CC, \BB, \norm{\sigma_{S}}_{L^{2}(Q)}, \norm{\sigma_{0}}_{H^{1}} \right ).
\end{align}

\paragraph{Third estimate.} Note that \eqref{nutrient:comparison} can be seen as a weak formulation for the following elliptic problem,
\begin{alignat}{3}\label{nutrient:elliptic}
-\Laplace \sigma + \sigma & = -\pd_{t}\sigma - \CC h(\phi) \sigma + \BB(\sigma_{S}  - \sigma) + \sigma && \text{ in } \Omega, \\
\notag \pdnu \sigma &= 0 && \text{ on } \Gamma.
\end{alignat}
As the right-hand side of \eqref{nutrient:elliptic} belongs to $L^{2}$ for a.e. $t \in (0,T)$, elliptic regularity theory \cite[Theorem 2.4.2.7]{book:Grisvard} yields that $\sigma(t) \in H^{2}$ for a.e. $t \in (0,T)$, with the estimate
\begin{align*}
\norm{\sigma}_{H^{2}}^{2} \leq C \left ( \norm{\pd_{t} \sigma}_{L^{2}}^{2} + \norm{\sigma}_{L^{2}}^{2} + \norm{\sigma_{S}}_{L^{2}}^{2} \right ),
\end{align*}
where $C$ is a positive constant not depending on $\sigma$ and $\phi$.  Integrating in time gives
\begin{align}\label{apriori:nutrient:L2H2}
\norm{\sigma}_{L^{2}(0,T;H^{2})}^{2} \leq C \left ( \norm{\pd_{t} \sigma}_{L^{2}(Q)}^{2} + \norm{\sigma}_{L^{2}(Q)}^{2} + \norm{\sigma_{S}}_{L^{2}(Q)}^{2} \right ).
\end{align}
The a priori estimates \eqref{apriori:nutrient:LinftyL2}, \eqref{apriori:nutrient:H1L2} and \eqref{apriori:nutrient:L2H2} are sufficient to deduce the existence of a weak solution $\sigma$ satisfying \eqref{nutrient:comparison}.  The initial condition is attained by the use of the continuous embedding $L^{2}(0,T;H^{1}) \cap H^{1}(0,T;L^{2}) \subset C^{0}([0,T];L^{2})$, and using weak lower semicontinuity of the norms, we obtain \eqref{nutrientestimates}.  We now establish the boundedness property and continuous dependence on the data $\phi$.

\paragraph{Boundedness.} Substituting $\zeta = -\sigma^{-} := -\max (-\sigma, 0)$ in \eqref{nutrient:comparison} leads to
\begin{align}\label{bdd:lower}
\frac{1}{2} \frac{\dd}{\dt} \norm{\sigma^{-}}_{L^{2}}^{2} + \int_{\Omega} \abs{\nabla \sigma^{-}}^{2} + \CC h(\phi) \abs{\sigma^{-}}^{2} + \BB \abs{\sigma^{-}}^{2} + \BB \sigma_{S} \sigma^{-} \dx = 0.
\end{align}
As the integrand is nonnegative, we neglect the second term on the left-hand side and upon integrating yields
\begin{align*}
\norm{\sigma^{-}(t)}_{L^{2}}^{2} \leq \norm{\sigma^{-}(0)}_{L^{2}}^{2} = 0 \quad \forall t \in (0,T],
\end{align*}
where we used that $\sigma_{0} \geq 0$ a.e. in $\Omega$, and so $\sigma^{-}(0) = 0$ a.e. in $\Omega$.  Thus $\sigma \geq 0$ a.e. in $Q$.  On the other hand, consider $\zeta = (\sigma - 1)^{+} = \max (\sigma - 1, 0)$ in \eqref{nutrient:comparison}, which yields
\begin{equation}\label{bdd:upper}
\begin{aligned}
& \frac{1}{2} \frac{\dd}{\dt} \norm{(\sigma - 1)^{+}}_{L^{2}}^{2} + \int_{\Omega}  \abs{\nabla (\sigma - 1)^{+}}^{2} + \left (\CC h(\phi) + \BB \right ) \abs{(\sigma - 1)^{+}}^{2} \dx \\
& \quad + \int_{\Omega} \CC h(\phi) (\sigma - 1)^{+} +  \BB (1 - \sigma_{S})(\sigma - 1)^{+} \dx = 0.
\end{aligned}
\end{equation}
Using that $h$ is nonnegative and $\sigma_{S} \leq 1$ a.e. in $Q$, so that $(1-\sigma_{S})(\sigma - 1)^{+}$ is nonnegative, we find that the integrand is nonnegative.  Thus, after integrating from $0$ to $t$, we obtain
\begin{align*}
\norm{(\sigma - 1)^{+}(t)}_{L^{2}}^{2} \leq \norm{(\sigma - 1)^{+}(0)}_{L^{2}}^{2} = 0 \quad \forall t \in (0,T],
\end{align*}
where we used that $\sigma_{0} \leq 1$ a.e. in $\Omega$.  This implies that $\sigma \leq 1$ a.e. in $Q$.

\paragraph{Continuous dependence.} Let $\{\sigma_{i}\}_{i = 1,2}$ denote two functions satisfying \eqref{nutrient:comparison} corresponding to $\{\phi_{i}\}_{i = 1,2} \subset L^{2}(Q)$, respectively, and with the same initial condition $\sigma_{0}$ and nutrient supply $\sigma_{S}$.  Then the difference $\sigma := \sigma_{1} - \sigma_{2}$ satisfies
\begin{align}\label{sigma:difference}
\int_{\Omega} \pd_{t} \sigma \zeta + \nabla \sigma \cdot \nabla \zeta + \left ( \CC (h(\phi_{1}) - h(\phi_{2})) \sigma_{1} + \CC h(\phi_{2}) \sigma \right ) \zeta + \BB \sigma \zeta \dx = 0
\end{align}
for a.e. $t \in (0,T)$ and for all $\zeta \in H^{1}$.  Substituting $\zeta = \sigma$ in \eqref{sigma:difference}, neglecting the nonnegative term $\CC h(\phi_{2}) \abs{\sigma}^{2} + \BB \abs{\sigma}^{2}$, and integrate over $[0,s]$ for $s \in (0,T]$, we obtain
\begin{align*}
\frac{1}{2} \norm{\sigma(s)}_{L^{2}}^{2} + \norm{\nabla \sigma}_{L^{2}(0,s;L^{2})}^{2} \leq \frac{(\CC L_{h})^{2}}{2} \norm{\phi_{1} - \phi_{2}}_{L^{2}(0,s;L^{2})}^{2} + \frac{1}{2} \norm{\sigma}_{L^{2}(0,s;L^{2})}^{2},
\end{align*}
where we have used the boundedness of $\sigma_{1}$ and the Lipschitz property of $h$.  Applying Gronwall's inequality \eqref{Gronwall} yields
\begin{align}\label{auxProb1:ctsdep:est1}
\norm{\sigma(s)}_{L^{2}}^{2} + 2 \norm{\nabla \sigma}_{L^{2}(0,s;L^{2})}^{2} \leq (\CC L_{h})^{2} \norm{\phi_{1} - \phi_{2}}_{L^{2}(0,s;L^{2})}^{2} e^{s} \text{ for } s \in (0,T],
\end{align}
where we used that 
\begin{align*}
\int_{0}^{s} \norm{\phi_{1} - \phi_{2}}_{L^{2}(0,t;L^{2})}^{2} e^{t} \dt \leq \norm{\phi_{1} - \phi_{2}}_{L^{2}(0,s;L^{2})}^{2} (e^{s} - 1).
\end{align*}
Next, substituting $\zeta = \pd_{t} \sigma$ in \eqref{sigma:difference} and integrate over $[0,s]$ leads to
\begin{align*}
& \norm{\pd_{t}\sigma}_{L^{2}(0,s;L^{2})}^{2} + \frac{1}{2} \norm{\nabla \sigma(s)}_{L^{2}}^{2} + \frac{\BB}{2} \norm{\sigma(s)}_{L^{2}}^{2} \\
& \quad \leq (\CC L_{h})^{2} \norm{\phi_{1} - \phi_{2}}_{L^{2}(0,s;L^{2})}^{2} + \CC^{2} \norm{\sigma}_{L^{2}(0,s;L^{2})}^{2} + \frac{2}{4} \norm{\pd_{t} \sigma}_{L^{2}(0,s;L^{2})}^{2}.
\end{align*}
From \eqref{auxProb1:ctsdep:est1} we have
\begin{equation}\label{auxProb1:ctsdep:sigmaL2L2}
\begin{aligned}
\norm{\sigma}_{L^{2}(0,s;L^{2})}^{2} = \int_{0}^{s} \norm{\sigma(t)}_{L^{2}}^{2} \dt &  \leq (\CC L_{h})^{2} \int_{0}^{s} \norm{\phi_{1} - \phi_{2}}_{L^{2}(0,t;L^{2})}^{2} e^{t} \dt \\
& \leq (\CC L_{h})^{2} (e^{s} - 1) \norm{\phi_{1} - \phi_{2}}_{L^{2}(0,s;L^{2})}^{2},
\end{aligned}
\end{equation}
and so this yields
\begin{align}\label{auxProb1:ctsdep:est2}
\norm{\pd_{t}\sigma}_{L^{2}(0,s;L^{2})}^{2} + \norm{\nabla \sigma(s)}_{L^{2}}^{2}  \leq (\CC L_{h})^{2} (1 + \CC^{2}(e^{s} - 1)) \norm{\phi_{1} - \phi_{2}}_{L^{2}(0,s;L^{2})}^{2}.
\end{align}

\end{proof}

\begin{remark}
The main reason we do not employ a Galerkin approximation for the state equation \eqref{eq:state} is that the computations in the weak comparison principle seems not to apply to the Galerkin solutions, in particular we cannot show that the Galerkin solutions to $\sigma$ is nonnegative and bounded above by $1$.  Indeed, for Galerkin solutions $\varphi_{n}$ and $\sigma_{n}$ belong to some finite dimensional subspace $W_{n}$ of $H^{1}$ satisfying
\begin{align*}
0 = \int_{\Omega} \pd_{t} \sigma_{n} v  + \nabla \sigma_{n} \cdot \nabla v + \CC h(\varphi_{n}) \sigma_{n} v + \BB \sigma_{n} - \BB \sigma_{S} v \dx
\end{align*}
for all $v \in W_{n}$, if we test with $v = \Pi_{n}(\sigma_{n}^{-}) \in W_{n}$ where $\Pi_{n}$ denotes the orthogonal projection to $W_{n}$, we have
\begin{align*}
\int_{\Omega} \pd_{t} \sigma_{n} \Pi_{n} (\sigma_{n}^{-}) \dx = \int_{\Omega} \pd_{t} \sigma_{n} \sigma_{n}^{-} \dx = - \frac{\dd}{\dt} \norm{\sigma_{n}^{-}}_{L^{2}}^{2},
\end{align*}
but we cannot deduce if
\begin{align*}
\int_{\Omega} h(\varphi_{n}) \sigma_{n} \Pi_{n}(\sigma_{n}^{-}) \dx = \int_{\Omega} \Pi_{n} (h(\varphi_{n}) \sigma_{n}) \sigma_{n}^{-} \dx
\end{align*}
is nonnegative.  There is also a similar issue with the nonnegativity of
\begin{align*}
\int_{\Omega} \CC h(\varphi_{n}) (\sigma_{n} - 1)\Pi_{n}((\sigma_{n}-1)^{+}) + \CC h(\varphi_{n}) \Pi_{n}(\sigma_{n}-1)^{+}) + \BB (1-\sigma_{S}) \Pi_{n}((\sigma_{n}-1)^{+}) \dx,
\end{align*}
as it is not guaranteed that the projection of a nonnegative function is nonnegative.  
\end{remark}
Due to the estimate \eqref{auxProb1:ctsdep:sigmaL2L2}, we can define a continuous mapping
\begin{alignat*}{4}
\mathcal{M}_{1} & : L^{2}(Q) && \to L^{\infty}(Q) \cap L^{2}(0,T;H^{2}) \cap L^{\infty}(0,T;H^{1}) \cap H^{1}(0,T;L^{2}) \\
& \quad \quad \phi && \mapsto \quad \quad  \sigma \text{ given by Lemma } \ref{lem:auxProb1:nutrient},
\end{alignat*}

\begin{lemma}\label{lem:auxCahnHilliard}
Let $\phi \in L^{2}(Q)$ be given.  Under Assumption \ref{assump:main}, there exists a unique solution pair 
\begin{align*}
\varphi & \in L^{\infty}(0,T;H^{2}) \cap L^{2}(0,T;H^{3}) \cap H^{1}(0,T;L^{2}) \cap C^{0}(\overline{Q}), \\
\mu & \in L^{2}(0,T;H^{2}) \cap L^{\infty}(0,T;L^{2}), 
\end{align*}
to \eqref{auxiliary:CH} such that $\varphi(0) = \varphi_{0}$ and satisfy
\begin{subequations}\label{auxProb2:weakform}
\begin{align}
0 & = \int_{\Omega} \pd_{t}\varphi \zeta + \nabla \mu \cdot \nabla \zeta - (\PP \mathcal{M}_{1}(\phi) - \AA - \alpha u)h(\varphi) \zeta \dx, \\
0 & = \int_{\Omega} \mu \zeta - A \Psi'(\varphi) \zeta - B \nabla \varphi \cdot \nabla \zeta \dx, 
\end{align}
\end{subequations}
for a.e. $t \in (0,T)$ and for all $\zeta \in H^{1}$.  Furthermore, there exists a positive constant $C_{\mathrm{AP2}}$ depending only on $T$, $\Omega$, $k_{0}$, $k_{1}$, $k_{2}$, $k_{3}$, $k_{4}$, $A$, $B$, $\PP$, $\AA$, $\alpha$, $\norm{\varphi_{0}}_{H^{3}}$ and $C_{\mathrm{AP}1}$, such that
\begin{equation}\label{auxProb2:Est}
\begin{aligned}
& \norm{\varphi}_{L^{\infty}(0,T;H^{2}) \cap L^{2}(0,T;H^{3}) \cap H^{1}(0,T;L^{2})} + \norm{\mu}_{L^{2}(0,T;H^{2}) \cap L^{\infty}(0,T;L^{2})} \leq C_{\mathrm{AP2}}.
\end{aligned}
\end{equation}
That is, $C_{\mathrm{AP2}}$ does not depend on $\phi$.
\end{lemma}

\begin{proof}
Let $\{w_{i}\}_{i \in \N}$ denote the eigenfunctions of the Neumann-Laplacian with corresponding eigenvalues $\{\lambda_{i}\}_{i \in \N}$:
\begin{align*}
-\Laplace w_{i} = \lambda_{i} w_{i} \text{ in } \Omega, \quad \pdnu w_{i} = 0 \text{ on } \Gamma.
\end{align*}
Then, it is well-known that $\{w_{i}\}_{i \in \N}$ forms an orthonormal basis of $L^{2}$ and an orthogonal basis of $H^{1}$.  As constant functions are eigenfunctions, we take $w_{1} = 1$ with $\lambda_{1} = 0$.  Let $n \in N$ be fixed and we define $W_{n} := \mathrm{span}\{w_{1}, \dots, w_{n}\}$ as the finite dimensional space spanned by the first $n$ eigenfunctions, with the corresponding projection operator $\Pi_{n}$.  We consider sequences $\{\phi_{n}\}_{n \in \N}, \{u_{n}\}_{n \in \N} \subset C^{0}(0,T;L^{2})$ such that $\phi_{n} \to \phi$ and $u_{n} \to u$ strongly in $L^{2}(0,T;L^{2})$ and look for functions of the form
\begin{align*}
\varphi_{n}(x,t) = \sum_{i=1}^{n} a_{n,i}(t) w_{i}(x), \quad \mu_{n}(x,t) := \sum_{i=1}^{n} b_{n,i}(t) w_{i}(x),
\end{align*}
where the coefficients $\bm{a}_{n} := \{a_{n,i}\}_{i=1}^{n}$ and $\bm{b}_{n} := \{b_{n,i}\}_{i=1}^{n}$ satisfy the following initial-value problem
\begin{subequations}\label{auxProb2:ODE}
\begin{align}
\bm{a}_{n}' & = - \bm{S} \bm{b}_{n} + \PP \bm{M}_{n} - \AA \bm{H}_{n} - \alpha \bm{U}_{n}, \label{auxProb2:ODE:varphi} \\
\bm{b}_{n} & =  A \bm{\psi}_{n} + B \bm{S} \bm{a}_{n}, \label{auxProb2:ODE:mu}, \\
\bm{a}_{n}(0) & = (\Pi_{n}\varphi_{0})_{i=1}^{n} = \left ( \int_{\Omega} \varphi_{0} w_{i}  \dx \right )_{i=1}^{n}
\end{align}
\end{subequations}
with prime denoting the time derivative and for $1 \leq i,j \leq n$, 
\begin{equation}\label{ODE:vectormatrix}
\begin{alignedat}{4}
\bm{S}_{ij} & := \int_{\Omega} \nabla w_{i} \cdot \nabla w_{j} \dx, \quad && (\bm{H}_{n})_{j} && := \int_{\Omega} h(\varphi_{n}) w_{j} \dx,\\
(\bm{U}_{n})_{j} & :=  \int_{\Omega} h(\varphi_{n}) u_{n} w_{j} \dx, \quad && (\bm{\psi}_{n})_{j} && := \int_{\Omega} \Psi'(\varphi_{n}) w_{j} \dx, \\
(\bm{M}_{n})_{ij} & :=  \int_{\Omega} h(\varphi_{n}) \mathcal{M}_{1}(\phi_{n}) w_{j} \dx. \quad && &&
\end{alignedat}
\end{equation}
Without loss of generality, we assume that $0 \leq u_{n} \leq 1$ a.e. in $Q$ for all $n \in \N$ and from Lemma \ref{lem:auxProb1:nutrient} it holds that $0 \leq \mathcal{M}_{1}(\phi_{n}) \leq 1$ a.e. in $Q$ and $\mathcal{M}_{1}(\phi_{n}) \in C^{0}([0,T];L^{2})$ for all $n \in \N$.  Substituting \eqref{auxProb2:ODE:mu} into \eqref{auxProb2:ODE:varphi} leads to a system of ODEs in $\bm{a}_{n}$ with right-hand side depending continuously on $t$ and $\bm{a}_{n}$.  By the Cauchy--Peano theorem \cite[Chapter 1, Theorem 1.2]{book:Coddington}, there exists a $t_{n} \in (0,T]$ such that \eqref{auxProb2:ODE} has a local solution $\bm{a}_{n}$ on $[0,t_{n})$ with $\bm{a}_{n} \in C^{1}([0,t_{n});\R^{n})$.  Then, $\bm{b}_{n}$ can be defined by the relation \eqref{auxProb2:ODE:mu}, and we obtain functions $\varphi_{n}, \mu_{n} \in C^{1}([0,t_{n});W_{n})$ satisfying
\begin{subequations}\label{auxProb2:Galerkin}
\begin{align}
\pd_{t}\varphi_{n} & = \Laplace \mu_{n} + \Pi_{n}(h(\varphi_{n})(\PP \mathcal{M}_{1}(\phi_{n}) - \AA - \alpha u_{n})),\label{auxProb2:strong:varphin} \\
\mu_{n} & = A \Pi_{n}(\Psi'(\varphi_{n})) - B \Laplace \varphi_{n}, \label{auxProb2:strong:mun} \\
\varphi_{n}(0) & = \Pi_{n}(\varphi_{0}).
\end{align}
\end{subequations}
In the following we will derive a series of a priori estimates leading to the uniform boundedness (in $n$) of $(\varphi_{n}, \mu_{n})$ in the following Bochner spaces:
\begin{enumerate}
\item $\Psi(\varphi_{n}) \in L^{\infty}(0,T;L^{1})$, $\varphi_{n} \in L^{\infty}(0,T;H^{1})$, $\mu_{n} \in L^{2}(0,T;H^{1})$,
\item $\varphi_{n} \in L^{2}(0,T;H^{3})$,
\item $\mu_{n} \in L^{\infty}(0,T;L^{2}) \cap L^{2}(0,T;H^{2})$, $\varphi_{n} \in L^{\infty}(0,T;H^{2})$, $\pd_{t}\varphi_{n} \in L^{2}(0,T;L^{2})$.
\end{enumerate}
In particular for the third estimate, we have to differentiate \eqref{auxProb2:Galerkin} in time to obtain a system of ODEs involving $\pd_{t}\mu_{n}$.  Thus, we prescribe additional initial conditions, namely we set
\begin{align*}
\mu_{0} := A \Psi'(\varphi_{0}) - B \Laplace \varphi_{0}, \quad \mu_{n}(0) := \Pi_{n} (\mu_{0}).
\end{align*}
Note that by Assumption \ref{assump:main}, there exists a positive constant $C_{\mathrm{ini}}$, not depending on $\phi$ and $n$, such that
\begin{align*}
\norm{\mu_{n}(0)}_{L^{2}} \leq \norm{\mu_{0}}_{L^{2}} \leq C_{\mathrm{ini}} \norm{\varphi_{0}}_{H^{3}}.
\end{align*}
Furthermore, to approximate $\varphi_{0}$ by a linear combination of eigenfunctions of the Neumann-Laplacian in $H^{2}$, we require that $\varphi_{0}$ satisfies zero Neumann boundary conditions.
\paragraph{First estimate.}  Multiplying \eqref{auxProb2:strong:varphin} with $\mu_{n}$ and \eqref{auxProb2:strong:mun} with $\pd_{t}\varphi_{n}$, integrate over $\Omega$ and integrate by parts, upon adding and using the boundedness of $h$, $\mathcal{M}_{1}(\phi_{n})$ and $u_{n}$, we obtain
\begin{equation}\label{auxProb2:first:estimate}
\begin{aligned}
& \frac{\dd}{\dt} \left ( A \norm{\Psi(\varphi_{n})}_{L^{1}} + \frac{B}{2} \norm{\nabla \varphi_{n}}_{L^{2}}^{2} \right ) \dx + \norm{\nabla \mu_{n}}_{L^{2}}^{2} \leq ( \PP + \AA + \alpha) \norm{\mu_{n}}_{L^{1}} .
\end{aligned}
\end{equation}
Let $C_{u} := \PP + \AA + \alpha$, then by the Poincar\'{e} inequality in $L^{1}$ (with constant $C_{p} > 0$ depending only on $\Omega$), H\"{o}lder's inequality and Young's inequality, the right-hand side of \eqref{auxProb2:first:estimate} can be estimated as follows,
\begin{equation}\label{sourceterm:est}
\begin{aligned}
& C_{u} \norm{\mu_{n}}_{L^{1}} \leq  C_{u} \bignorm{\mu_{n} - \frac{1}{\abs{\Omega}} \int_{\Omega} \mu_{n} \dx}_{L^{1}} + C_{u} \abs{\int_{\Omega} \mu_{n} \dx} \\
& \quad \leq C_{u} C_{p} \abs{\Omega}^{\frac{1}{2}} \norm{\nabla \mu_{n}}_{L^{2}} + C_{u} \abs{\int_{\Omega} \mu_{n} \dx} \leq \frac{1}{2} \norm{\nabla \mu_{n}}_{L^{2}}^{2} + \frac{C_{u}^{2} C_{p}^{2} \abs{\Omega}}{2} + C_{u} \abs{\int_{\Omega} \mu_{n} \dx}.
\end{aligned}
\end{equation}
From integrating \eqref{auxProb2:strong:mun} over $\Omega$, and \eqref{Psi'Psi}, we find that
\begin{align}\label{mean:mu}
\abs{\int_{\Omega} \mu_{n} \dx} \leq A \norm{\Psi'(\varphi_{n})}_{L^{1}} \leq Ak_{0} \int_{\Omega} \Psi(\varphi_{n}) \dx + A k_{1} \abs{\Omega}.
\end{align}
Hence, we obtain the following differential inequality
\begin{align*}
& \frac{\dd}{\dt} \left ( A \norm{\Psi(\varphi_{n})}_{L^{1}} + \frac{B}{2} \norm{\nabla \varphi_{n}}_{L^{2}}^{2} \right )  - A k_{0} C_{u} \norm{\Psi(\varphi_{n})}_{L^{1}} + \frac{1}{2} \norm{\nabla \mu_{n}}_{L^{2}}^{2} \\
& \quad \leq A k_{1} \abs{\Omega} C_{u} + \frac{C_{u}^{2} C_{p}^{2} \abs{\Omega}}{2} =: d_{0} .
\end{align*}
By the Sobolev embedding $H^{1} \subset L^{6}$ and the growth assumption \eqref{Psi:growth}, it holds that $\norm{\Psi(\varphi_{0})}_{L^{1}} \leq C(1 + \norm{\varphi_{0}}_{L^{4}}^{4}) \leq C(1 + \norm{\varphi_{0}}_{H^{1}}^{4})$.  Thus $c_{0} := A \norm{\Psi(\varphi_{0})}_{L^{1}} + \frac{B}{2} \norm{\nabla \varphi_{0}}_{L^{2}}^{2}$ is bounded.  Integrating over $[0,s]$ for $s \in (0,T]$ yields
\begin{align*}
& \left ( A \norm{\Psi(\varphi_{n}(s))}_{L^{1}} + \frac{B}{2} \norm{\nabla \varphi_{n}(s)}_{L^{2}}^{2} \right ) + \frac{1}{2} \norm{\nabla \mu_{n}}_{L^{2}(0,s;L^{2})}^{2} \\
& \quad \leq  k_{0} C_{u} \int_{0}^{s} \left ( A \norm{\Psi(\varphi_{n})}_{L^{1}} + \frac{B}{2} \norm{\nabla \varphi_{n}}_{L^{2}}^{2} \right ) \dt + \left ( c_{0} +  d_{0} s \right ).
\end{align*}
Applying the Gronwall inequality \eqref{Gronwall} gives
\begin{equation}\label{apriori:GL:t}
\begin{aligned}
\left ( A \norm{\Psi_{n}(\varphi(s))}_{L^{1}} + \frac{B}{2} \norm{\nabla \varphi_{n}(s)}_{L^{2}}^{2} \right ) + \frac{1}{2} \norm{\nabla \mu_{n}}_{L^{2}(0,s;L^{2})}^{2} \leq (c_{0} + d_{0} T) e^{k_{0} C_{u} s}
\end{aligned}
\end{equation}
for all $s \in (0,T]$.  Taking supremum in $s$ leads to
\begin{align}\label{apriori:auxProb2:GL}
\norm{\Psi(\varphi_{n})}_{L^{\infty}(0,T;L^{1})} + \norm{\nabla \varphi_{n}}_{L^{\infty}(0,T;L^{2})}^{2} + \norm{\nabla \mu_{n}}_{L^{2}(0,T;L^{2})}^{2} \leq C,
\end{align}
where the constant $C$ depends only on $T$, $C_{u}$, $C_{p}$,  $k_{0}$, $k_{1}$, $A$, $B$, $\abs{\Omega}$, and $\norm{\varphi_{0}}_{H^{1}}$.  From \eqref{mean:mu} and \eqref{apriori:auxProb2:GL}, the mean of $\mu_{n}$ is bounded in $L^{\infty}(0,T)$, and the Poincar\'{e} inequality gives that $\mu_{n}$ is bounded in $L^{2}(0,T;L^{2})$.  Meanwhile, by \eqref{PsiLower} we see that
\begin{align}\label{varphiL1Psi}
\abs{\int_{\Omega} \varphi_{n} \dx} \leq \int_{\Omega} \abs{\varphi_{n}} \dx \leq \frac{1}{k_{2}} \norm{\Psi(\varphi_{n})}_{L^{1}} + \frac{k_{3}}{k_{2}} \abs{\Omega},
\end{align}
and thus by \eqref{apriori:auxProb2:GL}, the mean of $\varphi_{n}$ is bounded in $L^{\infty}(0,T)$, and by the Poincar\'{e} inequality we obtain that $\varphi_{n}$ is also bounded in $L^{\infty}(0,T;L^{2})$.  Thus, there exists a positive constant $C$, not depending on $\phi_{n}$ and $n$ such that
\begin{align*}
\norm{\varphi_{n}}_{L^{\infty}(0,T;H^{1})}^{2} + \norm{\mu_{n}}_{L^{2}(0,T;H^{1})}^{2} \leq C,
\end{align*}
and as a result, this guarantees that the Galerkin solutions $(\varphi_{n}, \mu_{n})$ can be extended to the interval $[0,T]$, and thus $t_{n} = T$ for each $n \in \N$.

\paragraph{Second estimate.}  From \eqref{Psi:growth} and the Sobolev embedding $H^{1} \subset L^{6}$, we have that
\begin{align*}
\norm{\Psi'(\varphi_{n})}_{L^{2}}^{2} \leq C(k_{4}) \left ( \abs{\Omega} + \norm{\varphi_{n}}_{L^{6}}^{6} \right ) \leq C\left ( 1 + \norm{\varphi_{n}}_{H^{1}}^{6} \right),
\end{align*}
where $C$ is a positive constant depending only on $k_{4}$ and $\Omega$.  Since $\norm{\Pi_{n}(\Psi'(\varphi_{n}))}_{L^{2}} \leq \norm{\Psi'(\varphi_{n})}_{L^{2}}$, applying elliptic regularity to \eqref{auxProb2:strong:mun} yields that $\varphi_{n}(t) \in H^{2}$ for a.e. $t \in (0,T)$ and satisfies
\begin{align*}
\norm{\varphi_{n}}_{L^{2}(0,T;H^{2})}^{2} \leq C \left ( \norm{\mu_{n}}_{L^{2}(0,T;L^{2})}^{2} + \norm{\varphi_{n}}_{L^{2}(0,T;L^{2})}^{2} + \norm{\Psi'(\varphi_{n})}_{L^{2}(0,T;L^{2})}^{2} \right ),
\end{align*}
with a positive constant $C$ depending only on $\Omega$, $A$ and $B$.  Then, by the Gagliardo--Nirenberg inequality \eqref{GagNirenIneq} with $d = 3$, $p = 10$, $j = 0$, $r = 2$, $m = 2$, $q = 6$ and $\alpha = \frac{1}{5}$, we have
\begin{align*}
& \norm{f}_{L^{10}(Q)} \leq C \norm{f}_{L^{2}(0,T;H^{2})}^{\frac{1}{5}} \norm{f}_{L^{\infty}(0,T;L^{6})}^{\frac{4}{5}} \\
& \quad \Rightarrow \varphi_{n} \in L^{2}(0,T;H^{2}) \cap L^{\infty}(0,T;H^{1}) \subset L^{10}(Q),
\end{align*} 
and with $d = 3$, $p = \frac{10}{3}$, $j = 0$, $r = 2$, $m = 1$, $q = 2$, and $\alpha = \frac{3}{5}$, we have
\begin{align*}
& \norm{f}_{L^{\frac{10}{3}}(Q)} \leq C \norm{f}_{L^{2}(0,T;H^{1})}^{\frac{3}{5}} \norm{f}_{L^{\infty}(0,T;L^{2})}^{\frac{2}{5}} \\
& \quad \Rightarrow \nabla \varphi \in L^{2}(0,T;H^{1}) \cap L^{\infty}(0,T;L^{2}) \subset L^{\frac{10}{3}}(Q).
\end{align*}
Then, by \eqref{Psi:growth} we have
\begin{align*}
& \norm{\nabla (\Psi'(\varphi_{n}))}_{L^{2}(0,T;L^{2})} = \left ( \int_{0}^{T} \int_{\Omega} \abs{\Psi''(\varphi_{n})}^{2} \abs{\nabla \varphi_{n}}^{2} \dx \dt \right)^{\frac{1}{2}} \\
& \quad \leq \norm{\Psi''(\varphi_{n})}_{L^{5}(Q)} \norm{\nabla \varphi_{n}}_{L^{\frac{10}{3}}(Q)} \leq C(k_{4}) \left ( 1+ \norm{\varphi_{n}}_{L^{10}(Q)}^{2} \right) \norm{\nabla \varphi_{n}}_{L^{\frac{10}{3}}(Q)},
\end{align*}
and so $\Psi'(\varphi_{n}) \in L^{2}(0,T;H^{1})$.  Application of elliptic regularity yields that $\varphi_{n}(t) \in H^{3}$ for a.e. $t \in (0,T)$ and
\begin{align*}
\norm{\varphi_{n}}_{L^{2}(0,T;H^{3})}^{2} \leq C \left ( \norm{\mu_{n}}_{L^{2}(0,T;H^{1})}^{2} + \norm{\varphi_{n}}_{L^{2}(0,T;H^{1})} + \norm{\Psi'(\varphi_{n})}_{L^{2}(0,T;H^{1})}^{2} \right ),
\end{align*}
for a positive constant $C$ not depending on $\phi_{n}$ and $n$.

\paragraph{Third estimate.}  Differentiating \eqref{auxProb2:strong:mun} in time, we obtain
\begin{align}\label{auxProb2:strong:pdtmun}
\pd_{t} \mu_{n} = A \Pi_{n}(\Psi''(\varphi_{n}) \pd_{t}\varphi_{n}) - B \Laplace \pd_{t}\varphi_{n}.
\end{align}
Multiplying \eqref{auxProb2:strong:pdtmun} with $\mu_{n}$ and \eqref{auxProb2:strong:varphin} with $B \pd_{t}\varphi_{n}$, and then integrating over $\Omega$, we obtain upon summing
\begin{equation}\label{reg:mu}
\begin{aligned}
\frac{1}{2} \frac{\dd}{\dt} \norm{\mu_{n}}_{L^{2}}^{2} & + B \norm{\pd_{t}\varphi_{n}}_{L^{2}}^{2} \\
& = \int_{\Omega} Bh(\varphi_{n})( \PP \mathcal{M}_{1}(\phi_{n}) - \AA - \alpha u_{n}) \pd_{t}\varphi_{n} - A \Psi''(\varphi_{n}) \pd_{t}\varphi_{n} \mu_{n} \dx.
\end{aligned}
\end{equation}
From \eqref{Psi:growth}, we find that
\begin{align*}
\norm{\Psi''(\varphi_{n})}_{L^{\infty}(0,T;L^{3})}^{3} \leq C(k_{4}) \left ( \abs{\Omega} + \norm{\varphi_{n}}_{L^{\infty}(0,T;L^{6})}^{6} \right ) \leq C \left (1 + \norm{\varphi_{n}}_{L^{\infty}(0,T;H^{1})}^{6} \right ),
\end{align*}
and so $\Psi''(\varphi_{n})$ is bounded in $L^{\infty}(0,T;L^{3})$.  Applying H\"{o}lder's inequality on the right-hand side of \eqref{reg:mu} yields
\begin{align*}
\frac{1}{2} \frac{\dd}{\dt} \norm{\mu_{n}}_{L^{2}}^{2} + B \norm{\pd_{t}\varphi_{n}}_{L^{2}}^{2} & \leq BC_{u} \norm{\pd_{t}\varphi_{n}}_{L^{1}} + A \norm{\Psi''(\varphi_{n})}_{L^{3}} \norm{\pd_{t}\varphi_{n}}_{L^{2}} \norm{\mu_{n}}_{L^{6}} \\
& \leq BC_{u}^{2} + \frac{2B}{4} \norm{\pd_{t} \varphi_{n}}_{L^{2}}^{2} + \frac{A^{2}C_{\mathrm{Sob}}}{B} \norm{\Psi''(\varphi_{n})}_{L^{\infty}(0,T;L^{3})}^{2} \norm{\mu_{n}}_{H^{1}}^{2},
\end{align*}
where we recall $C_{u} = \PP + \AA + \alpha$ and $C_{\mathrm{Sob}}$ is the positive constant from the Sobolev embedding $H^{1} \subset L^{6}$ depending only on $\Omega$.  Then, integrating in time and using that $\mu_{n}$ is bounded in $L^{2}(0,T;H^{1})$, and $\norm{\mu_{n}(0)}_{L^{2}}^{2} \leq C_{\mathrm{ini}}\norm{\varphi_{0}}_{H^{3}}^{2}$, we have
\begin{align*}
\norm{\mu_{n}}_{L^{\infty}(0,T;L^{2})}^{2} + \norm{\pd_{t}\varphi_{n}}_{L^{2}(0,T;L^{2})}^{2} \leq C,
\end{align*}
where the positive constant $C$ depends only on $\Omega$, $C_{u}$, $A$, $B$, $\norm{\varphi_{n}}_{L^{\infty}(0,T;H^{1})}$, $\norm{\mu_{n}}_{L^{2}(0,T;H^{1})}$, $k_{4}$, and $\norm{\varphi_{0}}_{H^{3}}$.  Furthermore, by \eqref{Psi:growth} we have that
\begin{align*}
\norm{\Psi'(\varphi_{n})}_{L^{\infty}(0,T;L^{2})}^{2} \leq C(k_{4}) \left ( \abs{\Omega} + \norm{\varphi_{n}}_{L^{\infty}(0,T;L^{6})}^{6} \right ).
\end{align*}
Together with the improved regularity $\mu_{n} \in L^{\infty}(0,T;L^{2})$, when we revisit the elliptic equation \eqref{auxProb2:strong:mun} we find that 
\begin{align}\label{varphinLinftyH2}
\norm{\varphi_{n}}_{L^{\infty}(0,T;H^{2})}^{2} \leq C \left ( \norm{\mu_{n}}_{L^{\infty}(0,T;L^{2})}^{2} + \norm{\varphi_{n}}_{L^{\infty}(0,T;L^{2})}^{2} + \norm{\Psi'(\varphi_{n})}_{L^{\infty}(0,T;H^{1})}^{2} \right )
\end{align}
with a positive constant $C$ not depending on $\phi_{n}$ and $n$.  Similarly, viewing \eqref{auxProb2:strong:varphin} as an elliptic problem for $\mu_{n}$, and as $\pd_{t}\varphi_{n} \in L^{2}(0,T;L^{2})$, we have by elliptic regularity
\begin{align*}
\norm{\mu_{n}}_{L^{2}(0,T;H^{2})}^{2} \leq C \left ( \norm{\mu_{n}}_{L^{2}(0,T;L^{2})}^{2} + \norm{\pd_{t}\varphi_{n}}_{L^{2}(0,T;L^{2})}^{2} \right ),
\end{align*}
where the positive constant $C$ does not depend on $\phi_{n}$ or $n$.

\paragraph{Compactness.}
From the above a priori estimates, we obtain for a relabelled subsequence,
\begin{alignat*}{4}
\varphi_{n} & \to \varphi && \text{ weakly* } && \text{ in } L^{2}(0,T;H^{3}) \cap L^{\infty}(0,T;H^{2}) \cap H^{1}(0,T;L^{2}) \\
\mu_{n} & \to \mu && \text{ weakly* } && \text{ in } L^{\infty}(0,T;L^{2}) \cap L^{2}(0,T;H^{2}),
\end{alignat*}
and thanks to the compact embedding \cite[Theorem 6.3 part III]{book:AdamsFournier}
\begin{align*}
W^{j+m,p}(\Omega) \subset \subset C^{j}(\overline{\Omega}) \text{ if } mp > d,
\end{align*}
where $d$ is the space dimension, we find that $H^{2}(\Omega)$ is compactly embedded into $C^{0}(\overline{\Omega})$.  Hence, by \cite[\S 8, Corollary 4]{article:Simon86} we have the following strong convergences
\begin{align*}
\varphi_{n} \to \varphi \text{ strongly in } L^{2}(0,T;W^{2,r}) \cap C^{0}([0,T];W^{1,r}) \cap C^{0}(\overline{Q}),  
\end{align*}
for any $1 \leq r < 6$.  The initial condition $\varphi_{0}$ is attained from that fact that $\varphi \in C^{0}([0,T];H^{1})$.  It follows from standard arguments that the pair $(\varphi, \mu)$ satisfies \eqref{auxProb2:weakform}, see for instance \cite{article:GarckeLamNeumann,article:GarckeLamDirichlet}.  Furthermore, by weak lower semicontinuity of the norms, we obtain \eqref{auxProb2:Est}.

\paragraph{Continuous dependence.}
Let $\{(\varphi_{i},\mu_{i})\}_{i=1,2}$ denote two solution pairs satisfying \eqref{auxProb2:weakform} with the same initial condition $\varphi_{0}$ and corresponding data $\{(\phi_{i}, u_{i})\}_{i=1,2}$, respectively.  Then, it holds that the difference $\varphi := \varphi_{1} - \varphi_{2}$ and $\mu := \mu_{1} - \mu_{2}$ satisfy
\begin{subequations}
\begin{align}
0 & =  \int_{\Omega} \pd_{t}\varphi \zeta + \nabla \mu \cdot \nabla \zeta - h(\varphi_{2}) (\PP \overline{\mathcal{M}_{1}} - \alpha \overline{u}) \zeta - \overline{h}(\PP \mathcal{M}_{1}(\phi_{1}) - \AA - \alpha u_{1}) \zeta \dx \label{ctsdep:varphi} \\
0 & = \int_{\Omega} \mu \zeta - A(\Psi'(\varphi_{1}) - \Psi'(\varphi_{2})) \zeta - B \nabla \varphi \cdot \nabla \zeta \dx, \label{ctsdep:mu} 
\end{align}
\end{subequations}
where
\begin{align*}
\overline{\mathcal{M}_{1}} := \mathcal{M}_{1}(\phi_{1}) - \mathcal{M}_{1}(\phi_{2}), \quad \overline{u} := u_{1} - u_{2}, \quad \overline{h} := h(\varphi_{1}) - h(\varphi_{2}).
\end{align*}
Substituting $\zeta = B \varphi$ in \eqref{ctsdep:varphi} and $\zeta = \mu$ in \eqref{ctsdep:mu}, integrating over $[0,s]$ for $s \in (0,T]$ and upon adding we obtain
\begin{equation}\label{auxProb2:ctsdep1}
\begin{aligned}
& \frac{B}{2} \norm{\varphi(s)}_{L^{2}}^{2}  + \norm{\mu}_{L^{2}(0,s;L^{2})}^{2} \\
& \quad = \int_{0}^{s} \int_{\Omega} A (\Psi'(\varphi_{1}) - \Psi'(\varphi_{2})) \mu + h(\varphi_{2}) (\PP \overline{\mathcal{M}_{1}} - \alpha \overline{u}) B \varphi  \dx \dt \\
& \quad + \int_{0}^{s} \int_{\Omega} \overline{h}(\PP \mathcal{M}_{1}(\phi_{1}) - \AA - \alpha u_{1}) B \varphi \dx \dt .
\end{aligned}
\end{equation}
By the boundedness of $\mathcal{M}_{1}(\phi_{1})$ and $u_{1}$, the Lipschitz continuity of $h$, we obtain
\begin{align*}
\abs{ \int_{0}^{s} \int_{\Omega} (h(\varphi_{2}) - h(\varphi_{1}))(\PP \mathcal{M}_{1}(\phi_{1}) - \AA - \alpha u_{1}) B \varphi \dx \dt} & \leq B L_{h} C_{u} \norm{\varphi}_{L^{2}(0,s;L^{2})}^{2},
\end{align*}
while by H\"{o}lder's inequality and Young's inequality, and the boundedness of $h$, we have
\begin{align*}
& \abs{\int_{0}^{s} \int_{\Omega} h(\varphi_{2}) (\PP \overline{\mathcal{M}_{1}} - \alpha \overline{u}) B \varphi \dx \dt} \\
& \quad \leq \frac{1}{2} \norm{\overline{u}}_{L^{2}(0,s;L^{2})}^{2} + \frac{1}{2} \norm{\overline{\mathcal{M}_{1}}}_{L^{2}(0,s;L^{2})}^{2} + \frac{(B \alpha)^{2} + (B \PP)^{2}}{2}  \norm{\varphi}_{L^{2}(0,s;L^{2})}^{2}.
\end{align*}
Using \eqref{Psi'diff} and the fact that $\varphi_{i} \in C^{0}(\overline{Q})$, we find that
\begin{equation}\label{ctsdep:Psi':difference}
\begin{aligned}
& \abs{\int_{0}^{s} \int_{\Omega} A (\Psi'(\varphi_{1}) - \Psi'(\varphi_{2})) \mu \dx \dt} \leq A k_{5} \int_{0}^{s} \int_{\Omega} (1 + \abs{\varphi_{1}}^{2} + \abs{\varphi_{2}}^{2}) \abs{\varphi} \abs{\mu} \dx \dt \\
& \quad \leq \frac{A^{2}k_{5}^{2}}{2} \left (1 + \norm{\varphi_{1}}_{L^{\infty}(Q)}^{2} + \norm{\varphi_{2}}_{L^{\infty}(Q)}^{2} \right )^{2}\norm{\varphi}_{L^{2}(0,s;L^{2})}^{2} + \frac{1}{2}\norm{\mu}_{L^{2}(0,s;L^{2})}^{2}.
\end{aligned}
\end{equation}
Then, substituting the above three estimates into \eqref{auxProb2:ctsdep1} we obtain for $s \in (0,T]$,
\begin{align}\label{auxProb2:ctsdep:est}
B \norm{\varphi(s)}_{L^{2}}^{2} + \norm{\mu}_{L^{2}(0,s;L^{2})}^{2} \leq \mathcal{Q} \norm{\varphi}_{L^{2}(0,s;L^{2})}^{2} + \norm{\overline{u}}_{L^{2}(0,s;L^{2})}^{2}  +  \norm{\overline{\mathcal{M}_{1}}}_{L^{2}(0,s;L^{2})}^{2},
\end{align}
where
\begin{align*}
\mathcal{Q} := (B\alpha)^{2} + (B \PP)^{2} + A^{2} k_{5}^{2} \left (1 + \norm{\varphi_{1}}_{L^{\infty}(Q)}^{2} + \norm{\varphi_{2}}_{L^{\infty}(Q)}^{2} \right )^{2}
\end{align*}
is a positive constant.  Applying \eqref{Gronwall} yields for any $s \in (0,T]$,
\begin{align*}
& B \norm{\varphi_{1}(s) - \varphi_{2}(s)}_{L^{2}}^{2} + \norm{\mu_{1} - \mu_{2}}_{L^{2}(0,s;L^{2})}^{2} \\
& \quad \leq \left ( \norm{u_{1} - u_{2}}_{L^{2}(0,s;L^{2})}^{2}  +  \norm{\mathcal{M}_{1}(\phi_{1}) - \mathcal{M}_{1}(\phi_{2})}_{L^{2}(0,s;L^{2})}^{2} \right ) e^{\frac{\mathcal{Q}}{B}s} 
\end{align*}
where we used that $W(t) := \norm{\overline{u}}_{L^{2}(0,t;L^{2})}^{2} + \norm{\overline{\mathcal{M}_{1}}}_{L^{2}(0,t;L^{2})}^{2}$ is a nondecreasing function of $t$, and thus 
\begin{align*}
W(s) + \int_{0}^{s} W(t) \frac{\mathcal{Q}}{B} e^{\frac{\mathcal{Q}}{B}t} \dt \leq W(s) \left ( 1 + \int_{0}^{s} \frac{\mathcal{Q}}{B} e^{\frac{\mathcal{Q}}{B}t} \dt \right ) = W(s) e^{\frac{\mathcal{Q}}{B}s}.
\end{align*}
\end{proof}

We point out that although the source term in \eqref{state:varphi} closely resembles that of \cite{article:GarckeLamNeumann}, we obtain a priori estimates for potentials $\Psi$ with quartic growth (see \eqref{Psi:growth}), which is in contrast to the quadratic potentials considered in \cite{article:GarckeLamNeumann}.  The main difference is that here we have the boundedness of the nutrient, and thus we only require a bound on the mean of $\mu$ in \eqref{sourceterm:est}.  But in \cite{article:GarckeLamNeumann}, the presence of the active transport mechanism (modeled by the term $\div (n(\varphi)\chi \nabla \varphi)$ in the nutrient equation) prevents us from applying a weak comparison principle to deduce the boundedness of the nutrient.  Without the boundedness of the nutrient, we have to control the square of the mean of $\mu$ in order to estimate the source term in \eqref{sourceterm:est}. 

\subsection{Existence by Schauder's fixed point theorem}

Note that if $\{\phi_{n}\}_{n \in \N}$ is a bounded sequence in $L^{2}(Q)$, by Lemma \ref{lem:auxProb1:nutrient} the corresponding sequence $\{\sigma_{n} := \mathcal{M}_{1}(\phi_{n})\}_{n \in \N}$ satisfies $0 \leq \sigma_{n} \leq 1$ a.e. in $Q$, and by Lemma \ref{lem:auxCahnHilliard} we have that the corresponding solution pair $\{\varphi_{n},\mu_{n}\}_{n \in \N}$ is bounded uniformly in
\begin{align*}
\left (L^{\infty}(0,T;H^{2}) \cap L^{2}(0,T;H^{3}) \cap H^{1}(0,T;L^{2}) \right ) \times \left ( L^{2}(0,T;H^{2}) \cap L^{\infty}(0,T;L^{2}) \right )
\end{align*}
which yields a strongly convergent (relabelled) subsequence $\{\varphi_{n}\}_{n \in N}$ in $L^{2}(Q)$, due to the compact embedding
\begin{align*}
L^{2}(0,T;H^{1}) \cap H^{1}(0,T;L^{2}) \subset \subset L^{2}(Q).
\end{align*}
Thus, the mapping
\begin{alignat*}{4}
\mathcal{M} & :L^{2}(Q) && \to && L^{2}(Q) \\
 & \quad \quad \phi && \mapsto &&  \quad \varphi \text{ satisfying } \eqref{auxProb2:weakform} 
\end{alignat*}
is compact.  To apply Schauder's fixed point theorem and deduce the existence of a fixed point of the mapping $\mathcal{M}$, we need to check that if there exists a constant $M$ such that
\begin{align*}
\norm{\phi}_{L^{2}(Q)} \leq M \text{ for all } \phi \in L^{2}(Q) \text{ and for all } \lambda \in [0,1] \text{ satisfying } \phi = \lambda \mathcal{M}(\phi).
\end{align*}
The problem $\varphi = \lambda \mathcal{M}(\varphi)$ translates to
\begin{align*}
\pd_{t}\varphi & = \Laplace \mu + (\PP \sigma - \AA - \alpha u) h(\varphi), \\
\mu & = A \Psi'(\varphi) - B \Laplace \varphi, \\
\pd_{t} \sigma & = \Laplace \sigma - \CC h(\lambda \varphi) \sigma + \BB (\sigma_{S} - \sigma).
\end{align*}
By Lemma \ref{lem:auxProb1:nutrient} we have that $0 \leq \sigma \leq 1$ a.e. in $Q$ for all $\lambda \in [0,1]$, and thus we can choose $M$ to be the constant $C_{\mathrm{AP2}}$ in \eqref{auxProb2:Est} which does not depend on $\varphi$ and $\lambda \in [0,1]$.  Thus Schauder's fixed point theorem yields the existence of a weak solution $(\varphi, \mu, \sigma)$ to the state equations \eqref{eq:state} with $0 \leq \sigma \leq 1$ a.e. in $Q$ and
\begin{equation}\label{solution:estimate}
\begin{aligned}
& \norm{\varphi}_{L^{\infty}(0,T;H^{2}) \cap L^{2}(0,T;H^{3}) \cap H^{1}(0,T;L^{2})} \\
& \quad + \norm{\mu}_{L^{2}(0,T;H^{2}) \cap L^{\infty}(0,T;L^{2})} + \norm{\sigma}_{L^{2}(0,T;H^{2}) \cap L^{\infty}(0,T;H^{1}) \cap H^{1}(0,T;L^{2})} \leq \overline{C},
\end{aligned}
\end{equation}  
for some positive constant $\overline{C}$ not depending on $(\varphi, \mu, \sigma, u)$.

\subsection{Continuous dependence}
We now establish continuous dependence on the control $u$.  For this purpose, let $u_{1}, u_{2}  \in \mathcal{U}_{\mathrm{ad}}$ be given, along with the corresponding solution triplet $(\varphi_{1}, \mu_{1}, \sigma_{1})$ and $(\varphi_{2}, \mu_{2}, \sigma_{2})$ satisfying the same initial data $\varphi_{0}$ and $\sigma_{0}$.  Let $\varphi = \varphi_{1} - \varphi_{2}$, $\mu = \mu_{1} - \mu_{2}$ and $\sigma = \sigma_{1} - \sigma_{2}$, then from
\eqref{auxProb1:ctsdep:sigmaL2L2} we obtain
\begin{align*}
\norm{\sigma}_{L^{2}(0,s;L^{2})}^{2} = \int_{0}^{s} \norm{\sigma(t)}_{L^{2}}^{2} \dt \leq \CC L_{h} \int_{0}^{s} \norm{\varphi}_{L^{2}(0,t;L^{2})}^{2} e^{t} \dt \leq \CC L_{h} (e^{s} - 1) \norm{\varphi}_{L^{2}(0,s;L^{2})}^{2}.
\end{align*}
Substituting this into \eqref{auxProb2:ctsdep:est} leads to
\begin{align*}
B\norm{\varphi(s)}_{L^{2}}^{2} + \norm{\mu}_{L^{2}(0,s;L^{2})}^{2} & \leq \mathcal{Q} \norm{\varphi}_{L^{2}(0,s;L^{2})}^{2} + \norm{u_{1} - u_{2}}_{L^{2}(0,s;L^{2})}^{2} + \norm{\sigma}_{L^{2}(0,s;L^{2})}^{2}  \\
& \leq \left ( \mathcal{Q} + \CC L_{h} (e^{s} - 1) \right ) \norm{\varphi}_{L^{2}(0,s;L^{2})}^{2} + \norm{u_{1} - u_{2}}_{L^{2}(0,s;L^{2})}^{2}.
\end{align*}
Setting
\begin{align*}
W(s) = \norm{u_{1} - u_{2}}_{L^{2}(0,s;L^{2})}^{2}, \; X(t) = \frac{\mathcal{Q} + \CC L_{h} (e^{s} - 1)}{B}, \; Y(s) = B \norm{\varphi(s)}_{L^{2}}^{2}, \; Z(t) = \norm{\mu}_{L^{2}}^{2}, 
\end{align*}
we obtain from \eqref{Gronwall} that
\begin{align*}
B \norm{\varphi(s)}_{L^{2}}^{2} + \norm{\mu}_{L^{2}(0,s;L^{2})}^{2} \leq \norm{u_{1} - u_{2}}_{L^{2}(0,s;L^{2})}^{2} \exp \left (s \left ( \frac{\mathcal{Q} + \CC L_{h} (e^{s} - 1)}{B} \right ) \right ) \text{ for } s \in (0,T].
\end{align*}
Combining with \eqref{auxProb1:ctsdep:est1} and \eqref{auxProb1:ctsdep:est2}, we find that there exists a positive constant $C_{1}$, depending only on $B$, $\mathcal{Q}$, $\CC$, $L_{h}$, $T$ such that
\begin{align}\label{ctsdep:main:est1}
\norm{\varphi(s)}_{L^{2}}^{2} + \norm{\sigma(s)}_{H^{1}}^{2}  + \norm{\mu}_{L^{2}(0,s;L^{2})}^{2} + \norm{\pd_{t}\sigma}_{L^{2}(0,s;L^{2})}^{2} \leq C_{1} \norm{u_{1} - u_{2}}_{L^{2}(0,s;L^{2})}^{2}
\end{align}
for $s \in (0,T]$.  Next, we find using \eqref{Psi'diff} and the fact that $\varphi_{i} \in C^{0}(\overline{Q})$ for $i = 1,2$,
\begin{align*}
\norm{\Psi'(\varphi_{1}) - \Psi'(\varphi_{2})}_{L^{2}(0,s;L^{2})}^{2} \leq k_{5}^{2} \left ( 1 + \norm{\varphi_{1}}_{L^{\infty}(Q)} + \norm{\varphi_{2}}_{L^{\infty}(Q)} \right )^{4} \norm{\varphi}_{L^{2}(0,s;L^{2})}^{2},
\end{align*}
and so viewing \eqref{ctsdep:mu} as an elliptic problem for $\varphi$, we obtain by elliptic regularity
\begin{align*}
\norm{\varphi}_{L^{2}(0,s;H^{2})}^{2} & \leq C \left ( \norm{\varphi}_{L^{2}(0,s;L^{2})}^{2} + \norm{\Psi'(\varphi_{1}) - \Psi'(\varphi_{2})}_{L^{2}(0,s;L^{2})}^{2} + \norm{\mu}_{L^{2}(0,s;L^{2})}^{2} \right ) \\
& \leq C_{2} \norm{u_{1} - u_{2}}_{L^{2}(0,s;L^{2})}^{2},
\end{align*}
where $C_{2}$ is a positive constant depending only on $\Omega$, $A$, $k_{5}$, $\norm{\varphi_{i}}_{L^{\infty}(Q)}$, $T$ and $C_{1}$.  


\section{Existence of a minimizer}\label{sec:minimizer}
From \eqref{solution:estimate} it holds that
\begin{align*}
\frac{1}{r} \int_{\tau-r}^{\tau} \int_{\Omega} \varphi \dx \dt \geq - \frac{1}{r} \norm{\varphi}_{L^{1}(0,T;L^{1})} \geq - \overline{C},
\end{align*}
where $\overline{C}$ is a positive constant independent of $(\varphi, \mu, \sigma, u)$.  Hence, we obtain that
\begin{align*}
J_{r}(\varphi, u, \tau) \geq \frac{\beta_{S}}{2} \frac{1}{r} \int_{\tau-r}^{\tau} \int_{\Omega} \varphi \dx \dt \geq - \frac{\beta_{S}}{2} \overline{C} > - \infty.
\end{align*}
As $J_{r}$ is bounded from below, we can consider a minimising sequence $(u_{n}, \tau_{n})_{n \in \N}$ with $u_{n} \in \mathcal{U}_{\mathrm{ad}}$, $\tau_{n} \in (0,T)$  and corresponding weak solutions $(\varphi_{n}, \mu_{n}, \sigma_{n})_{n \in \N}$ on the interval $[0,T]$ with $\varphi_{n}(0) = \varphi_{0}$ and $\sigma_{n}(0) = \sigma_{0}$ for all $n \in \N$, such that
\begin{align*}
\lim_{n \to \infty} J_{r}(\varphi_{n}, u_{n}, \tau_{n}) = \inf_{(\phi, w, s)} J_{r}(\phi, w, s).
\end{align*}
In particular, $u_{n} \in \mathcal{U}_{\mathrm{ad}}$ implies that $0 \leq u_{n} \leq 1$ a.e. in $Q$ for all $n \in \N$.  As $\{\tau_{n}\}_{n \in \N}$ is a bounded sequence, there exists a relabelled subsequence such that
\begin{align*}
\tau_{n} \to \Optime \in [0,T] \text{ as } n \to \infty,
\end{align*}
and
\begin{alignat*}{3}
u_{n} & \to u_{*} && \text{ weakly* } && \text{ in }  L^{\infty}(Q), \\
\varphi_{n} & \to \varphi_{*} && \text{ weakly* } && \text{ in }  L^{\infty}(0,T;H^{2}) \cap L^{2}(0,T;H^{3}) \cap H^{1}(0,T;L^{2}), \\
\varphi_{n} & \to \varphi_{*} && \text{ strongly } && \text{ in } C^{0}([0,T];L^{2}) \cap L^{2}(0,T;L^{2}), \\
\mu_{n} & \to \mu_{*} && \text{ weakly* } && \text{ in }  L^{2}(0,T;H^{2}) \cap L^{\infty}(0,T;L^{2}), \\
\sigma_{n} & \to \sigma_{*} && \text{ weakly* } && \text{ in }  L^{\infty}(0,T;H^{1}) \cap L^{2}(0,T;H^{2}) \cap H^{1}(0,T;L^{2}) \cap L^{\infty}(Q),
\end{alignat*}
where $(\varphi_{*}, \mu_{*}, \sigma_{*}, u_{*})$ satisfy \eqref{weak:state} with $0 \leq u_{*}, \sigma_{*} \leq 1$ a.e. in $Q$.  Note that by the dominated convergence theorem, for all $p \in [1,\infty)$,
\begin{align*}
\chi_{[0,\tau_{n}]}(t) \to \chi_{[0,\Optime]}(t), \quad  \chi_{[\tau_{n} - r,\tau_{n}]}(t) \to \chi_{[\Optime - r,\Optime]}(t) \quad \text{ strongly in } L^{p}(0,T).
\end{align*}
Then, by the strong convergence of $\varphi_{n} - \varphi_{Q}$ to $\varphi_{*} - \varphi_{Q}$ in $L^{2}(Q)$ and the strong convergence $\chi_{[0,\tau_{n}]}(t)$ to $\chi_{[0,\Optime]}(t)$ also in $L^{2}(Q)$, we have
\begin{equation}\label{minimizer:1}
\begin{aligned}
& \int_{0}^{\tau_{n}} \int_{\Omega} \abs{\varphi_{n} - \varphi_{Q}}^{2} \dx \dt  = \int_{0}^{T} \norm{\varphi_{n} - \varphi_{Q}}_{L^{2}}^{2} \chi_{[0,\tau_{n}]}(t) \dt \\
& \quad \longrightarrow \int_{0}^{T} \norm{\varphi_{*} - \varphi_{Q}}_{L^{2}}^{2} \chi_{[0,\Optime]}(t) \dt = \int_{0}^{\Optime} \int_{\Omega} \abs{\varphi_{*} - \varphi_{Q}}^{2} \dx \dt \text{ as } n \to \infty.
\end{aligned}
\end{equation}
A similar argument yields
\begin{equation}\label{minimizer:2}
\begin{aligned}
& \frac{1}{r} \int_{\tau_{n} - r}^{\tau_{n}} \left ( \frac{\beta_{\Omega}}{2} \norm{\varphi_{n} - \varphi_{\Omega}}^{2} + \frac{\beta_{S}}{2} \int_{\Omega} 1 + \varphi_{n} \dx \right ) \dt \\
& \quad \longrightarrow \frac{1}{r} \int_{\Optime - r}^{\Optime} \left ( \frac{\beta_{\Omega}}{2} \norm{\varphi_{*} - \varphi_{\Omega}}^{2} + \frac{\beta_{S}}{2} \int_{\Omega} 1 + \varphi_{*} \dx \right ) \dt \text{ as } n \to \infty.
\end{aligned}
\end{equation}
Next, using the weak lower semicontinuity of the $L^{2}(Q)$-norm, we have 
\begin{align*}
\norm{u_{*}}_{L^{2}(Q)} \leq \liminf_{n \to \infty} \norm{u_{n}}_{L^{2}(Q)},
\end{align*}
and so, passing to the limit $n \to \infty$ in the above inequality and using the convergence $\tau_{n} \to \Optime$, we see that
\begin{align}\label{minimizer:3}
\liminf_{n \to \infty} \int_{0}^{\tau_{n}} \norm{u_{n}}_{L^{2}}^{2} \dt - \int_{0}^{\Optime} \norm{u_{*}}_{L^{2}}^{2} \dt \geq 0.
\end{align}
Then, by passing to the limit $n \to \infty$ in $J_{r}(\varphi_{n}, u_{n}, \tau_{n})$ and using \eqref{minimizer:1}, \eqref{minimizer:2}, and \eqref{minimizer:3}, we have
\begin{align*}
\inf_{(\phi, w, s)} J_{r}(\phi, w, s) = \lim_{n \to \infty} J_{r}(\varphi_{n}, u_{n}, \tau_{n}) \geq J_{r}(\varphi_{*}, u_{*}, \Optime),
\end{align*}
which implies that $(u_{*},\Optime)$ is a minimizer of \eqref{OCProblem}.

\section{Fr\'{e}chet differentibility of the solution operator} \label{sec:Frechetdiff}
\subsection{Unique solavability of the linearized state equations}\label{sec:LinearizedState}
Recalling the set $\{w_{i}\}_{i \in \N}$ of eigenfunctions of the Neumann-Laplacian from the proof of Lemma \ref{lem:auxCahnHilliard}, we look for functions of the form
\begin{align*}
\Phi_{n}(x,t) := \sum_{i=1}^{n} \gamma_{n,i}(t) w_{i}(x), \quad \Xi_{n}(x,t) := \sum_{i=1}^{n} \delta_{n,i}(t)w_{i}(x), \quad \Sigma_{n}(x,t) := \sum_{i=1}^{n} \eta_{n,i}(t)w_{i}(x)
\end{align*}
satisfying
\begin{subequations}\label{Galerkin:Linearized}
\begin{align}
0 & = \int_{\Omega} \pd_{t} \Phi_{n} v + \nabla \Xi_{n} \cdot \nabla v - h(\overline{\varphi})(\PP \Sigma_{n} - \alpha w) v - h'(\overline{\varphi}) (\PP \overline{\sigma} - \AA - \alpha \overline{u}) \Phi_{n} v \dx, \label{Linearized:varphi} \\
0 & = \int_{\Omega} \Xi_{n} v - A \Psi''(\overline{\varphi}) \Phi_{n} v - B \nabla \Phi_{n} \cdot \nabla v \dx, \label{Linearized:mu} \\
0 & = \int_{\Omega} \pd_{t} \Sigma_{n} v + \nabla \Sigma_{n} \cdot \nabla v  + \BB \Sigma_{n} v +  \CC (h(\overline{\varphi}) \Sigma_{n} + h'(\overline{\varphi}) \Phi_{n} \overline{\sigma}) v \dx, \label{Linearized:sigma}
\end{align}
\end{subequations}
for all $v \in W_{n}$.  Substituting $v = w_{j}$ leads to 
\begin{subequations}\label{Linearized:ODE}
\begin{align}
\bm{\gamma}_{n}' & = - \bm{S} \bm{\delta}_{n} - \bm{M}_{n}^{h} \PP \bm{\eta}_{n} + \bm{J}_{n} - \bm{K}_{n} \bm{\gamma}_{n}, \\
\bm{\delta}_{n} & = A \bm{\phi}_{n} + B \bm{S} \bm{\gamma}_{n}, \\
\bm{\eta}_{n}' & = - \bm{S} \bm{\eta}_{n} - \BB \bm{\eta}_{n} - \CC \bm{M}_{n}^{h} \bm{\eta}_{n} - \CC \bm{L}_{n} \bm{\gamma}_{n},
\end{align}
\end{subequations}
where the matrix $\bm{S}$ has been defined in \eqref{ODE:vectormatrix}, and for $1 \leq i,j \leq n$,
\begin{alignat*}{4}
(\bm{M}_{n}^{h})_{ij} & := \int_{\Omega} h(\overline{\varphi}) w_{i} w_{j} dx, \quad && (\bm{J}_{n})_{j} && := \int_{\Omega} h(\overline{\varphi}) \alpha w w_{j} \dx, \\
(\bm{K}_{n})_{ij} & := \int_{\Omega} h'(\overline{\varphi}) (\PP \overline{\sigma} - \AA - \alpha \overline{u}) w_{i} w_{j} \dx, \quad && (\bm{\phi}_{n})_{j} && := \int_{\Omega} \Psi''(\overline{\varphi}) \Phi_{n} w_{j} \dx, \\
(\bm{L}_{n})_{ij} & := \int_{\Omega} h'(\overline{\varphi}) \overline{\sigma} w_{i} w_{j} \dx. \quad && &&
\end{alignat*}
Taking an approximating sequence in $C^{0}([0,T];L^{2})$ for $\overline{u}$, which we will abuse notation and reuse the variable $\overline{u}$, and then supplementing \eqref{Linearized:ODE} with the initial conditions $\bm{\gamma}_{n}(0) = \bm{0}$ and $\bm{\eta}_{n}(0) = \bm{0}$ leads to a system of ODEs with right-hand sides depending continuously on $(t, \bm{\gamma}_{n}, \bm{\eta}_{n})$.  Thus, by the Cauchy--Peano theorem, there exists $t_{n} \in (0,T]$ such that \eqref{Linearized:ODE} has a local solution $(\bm{\gamma}_{n}, \bm{\delta}_{n}, \bm{\eta}_{n})$ on $[0,t_{n}]$ with $\bm{\gamma}_{n}, \bm{\delta}_{n}, \bm{\eta}_{n} \in C^{1}([0,t_{n});\R^{n})$.  Then,  we obtain functions $\Phi_{n}, \Xi_{n}, \Sigma_{n} \in C^{1}([0,t_{n});W_{n})$ satisfying \eqref{Galerkin:Linearized}.

\paragraph{First estimate.}
Substituting $v = \Phi_{n}$ in \eqref{Linearized:varphi}, $v = \Laplace \Phi_{n}$ in \eqref{Linearized:mu} and $v = \Sigma_{n}$ in \eqref{Linearized:sigma}, integrating over $[0,t]$ for $t \in (0,T]$, and integrating by parts, we obtain after summation
\begin{align*}
\frac{1}{2} & \left ( \norm{\Phi_{n}(t)}_{L^{2}}^{2} + \norm{\Sigma_{n}(t)}_{L^{2}}^{2} \right ) + B \norm{\Laplace \Phi_{n}}_{L^{2}(0,t;L^{2})}^{2} + \norm{\nabla \Sigma_{n}}_{L^{2}(0,t;L^{2})}^{2}  \\
& \leq \int_{0}^{t} \int_{\Omega} h(\overline{\varphi}) (\PP \Sigma_{n} - \alpha w) \Phi_{n} + h'(\overline{\varphi}) (\PP \overline{\sigma} - \AA - \alpha \overline{u}) \abs{\Phi_{n}}^{2} + A \Psi''(\overline{\varphi}) \Phi_{n} \Laplace \Phi_{n} \dx \dt \\
& - \int_{0}^{t} \int_{\Omega} \CC h'(\overline{\varphi}) \overline{\sigma} \Phi_{n} \Sigma_{n} \dx \dt =: I_{1} + I_{2} + I_{3} + I_{4},
\end{align*}
where we used that $\Sigma_{n}(0) = \Phi_{n}(0) = 0$ and have neglected the nonnegative term $(\BB + \CC h(\overline{\varphi})) \abs{\Sigma_{n}}^{2}$.  From Theorem \ref{thm:state}, we have $\overline{\varphi} \in C^{0}(\overline{Q})$, and as $\Psi''$, $\Psi'''$, $h'$ and $h''$ are continuous with respect to their arguments, it holds that there exists a constant $C_{*} > 0$ such that
\begin{align}\label{BoundingConst:hPsi}
\sup_{(x,t) \in \overline{Q} } \left ( \abs{h'(\overline{\varphi}(x,t))} + \abs{h''(\overline{\varphi}(x,t))} + \abs{\Psi''(\overline{\varphi}(x,t))} + \abs{\Psi'''(\overline{\varphi}(x,t))} \right ) \leq C_{*}.
\end{align}
Then, applying H\"{o}lder's inequality and Young's inequality we obtain
\begin{align*}
\abs{I_{4}} & \leq \CC^{2} C_{*}^{2} \norm{\Phi_{n}}_{L^{2}(0,t;L^{2})}^{2} + \frac{1}{4} \norm{\Sigma_{n}}_{L^{2}(0,t;L^{2})}^{2}, \\
\abs{I_{3}} & \leq \frac{B}{2} \norm{\Laplace \Phi_{n}}_{L^{2}(0,t;L^{2})}^{2} + \frac{(AC_{*})^{2}}{2B} \norm{\Phi_{n}}_{L^{2}(0,t;L^{2})}^{2}, \\
\abs{I_{2}} & \leq C_{*}  \left (\PP + \AA + \alpha \right ) \norm{\Phi_{n}}_{L^{2}(0,t;L^{2})}^{2}, \\
\abs{I_{1}} & \leq \left ( \PP^{2} + 1 \right )\norm{\Phi_{n}}_{L^{2}(0,t;L^{2})}^{2} + \frac{1}{4} \norm{\Sigma_{n}}_{L^{2}(0,t;L^{2})}^{2} + \frac{\alpha^{2}}{4} \norm{w}_{L^{2}(0,t;L^{2})}^{2}.
\end{align*}
Using the estimates for $I_{1}$, $I_{2}$, $I_{3}$ and $I_{4}$, we obtain
\begin{align*}
 \norm{\Phi_{n}(t)}_{L^{2}}^{2} &  - C_{5} \norm{\Phi_{n}}_{L^{2}(0,t;L^{2})}^{2}  + \norm{\Sigma_{n}(t)}_{L^{2}}^{2} - C_{5} \norm{\Sigma_{n}}_{L^{2}(0,t;L^{2})}^{2} \\
& + \norm{\Laplace \Phi_{n}}_{L^{2}(0,t;L^{2})}^{2} + \norm{\nabla \Sigma_{n}}_{L^{2}(0,t;L^{2})}^{2}   \leq C_{6} \norm{w}_{L^{2}(0,t;L^{2})}^{2},
\end{align*}
where $C_{5}, C_{6} > 0$ are positive constants depending only on $C_{*}$, $A$, $B$, $\PP$, $\AA$, $\CC$, and $\alpha$.  Applying the integral form of Gronwall's inequality we obtain that
\begin{equation}\label{Linearized:est1}
\begin{aligned}
\norm{\Phi_{n}}_{L^{\infty}(0,T;L^{2})}^{2} + \norm{\Sigma_{n}}_{L^{\infty}(0,T;L^{2})}^{2} + \norm{\Sigma_{n}}_{L^{2}(0,T;H^{1})}^{2} + \norm{\Laplace \Phi_{n}}_{L^{2}(Q)}^{2} \leq D_{1} \norm{w}_{L^{2}(Q)}^{2},
\end{aligned}
\end{equation}
for some constant $D_{1}$ not depending on $n$, which in turn implies that
\begin{align*}
\{\Phi_{n}\}_{n \in \N} & \text{ is bounded uniformly in } L^{\infty}(0,T;L^{2}), \\
\{\Laplace \Phi_{n}\}_{n \in \N} & \text{ is bounded uniformly in } L^{2}(0,T;L^{2}), \\
\{\Sigma_{n}\}_{n \in \N} & \text{ is bounded uniformly in } L^{\infty}(0,T;L^{2}) \cap L^{2}(0,T;H^{1}).
\end{align*}

\paragraph{Second estimate.}
Substituting $v = \pd_{t} \Sigma_{n}$ in \eqref{Linearized:sigma}, we obtain
\begin{align*}
\frac{1}{2} \frac{\dd}{\dt} \norm{\nabla \Sigma_{n}}_{L^{2}}^{2} + \norm{\pd_{t} \Sigma_{n}}_{L^{2}}^{2} & = - \int_{\Omega} \BB \Sigma_{n} \pd_{t} \Sigma_{n} + \CC (h(\overline{\varphi}) \Sigma_{n} \pd_{t} \Sigma_{n} + h'(\overline{\varphi}) \Phi_{n} \pd_{t} \Sigma_{n}) \dx \\
& \leq \frac{3}{4} \norm{\pd_{t} \Sigma_{n}}_{L^{2}}^{2} + (\BB^{2} + \CC^{2})\norm{\Sigma_{n}}_{L^{2}}^{2} + \CC^{2} C_{*}^{2} \norm{\Phi_{n}}_{L^{2}}^{2}.
\end{align*}
Applying Gronwall's inequality yields that
\begin{equation}\label{Linearized:est2}
\begin{aligned}
\norm{\Sigma_{n}}_{L^{\infty}(0,T;H^{1})}^{2} + \norm{\pd_{t}\Sigma_{n}}_{L^{2}(Q)}^{2} \leq D_{2} \norm{w}_{L^{2}(Q)}^{2},
\end{aligned}
\end{equation}
where $D_{2}$ is a positive constant not depending on $n$.  Hence,
\begin{align*}
\{\Sigma_{n}\}_{n \in \N} \text{ is bounded uniformly in } L^{\infty}(0,T;H^{1}) \cap H^{1}(0,T;L^{2}).
\end{align*}
Furthermore, since $\CC(h(\overline{\varphi}) \Sigma_{n} + h'(\overline{\varphi}) \Phi_{n} \overline{\sigma}) - \pd_{t} \Sigma_{n} \in L^{2}$ for a.e. $t \in (0,T)$, we obtain from elliptic regularity theory that
\begin{equation}\label{Linearized:est3}
\begin{aligned}
\norm{\Sigma_{n}}_{L^{2}(0,T;H^{2})}^{2} & \leq  C \left ( \norm{\Sigma_{n}}_{L^{2}(Q)}^{2} + \norm{\pd_{t} \Sigma_{n}}_{L^{2}(Q)}^{2} + \norm{\Phi_{n}}_{L^{2}(Q)}^{2} \right ) \leq D_{3} \norm{w}_{L^{2}(Q)}^{2},
\end{aligned}
\end{equation}
where $C$ and $D_{3}$ are positive constants not depending on $n$.  Thus,
\begin{align*}
\{\Sigma_{n}\}_{n \in N} \text{ is bounded uniformly in } L^{2}(0,T;H^{2}) .
\end{align*}

\paragraph{Third estimate.}
Substituting $v = 1$ in \eqref{Linearized:mu} yields
\begin{align*}
\abs{\int_{\Omega} \Xi_{n} \dx} = \abs{\int_{\Omega} A \Psi''(\overline{\varphi}) \Phi_{n} \dx} \leq A C_{*} \norm{\Phi_{n}}_{L^{1}} \leq AC_{*} \abs{\Omega}^{\frac{1}{2}} \norm{\Phi_{n}}_{L^{2}}.
\end{align*}
Then, by the Poincar\'{e} inequality we find that
\begin{align*}
\norm{\Xi_{n}}_{L^{2}} \leq C_{p}\norm{\nabla \Xi_{n}}_{L^{2}} + \frac{1}{\abs{\Omega}^{\frac{1}{2}}} \abs{\int_{\Omega} \Xi_{n}\dx} \leq C_{p} \norm{\nabla \Xi_{n}}_{L^{2}} + AC_{*} \norm{\Phi_{n}}_{L^{2}},
\end{align*}
and thus
\begin{align}\label{meanvalue:XiL2L2}
\norm{\Xi_{n}}_{L^{2}(0,t;L^{2})}^{2} \leq 2C_{p}^{2} \norm{\nabla \Xi_{n}}_{L^{2}(0,t;L^{2})}^{2} + 2 A^{2}C_{*}^{2} \norm{\Phi_{n}}_{L^{2}(0,t;L^{2})}^{2}.
\end{align}
Substituting $v = \Xi_{n}$ in \eqref{Linearized:varphi} and $v = - \pd_{t} \Phi_{n}$ in \eqref{Linearized:mu}, and upon summing and integrating over $[0,t]$ for $t \in (0,T]$, we obtain
\begin{equation}\label{Linearized:apriori}
\begin{aligned}
& \frac{B}{2}  \norm{\nabla \Phi_{n}(t)}_{L^{2}}^{2} + \norm{\nabla \Xi_{n}}_{L^{2}(0,t;L^{2})}^{2} \\
& \quad = \int_{0}^{t} \int_{\Omega} - A \Psi''(\overline{\varphi}) \Phi_{n} \pd_{t} \Phi_{n} + h(\overline{\varphi}) (\PP \Sigma_{n} - \alpha w) \Xi_{n} \dx \dt \\
& \quad  + \int_{0}^{t} \int_{\Omega}  h'(\overline{\varphi}) (\PP \overline{\sigma} - \AA - \alpha \overline{u}) \Phi_{n} \Xi_{n}  \dx \dt  =: J_{1} + J_{2} + J_{3}.
\end{aligned}
\end{equation}
Applying H\"{o}lder's inequality and Young's inequality and \eqref{meanvalue:XiL2L2}, we observe that
\begin{align*}
\abs{J_{3}} & \leq C_{*} \left  (\PP + \AA + \alpha  \right ) \norm{\Phi_{n}}_{L^{2}(0,t;L^{2})} \norm{\Xi_{n}}_{L^{2}(0,t;L^{2})} \\
& \leq \frac{1}{4} \norm{\nabla \Xi_{n}}_{L^{2}(0,t;L^{2})}^{2} + C_{7}\norm{\Phi_{n}}_{L^{2}(0,t;L^{2})}^{2},  \\
\abs{J_{2}} & \leq 2 C_{p}^{2} \left ( \PP \norm{\Sigma_{n}}_{L^{2}(0,t;L^{2})} + \alpha \norm{w}_{L^{2}(0,t;L^{2})} \right )^{2} + \frac{1}{8C_{p}^{2}} \norm{\Xi_{n}}_{L^{2}(0,t;L^{2})}^{2} \\
& \leq \frac{1}{4} \norm{\nabla \Xi_{n}}_{L^{2}(0,t;L^{2})}^{2} + C_{8} \left ( \norm{\Phi_{n}}_{L^{2}(0,t;L^{2})}^{2} + \norm{\Sigma_{n}}_{L^{2}(0,t;L^{2})}^{2} + \norm{w}_{L^{2}(0,t;L^{2})}^{2} \right ),
\end{align*}
where $C_{7}, C_{8} > 0$ are positive constants depending only on $C_{*}$, $A$, $\PP$, $\AA$, and $\alpha$.  To estimate $J_{1}$ we first obtain an estimate for $\norm{\pd_{t}\Phi_{n}}_{L^{2}(0,t(H^{1})^{*})}$ by considering $v \in L^{2}(0,T;H^{1})$ in \eqref{Linearized:varphi} and integrating over $[0,t]$.  Then, we obtain that
\begin{equation}\label{Linearized:timederivative}
\begin{aligned}
\norm{\pd_{t}\Phi_{n}}_{L^{2}(0,t;(H^{1})^{*})} & \leq \norm{\nabla \Xi_{n}}_{L^{2}(0,t;L^{2})} + \PP \norm{\Sigma_{n}}_{L^{2}(0,t;L^{2})} + \alpha \norm{w}_{L^{2}(0,t;L^{2})}  \\
& + C_{*} \left (\PP + \AA + \alpha \right ) \norm{\Phi_{n}}_{L^{2}(0,t;L^{2})}.
\end{aligned}
\end{equation}
Thus, for $J_{1}$ we have
\begin{align*}
\abs{J_{1}}& \leq AC_{*} \left ( \norm{\Phi_{n}}_{L^{2}(0,t;L^{2})} + \norm{\nabla \Phi_{n}}_{L^{2}(0,t;L^{2})} \right ) \norm{\pd_{t}\Phi_{n}}_{L^{2}(0,t;(H^{1})^{*})} \\
& \leq C_{9} \left (\norm{\Phi_{n}}_{L^{2}(0,t;L^{2})}^{2} + \norm{\nabla \Phi}_{L^{2}(0,t;L^{2})}^{2} \right ) +  \frac{1}{4} \norm{\nabla \Xi_{n}}_{L^{2}(0,t;L^{2})}^{2} \\
& + C_{9} \left ( \norm{\Sigma_{n}}_{L^{2}(0,t;L^{2})}^{2} + \norm{w}_{L^{2}(0,t;L^{2})}^{2} + \norm{\Phi_{n}}_{L^{2}(0,t;L^{2})}^{2} \right ),
\end{align*}
where $C_{9} > 0$ is a positive constant depending only in $A$, $\PP$, $\alpha$, $C_{*}$, and $\AA$. Returning to \eqref{Linearized:apriori} we have
\begin{align*}
B \norm{\nabla \Phi_{n}(t)}_{L^{2}}^{2} & - C_{9} \norm{\nabla \Phi_{n}}_{L^{2}(0,t;L^{2})}^{2} + \frac{1}{4} \norm{\nabla \Xi_{n}}_{L^{2}(0,t;L^{2})}^{2} \\
& \leq C(C_{7},C_{8},C_{9}) \left ( \norm{\Sigma_{n}}_{L^{2}(0,t;L^{2})}^{2} + \norm{w}_{L^{2}(0,t;L^{2})}^{2} + \norm{\Phi_{n}}_{L^{2}(0,t;L^{2})}^{2} \right ).
\end{align*}
Applying the integral form of Gronwall's inequality and recalling \eqref{meanvalue:XiL2L2} and \eqref{Linearized:timederivative}, we find that
\begin{equation}\label{Linearized:est4}
\begin{aligned}
\norm{\Phi_{n}}_{L^{\infty}(0,T;H^{1})}^{2} + \norm{\Xi_{n}}_{L^{2}(0,T;H^{1})}^{2} + \norm{\pd_{t}\Phi_{n}}_{L^{2}(0,T;(H^{1})^{*})} \leq D_{4} \norm{w}_{L^{2}(Q)}^{2},
\end{aligned}
\end{equation}
where $D_{4}$ is a positive constant not depending on $n$, and so
\begin{align*}
\{\Xi_{n}\}_{n \in \N} & \text{ is bounded uniformly in } L^{2}(0,T;H^{1}), \\
\{\Phi_{n}\}_{n \in \N} & \text{ is bounded uniformly in } L^{\infty}(0,T;H^{1}) \cap H^{1}(0,T;(H^{1})^{*}).
\end{align*}
Furthermore, as $\Xi_{n} - A \Psi''(\overline{\varphi}) \Phi_{n} \in H^{1}$ for a.e. $t \in (0,T)$, applying elliptic regularity to \eqref{Linearized:mu} yields that
\begin{equation}\label{Linearized:est5}
\begin{aligned}
\norm{\Phi_{n}}_{L^{2}(0,T;H^{3})}^{2} \leq C \left ( \norm{\Xi_{n}}_{L^{2}(0,T;H^{1})}^{2} + \norm{\Phi_{n}}_{L^{2}(0,T;H^{1})}^{2} \right ) \leq D_{5} \norm{w}_{L^{2}(Q)}^{2},
\end{aligned}
\end{equation}
where $C$ and $D_{5}$ are positive constants not depending on $n$.  This implies that
\begin{align*}
\{\Phi_{n}\}_{n \in \N} \text{ is bounded uniformly in } L^{2}(0,T;H^{3}).
\end{align*}
The a priori estimates \eqref{Linearized:est1}, \eqref{Linearized:est2}, \eqref{Linearized:est3}, \eqref{Linearized:est4} and \eqref{Linearized:est5} imply that $(\Phi_{n}, \Xi_{n}, \Sigma_{n})$ can be extended to the interval $[0,T]$, and thus $t_{n} = T$ for each $n \in \N$.  Furthermore, there exists a relabelled subsequence such that
\begin{alignat*}{3}
\Phi_{n} & \to \Phi \text{ weakly*} && \text{ in } L^{\infty}(0,T;H^{1}) \cap L^{2}(0,T;H^{3}) \cap H^{1}(0,T;(H^{1})^{*}), \\
\Xi_{n} & \to \Xi \text{ weakly } && \text{ in }  L^{2}(0,T;H^{1}), \\
\Sigma_{n} & \to \Sigma \text{ weakly* } && \text{ in }  L^{\infty}(0,T;L^{2}) \cap L^{2}(0,T;H^{2}) \cap H^{1}(0,T;L^{2}),
\end{alignat*}
and a standard argument shows that the limit functions $(\Phi, \Xi, \Sigma)$ satisfy \eqref{weak:Linearized}.

%
%

\paragraph{Uniqueness.}
Let $(\Phi_{i}, \Xi_{i}, \Sigma_{i})_{i =1,2}$ denote two weak solution triplets to \eqref{eq:linearstate} with the same data $w \in L^{2}(Q)$.  Then, as \eqref{eq:linearstate} is linear in $(\Phi, \Xi, \Sigma)$, the differences $\Phi := \Phi_{1} - \Phi_{2}$, $\Xi := \Xi_{1} - \Xi_{2}$ and $\Sigma := \Sigma_{1} - \Sigma_{2}$ satisfy \eqref{eq:linearstate} with $w = 0$.  Due to the regularity of the solutions, the derivation of \eqref{Linearized:est1}, \eqref{Linearized:est2} and \eqref{Linearized:est3} remain valid, which implies that
\begin{align*}
\norm{\Phi}_{L^{\infty}(0,T;L^{2})}^{2} +  \norm{\Sigma}_{L^{\infty}(0,T;H^{1}) \cap H^{1}(0,T;L^{2}) \cap L^{2}(0,T;H^{2})}^{2} \leq 0,
\end{align*}
and so $\Phi = \Sigma = 0$.  Substituting $\Phi = 0$ in \eqref{weak:Linearized:mu} yields that $\Xi = 0$.

\subsection{Fr\'{e}chet differentiability with respect to the control}
In this section, we use the notation $\varphi^{w} = \hat{\varphi}$, $\mu^{w} = \hat{\mu}$, $\sigma^{w} = \hat{\sigma}$.  The remainders $(\theta^{w}, \rho^{w}, \xi^{w})$ from \eqref{Fdiff:remainders} satisfy
\begin{equation*}
\begin{aligned}
0 & = \inner{\pd_{t}\theta^{w}}{\zeta}_{H^{1}} + \int_{\Omega} \nabla \rho^{w} \cdot \nabla \zeta - h(\varphi^{w}) (\PP \sigma^{w} - \AA - \alpha (\overline{u} + w) ) \zeta \dx \\
 & + \int_{\Omega} h(\overline{\varphi}) (\PP (\overline{\sigma} + \Sigma^{w}) - \AA - \alpha (\overline{u} + w)) \zeta + h'(\overline{\varphi}) \Phi^{w} (\PP \overline{\sigma} - \AA - \alpha \overline{u}) \zeta \dx, \\
0 & = \int_{\Omega} \rho^{w} \zeta - B \nabla \theta^{w} \cdot \nabla \zeta - A \left ( \Psi'(\varphi^{w}) - \Psi'(\overline{\varphi}) - \Psi''(\overline{\varphi}) \Phi^{w} \right ) \zeta \dx, \\
0 & = \int_{\Omega} \pd_{t} \xi^{w} \zeta + \nabla \xi^{w} \cdot \nabla \zeta + \BB \xi^{w} \zeta + \CC \left ( h(\varphi^{w}) \sigma^{w} - h(\overline{\varphi}) (\overline{\sigma} + \Sigma^{w})- h'(\overline{\varphi}) \overline{\sigma} \Phi^{w} \right ) \zeta \dx,
\end{aligned}
\end{equation*}
for a.e. $t \in (0,T)$ and for all $\zeta \in H^{1}$ with
\begin{align*}
\theta^{w}(0) = 0, \quad \xi^{w}(0) = 0.
\end{align*}
Using the Taylor's theorem with integral remainder \eqref{Taylor} we see that
\begin{align*}
f(\varphi^{w}) = f(\overline{\varphi}) + f'(\overline{\varphi})(\varphi^{w} - \overline{\varphi}) + (\varphi^{w} - \overline{\varphi})^{2} \int_{0}^{1} f''(\overline{\varphi} + z (\varphi^{w} - \overline{\varphi})) \dz,
\end{align*}
and so for $\varphi^{w} - \overline{\varphi} = \Phi^{w} + \theta^{w}$, we have
\begin{align*}
\Psi'(\varphi^{w}) - \Psi'(\overline{\varphi}) - \Psi''(\overline{\varphi}) \Phi^{w} & = \Psi''(\overline{\varphi}) \theta^{w} + (\varphi^{w} - \overline{\varphi})^{2} R_{1}^{w}, \\
h(\varphi^{w}) - h(\overline{\varphi}) - h'(\overline{\varphi}) \Phi^{w} & = h'(\overline{\varphi}) \theta^{w} + (\varphi^{w} - \overline{\varphi})^{2} R_{2}^{w},
\end{align*}
where
\begin{align*}
R_{1}^{w} := \int_{0}^{1} \Psi'''(\overline{\varphi} + z (\varphi^{w} - \overline{\varphi})) (1-z) \dz, \quad R_{2}^{w} := \int_{0}^{1} h''(\overline{\varphi} + z (\varphi^{w} - \overline{\varphi})) (1-z) \dz.
\end{align*}
Thanks to the fact that $\overline{\varphi}, \varphi^{w} \in C^{0}(\overline{Q})$ and the continuity of $\Psi'''$ and $h''$, we see that there exists a constant $C_{**} > 0$ such that
\begin{align}\label{Psi'''h''bdd}
\norm{R_{1}^{w}}_{L^{\infty}(Q)} + \norm{R_{2}^{w}}_{L^{\infty}(Q)} \leq C_{**}.
\end{align}
Furthermore, we can express
\begin{equation}
\begin{aligned}
& h(\varphi^{w}) \sigma^{w} - h(\overline{\varphi}) \overline{\sigma} - h(\overline{\varphi}) \Sigma^{w} - h'(\overline{\varphi}) \Phi^{w} \overline{\sigma} \\
& \quad = (h(\varphi^{w}) - h(\overline{\varphi}))(\sigma^{w} - \overline{\sigma}) + \overline{\sigma} (h(\varphi^{w}) - h(\overline{\varphi}) - h'(\overline{\varphi}) \Phi^{w}) + h(\overline{\varphi}) (\sigma^{w} - \overline{\sigma} - \Sigma^{w}) \\
& \quad = (h(\varphi^{w}) - h(\overline{\varphi}))(\sigma^{w} - \overline{\sigma}) + \overline{\sigma} (h'(\overline{\varphi}) \theta^{w} + (\varphi^{w} - \overline{\varphi})^{2} R_{2}^{w}) + h(\overline{\varphi}) \xi^{w}.
\end{aligned}
\end{equation}
Let $X^{w} := \PP \sigma^{w} - \AA - \alpha (\overline{u} + w)$ and $\overline{X} := \PP \overline{\sigma} - \AA - \alpha \overline{u}$.  Then, it holds similarly that
\begin{equation}
\begin{aligned}
& h(\varphi^{w}) X^{w} - h(\overline{\varphi}) \overline{X} - h(\overline{\varphi})(\PP \Sigma^{w} - \alpha w) - h'(\overline{\varphi}) \Phi^{w} \overline{X} \\
& \quad = (h(\varphi^{w}) - h(\overline{\varphi}))(X^{w} - \overline{X}) + \overline{X}(h(\varphi^{w}) - h(\overline{\varphi}) - h'(\overline{\varphi}) \Phi^{w}) \\
& \quad + h(\overline{\varphi})(X^{w} - \overline{X} - \PP \Sigma + \alpha w) \\
& \quad = (h(\varphi^{w}) - h(\overline{\varphi}))(X^{w} - \overline{X}) + \overline{X}(h'(\overline{\varphi}) \theta^{w} + R_{2}^{w}(\varphi^{w} - \overline{\varphi})^{2}) + h(\overline{\varphi}) \PP \xi^{w},
\end{aligned}
\end{equation}
and thus, we see that $(\theta^{w}, \rho^{w}, \xi^{w})$ satisfy
\begin{subequations}
\begin{align}
0 & = \inner{\pd_{t}\theta^{w}}{\zeta}_{H^{1}} + \int_{\Omega} \nabla \rho^{w} \cdot \nabla \zeta -  (h(\varphi^{w}) - h(\overline{\varphi}))(X^{w} - \overline{X}) \zeta \dx \label{Fdiff:varphi} \\
\notag & + \int_{\Omega} (\overline{X} (h'(\overline{\varphi}) \theta^{w} + (\varphi^{w} - \overline{\varphi})^{2}R_{2}^{w}) + h(\overline{\varphi}) \PP \xi^{w}) \zeta \dx, \\
0 & = \int_{\Omega} \rho^{w} \zeta - B \nabla \theta^{w} \cdot \nabla \zeta - A (\Psi''(\overline{\varphi}) \theta^{w} + (\varphi^{w} - \overline{\varphi})^{2} R_{1}^{w}) \zeta \dx, \label{Fdiff:mu} \\
0 & = \int_{\Omega} \pd_{t} \xi^{w} \zeta + \nabla \xi^{w} \cdot \nabla \zeta + \BB \xi^{w} \zeta + \CC (h(\varphi^{w}) - h(\overline{\varphi}))(\sigma^{w} - \overline{\sigma}) \zeta \dx \label{Fdiff:sigma} \\
\notag & + \int_{\Omega} \CC (\overline{\sigma} (h'(\overline{\varphi}) \theta^{w} + (\varphi^{w} - \overline{\varphi})^{2} R_{2}^{w}) + h(\overline{\varphi}) \xi^{w}) \zeta \dx,
\end{align}
\end{subequations}
for a.e. $t \in (0,T)$ and for all $\zeta \in H^{1}$.  

\paragraph{First estimate.}  Let us first compute the following preliminary estimates, using the continuous dependence estimate \eqref{solution:ctsdep:s}, the Lipschitz continuity of $h$, H\"{o}lder's inequality, Young's inequality and the embedding $L^{2}(0,T;H^{2}) \subset L^{2}(0,T;L^{\infty})$, we have that
\begin{align*}
& \int_{0}^{s} \int_{\Omega} \CC \abs{h(\varphi^{w}) - h(\overline{\varphi})} \abs{\sigma^{w} - \overline{\sigma}} \abs{\xi^{w}} \dx \dt \leq \CC L_{h} \int_{0}^{s} \norm{\xi^{w}}_{L^{2}} \norm{\sigma^{w} - \overline{\sigma}}_{L^{2}} \norm{\varphi^{w} - \overline{\varphi}}_{L^{\infty}} \dt \\
& \quad \leq \CC L_{h}C_{\mathrm{Sob}} \norm{\xi^{w}}_{L^{2}(0,s;L^{2})} \norm{\sigma^{w} - \overline{\sigma}}_{L^{\infty}(0,s;L^{2})} \norm{\varphi^{w} - \overline{\varphi}}_{L^{2}(0,s;H^{2})} \\
& \quad \leq C_{10} \norm{w}_{L^{2}(0,s;L^{2})}^{4} + \frac{1}{4}\norm{\xi^{w}}_{L^{2}(0,s;L^{2})},
\end{align*}
where $C_{10}$ is a positive constant depending only on $C_{\mathrm{cts}}$, $C_{\mathrm{Sob}}$ and $L_{h}$.  Meanwhile, using the boundedness of $\overline{\sigma}$, $h'(\overline{\varphi})$ and $R_{2}^{w}$ in $Q$, we see that
\begin{align*}
& \int_{0}^{s} \int_{\Omega} \CC \abs{\overline{\sigma}} \abs{ h'(\overline{\varphi}) \theta^{w} \xi^{w} + (\varphi^{w} - \overline{\varphi})^{2} R_{2}^{w} \xi^{w}} \dx \dt \\
& \quad \leq \CC \int_{0}^{s} C_{*} \norm{\theta^{w}}_{L^{2}} \norm{\xi^{w}}_{L^{2}} + C_{**} \norm{\varphi^{w} - \overline{\varphi}}_{L^{\infty}} \norm{\varphi^{w} - \overline{\varphi}}_{L^{2}} \norm{\xi^{w}}_{L^{2}} \dt \\
& \quad \leq 2\CC^{2}C_{*}^{2} \norm{\theta^{w}}_{L^{2}(0,s;L^{2})}^{2} + \frac{2}{8}\norm{\xi^{w}}_{L^{2}(0,s;L^{2})}^{2} + 2 \CC^{2} C_{**}^{2} \norm{\varphi^{w} - \overline{\varphi}}_{L^{\infty}(0,s;L^{2})}^{2} \norm{\varphi^{w} - \overline{\varphi}}_{L^{2}(0,s;L^{\infty})}^{2} \\
& \quad \leq 2 \CC^{2} C_{*}^{2} \norm{\theta^{w}}_{L^{2}(0,s;L^{2})}^{2} + \frac{1}{4}\norm{\xi^{w}}_{L^{2}(0,s;L^{2})}^{2} + 2C_{**}^{2} \CC^{2} C_{\mathrm{cts}}^{2} \norm{w}_{L^{2}(0,s;L^{2})}^{4}.
\end{align*}
Thus, when we substitute $\zeta = \xi^{w}$ in \eqref{Fdiff:sigma}, integrating over $[0,s]$ for $s \in (0,T]$, and neglecting the nonnegative term $\BB \abs{\xi^{w}}^{2} + \CC h(\overline{\varphi}) \abs{\xi^{w}}^{2}$, we obtain
\begin{equation}\label{Fdiff:est1}
\begin{aligned}
& \frac{1}{2} \norm{\xi^{w}(s)}_{L^{2}}^{2} + \norm{\nabla \xi^{w}}_{L^{2}(0,s;L^{2})}^{2} \\
& \quad \leq  \CC\int_{0}^{s} \int_{\Omega} (h(\overline{\varphi}) - h(\varphi^{w}))(\sigma^{w} - \overline{\sigma}) \xi^{w} + \overline{\sigma} (h'(\overline{\varphi}) \theta^{w} \xi^{w} + (\varphi^{w} - \overline{\varphi}^{2}) R_{2}^{w} \xi^{w}) \dx \dt \\
& \quad \leq (C_{10} + 2 C_{**}^{2} \CC^{2} C_{\mathrm{cts}}^{2}) \norm{w}_{L^{2}(0,s;L^{2})}^{4} +  2\CC^{2} C_{*}^{2} \norm{\theta^{w}}_{L^{2}(0,s;L^{2})}^{2} + \frac{1}{2} \norm{\xi^{w}}_{L^{2}(0,s;L^{2})}^{2}.
\end{aligned}
\end{equation}
Next, substituting $\zeta = \theta^{w}$ in \eqref{Fdiff:varphi}, $\zeta = \theta^{w}$ in \eqref{Fdiff:mu} and $\zeta = \frac{1}{B} \rho^{w}$ in \eqref{Fdiff:mu}, integrating by parts and integrating over $[0,s]$ for $s \in (0,T]$, and upon adding leads to
\begin{equation}\label{Fdiff:est2:prelim}
\begin{aligned}
& \frac{1}{2} \norm{\theta^{w}(s)}_{L^{2}}^{2} + B \norm{\nabla \theta^{w}}_{L^{2}(0,s;L^{2})}^{2} + \frac{1}{B} \norm{\rho^{w}}_{L^{2}(0,s;L^{2})}^{2} \\
& \quad \leq \int_{0}^{s} \int_{\Omega} \abs{(\Psi''(\overline{\varphi}) \theta^{w} + (\varphi^{w} - \overline{\varphi})^{2} R_{1}^{w})} \abs{A \theta^{w} + \frac{A}{B} \rho^{w}} + \abs{\rho^{w} \theta^{w}} \dx \dt \\
& \quad + \int_{0}^{s} \int_{\Omega} \abs{(h(\varphi^{w}) - h(\overline{\varphi}))(X^{w} - \overline{X}) \theta^{w}} \dx \dt \\
& \quad + \int_{0}^{s} \int_{\Omega} \abs{(\overline{X}(h'(\overline{\varphi}) \theta^{w} + (\varphi^{w} - \overline{\varphi})^{2} R_{2}^{w}) + h(\overline{\varphi}) \PP \xi^{w})} \abs{\theta^{w}} \dx \dt =: K_{1} + K_{2} + K_{3}.
\end{aligned}
\end{equation}
Using \eqref{solution:ctsdep:s}, H\"{o}lder's inequality, Young's inequality, the boundedness of $\Psi''(\overline{\varphi})$ and $R_{1}^{w}$ in $Q$, we have
\begin{align*}
K_{1} & \leq \frac{1}{4B} \norm{\rho^{w}}_{L^{2}(0,s;L^{2})}^{2} + B \norm{\theta^{w}}_{L^{2}(0,s;L^{2})}^{2} \\ & + \left ( C_{*} \norm{\theta^{w}}_{L^{2}(0,s;L^{2})} + C_{**} C_{\mathrm{cts}}^{2}\norm{w}_{L^{2}(0,s;L^{2})}^{2} \right ) \left (A \norm{\theta^{w}}_{L^{2}(0,s;L^{2})} + \frac{A}{B} \norm{\rho^{w}}_{L^{2}(0,s;L^{2})} \right ) \\
& \leq \frac{1}{2B} \norm{\rho^{w}}_{L^{2}(0,s;L^{2})}^{2} + C_{11} \left ( \norm{\theta^{w}}_{L^{2}(0,s;L^{2})}^{2} + \norm{w}_{L^{2}(0,s;L^{2})}^{4} \right ),
\end{align*}
where $C_{11}$ is a positive constant depending only on $C_{*}$, $C_{**}$, $C_{\mathrm{cts}}$, $A$ and $B$.  Meanwhile, by the Lipschitz continuity of $h$, and the fact that
\begin{align*}
X^{w} - \overline{X} = \PP (\sigma^{w} - \overline{\sigma}) - \alpha w,
\end{align*}
we see that
\begin{align*}
K_{2} & \leq \int_{0}^{s} \int_{\Omega} L_{h} \abs{\varphi^{w} - \overline{\varphi}} \abs{\PP (\sigma^{w} - \overline{\sigma}) - \alpha w} \abs{\theta^{w}} \dx \dt \\
& \leq L_{h} \PP \norm{\varphi^{w} - \overline{\varphi}}_{L^{2}(0,s;L^{\infty})} \norm{\sigma^{w} - \overline{\sigma}}_{L^{\infty}(0,s;L^{2})} \norm{\theta^{w}}_{L^{2}(0,s;L^{2})} \\
& + L_{h} \alpha \norm{w}_{L^{2}(0,s;L^{2})} \norm{\varphi^{w} - \overline{\varphi}}_{L^{\infty}(0,s;L^{3})} \norm{\theta^{w}}_{L^{2}(0,s;L^{6})} \\
& \leq C_{12} \norm{w}_{L^{2}(0,s;L^{2})}^{4} + \frac{B}{2} \left ( \norm{\theta^{w}}_{L^{2}(0,s;L^{2})}^{2} + \norm{\nabla \theta^{w}}_{L^{2}(0,s;L^{2})}^{2} \right ),
\end{align*}
for some positive constant $C_{12}$ depending only on $B$, $L_{h}$, $\PP$ and $\alpha$.  Furthermore, using the boundedness of $\overline{X}$, $h'(\overline{\varphi})$, $R_{2}^{w}$ and $h(\overline{\varphi})$ in $Q$, we have
\begin{align*}
K_{3} & \leq (\PP + \AA + \alpha) C_{*} \norm{\theta^{w}}_{L^{2}(0,s;L^{2})}^{2} + \PP \norm{\xi^{w}}_{L^{2}(0,s;L^{2})} \norm{\theta^{w}}_{L^{2}(0,s;L^{2})} \\
& + (\PP + \AA + \alpha) C_{**} \norm{\varphi^{w} - \overline{\varphi}}_{L^{2}(0,s;L^{\infty})} \norm{\varphi^{w} - \overline{\varphi}}_{L^{\infty}(0,s;L^{2})} \norm{\theta^{w}}_{L^{2}(0,s;L^{2})} \\
& \leq \left (C_{u}C_{*} + \frac{1}{2} \right ) \norm{\theta^{w}}_{L^{2}(0,s;L^{2})}^{2} + \PP^{2} \norm{\xi^{w}}_{L^{2}(0,s;L^{2})}^{2} + C_{u}^{2} C_{**}^{2} C_{\mathrm{cts}}^{2} \norm{w}_{L^{2}(0,s;L^{2})}^{4},
\end{align*}
where we recall $C_{u} = (\PP + \AA + \alpha)$.  Substituting the above estimates into \eqref{Fdiff:est2:prelim} we obtain
\begin{equation}\label{Fdiff:est2}
\begin{aligned}
& \frac{1}{2} \norm{\theta^{w}(s)}_{L^{2}}^{2} + \frac{B}{2} \norm{\nabla \theta^{w}}_{L^{2}(0,s;L^{2})}^{2} + \frac{1}{2B} \norm{\rho^{w}}_{L^{2}(0,s;L^{2})}^{2} \\
& \quad  \leq \left (C_{11} + C_{u} C_{*} + \frac{B+1}{2} \right ) \norm{\theta^{w}}_{L^{2}(0,s;L^{2})}^{2} + (C_{11} + C_{12} + C_{u}^{2} C_{**}^{2} C_{\mathrm{cts}}^{2} ) \norm{w}_{L^{2}(0,s;L^{2})}^{4} \\
& \quad + \PP^{2} \norm{\xi^{w}}_{L^{2}(0,s;L^{2})}^{2}.
\end{aligned}
\end{equation}
Then, adding \eqref{Fdiff:est1} and \eqref{Fdiff:est2} we have for $s \in (0,T]$,
\begin{equation}\label{Fdiff:est3}
\begin{aligned}
& \norm{\xi^{w}(s)}_{L^{2}}^{2} + \norm{\nabla \xi^{w}}_{L^{2}(0,s;L^{2})}^{2} + \norm{\theta^{w}(s)}_{L^{2}}^{2} + \norm{\nabla \theta^{w}}_{L^{2}(0,s;L^{2})}^{2} + \norm{\rho^{w}}_{L^{2}(0,s;L^{2})}^{2} \\
& \quad  \leq C_{13} \norm{w}_{L^{2}(0,s;L^{2})}^{4} + C_{14} \left ( \norm{\theta^{w}}_{L^{2}(0,s;L^{2})}^{2} + \norm{\xi^{w}}_{L^{2}(0,s;L^{2})}^{2} \right ),
\end{aligned}
\end{equation}
where the positive constants $C_{13}$, $C_{14}$ depend only on $\CC$, $C_{\mathrm{cts}}$, $C_{*}$, $C_{**}$, $C_{10}$, $C_{11}$, $C_{12}$, $\PP$, $\AA$, $\alpha$, and $B$.  Applying Gronwall's inequality to \eqref{Fdiff:est3} we have that
\begin{equation}\label{Fdiff:1}
\begin{aligned}
& \norm{\xi^{w}(s)}_{L^{2}}^{2} + \norm{\theta^{w}(s)}_{L^{2}}^{2} \\
& \quad + \norm{\nabla \theta^{w}}_{L^{2}(0,s;L^{2})}^{2} + \norm{\nabla \xi^{w}}_{L^{2}(0,s;L^{2})}^{2} + \norm{\rho^{w}}_{L^{2}(0,s;L^{2})}^{2} \leq C_{15} \norm{w}_{L^{2}(0,s;L^{2})}^{4},
\end{aligned}
\end{equation}
for some positive constant $C_{15}$ depending only on $C_{13}$ and $C_{14}$.

\paragraph{Second estimate.}
Substituting $\zeta = \pd_{t} \xi^{w}$ in \eqref{Fdiff:sigma}, integrating over $[0,s]$ for $s \in (0,T]$ leads to
\begin{align*}
& \norm{\pd_{t} \xi^{w}}_{L^{2}(0,s;L^{2})}^{2} + \norm{\nabla \xi^{w}(s)}_{L^{2}}^{2} + \BB \norm{\xi^{w}(s)}_{L^{2}}^{2}  \leq \int_{0}^{s} \int_{\Omega} \CC \abs{h(\varphi^{w}) - h(\overline{\varphi})}\abs{\sigma^{w} - \overline{\sigma}} \abs{\pd_{t}\xi^{w}} \dx \dt \\
& \quad + \int_{0}^{s} \int_{\Omega} \CC  \abs{\overline{\sigma}(h'(\overline{\varphi}) \theta^{w} + (\varphi^{w} - \overline{\varphi})^{2} R_{2}^{w}) + h(\overline{\varphi}) \xi^{w}} \abs{\pd_{t} \xi^{w}} \dx \dt.
\end{align*}
Using the Lipschitz continuity of $h$, the boundedness of $\overline{\sigma}$, $h(\overline{\varphi})$, $h'(\overline{\varphi})$, and $R_{2}^{w}$ in $Q$, H\"{o}lder's inequality, Young's inequality and \eqref{Fdiff:1}, we obtain
\begin{equation}\label{Fdiff:2}
\begin{aligned}
& \frac{1}{2} \norm{\pd_{t}\xi^{w}}_{L^{2}(0,s;L^{2})}^{2} + \norm{\nabla \xi^{w}(s)}_{L^{2}}^{2} \\
& \quad \leq (\CC^{2} C_{\mathrm{cts}}^{2} L_{h}^{2} + C_{**}^{2} \CC^{2} C_{\mathrm{cts}}^{2}) \norm{w}_{L^{2}(0,s;L^{2})}^{4} + \CC^{2} C_{*}^{2} \norm{\theta^{w}}_{L^{2}(0,s;L^{2})}^{2} + \norm{\xi^{w}}_{L^{2}(0,s;L^{2})}^{2} \\
& \quad \leq C_{16} \norm{w}_{L^{2}(0,s;L^{2})}^{4},
\end{aligned}
\end{equation}
for some positive constant $C_{16}$ depending only on $T$, $L_{h}$, $\CC$, $C_{\mathrm{cts}}$, $C_{*}$, $C_{**}$ and $C_{15}$.

\paragraph{Third estimate.}
Viewing \eqref{Fdiff:mu} as the weak formulation of an elliptic problem for $\theta^{w}$, by elliptic regularity we obtain
\begin{align*}
\norm{\theta^{w}}_{L^{2}(0,s;H^{2})}^{2} & \leq C_{17} \left ( \norm{\rho^{w}}_{L^{2}((0,s;L^{2})}^{2} + \norm{\theta^{w}}_{L^{2}(0,s;L^{2})}^{2} \right ) \\
& + C_{17} \norm{A (\Psi''(\overline{\varphi}) \theta^{w} + (\varphi^{w} - \overline{\varphi})^{2} R_{1}^{w})}_{L^{2}(0,s;L^{2})}^{2},
\end{align*}
for some positive constant $C_{17}$ not depending on $\theta^{w}$, $\rho^{w}$ and $w$.  Applying \eqref{Fdiff:1}, the boundedness of $\Psi''(\overline{\varphi})$ and $R_{1}^{w}$ in $Q$, we have
\begin{align*}
\norm{\theta^{w}}_{L^{2}(0,s;H^{2})}^{2} \leq C_{18}\norm{w}_{L^{2}(0,s;L^{2})}^{4},
\end{align*}
for some positive constant $C_{18}$ depending only on $C_{15}$, $C_{17}$, $C_{*}$, $C_{**}$, $C_{\mathrm{cts}}$ and $A$.  Then, upon integrating \eqref{Fdiff:varphi} over $[0,s]$ for $s \in (0,T]$, integrating by parts then yields
\begin{align*}
\int_{0}^{s} \abs{\inner{\pd_{t}\theta^{w}}{\zeta}_{H^{1}}} \dt & \leq \int_{0}^{s} \int_{\Omega} \abs{\rho^{w}} \abs{\Laplace \zeta} + L_{h} \abs{\varphi^{w} - \overline{\varphi}} \abs{\PP (\sigma^{w} - \overline{\sigma}) + \alpha w} \abs{\zeta} \dx \dt \\
& + \int_{0}^{s} \int_{\Omega} \abs{\overline{X} (h'(\overline{\varphi}) \theta^{w} + (\varphi^{w} - \overline{\varphi})^{2} R_{2}^{w}) + h(\overline{\varphi}) \PP \xi^{w}} \abs{\zeta} \dx \dt \\
& =: L_{1} + L_{2}.
\end{align*}
By H\"{o}lder's inequality, the boundedness of $\overline{X} = \PP \overline{\sigma} - \AA - \alpha \overline{u}$, $h'(\overline{\varphi})$, $R_{2}^{w}$, and $h(\overline{\varphi})$ in $Q$, \eqref{solution:ctsdep:s}, and \eqref{Fdiff:1} we have that
\begin{align*}
L_{2} & \leq C_{u} \left ( C_{*} \norm{\theta^{w}}_{L^{2}(0,s;L^{2})} + C_{**} \norm{\varphi^{w} - \overline{\varphi}}_{L^{\infty}(0,s;L^{2})} \norm{\varphi^{w} - \overline{\varphi}}_{L^{2}(0,s;L^{\infty})} \right ) \norm{\zeta}_{L^{2}(0,s;L^{2})} \\
& + \PP \norm{\xi^{w}}_{L^{2}(0,s;L^{2})}\norm{\zeta}_{L^{2}(0,s;L^{2})} \\
& \leq C_{19} \norm{w}_{L^{2}(0,s;L^{2})}^{2} \norm{\zeta}_{L^{2}(0,s;L^{2})}
\end{align*}
for some positive constant $C_{19}$ depending only on $\PP$, $\AA$, $\alpha$, $C_{*}$, $C_{**}$, $C_{\mathrm{cts}}$, $T$ and $C_{15}$.  Meanwhile,
\begin{align*}
L_{1} & \leq \norm{\rho^{w}}_{L^{2}(0,s;L^{2})} \norm{\zeta}_{L^{2}(0,s;H^{2})} + L_{h} \PP \norm{\varphi^{w} - \overline{\varphi}}_{L^{2}(0,s;L^{\infty})} \norm{\sigma^{w} - \overline{\sigma}}_{L^{\infty}(0,s;L^{2})} \norm{\zeta}_{L^{2}(0,s;L^{2})} \\
& +  L_{h} \alpha \norm{w}_{L^{2}(0,s;L^{2})} \norm{\varphi^{w} - \overline{\varphi}}_{L^{\infty}(0,s;L^{3})} \norm{\zeta}_{L^{2}(0,s;L^{6})} \\
& \leq C_{20} \norm{w}_{L^{2}(0,s;L^{2})}^{2} \norm{\zeta}_{L^{2}(0,s;H^{2})}
\end{align*}
where $C_{20}$ is a positive constant depending only on $C_{15}$, $L_{h}$, $\PP$, $\alpha$, $C_{\mathrm{cts}}$ and $\Omega$ (via the Sobolev embedding $H^{1} \subset L^{6}$).  Hence, we see that
\begin{align*}
\norm{\pd_{t} \theta^{w}}_{L^{2}(0,s;(H^{2})^{*})} \leq (C_{19} + C_{20}) \norm{w}_{L^{2}(0,s;L^{2})}^{2}.
\end{align*}
By the continuous embedding $L^{2}(0,T;H^{2}) \cap H^{1}(0,T;(H^{2})^{*}) \subset C^{0}([0,T];L^{2})$, we find that there exists a positive constant $C_{21}$ depending only on $C_{18}$, $C_{19}$ and $C_{20}$ such that
\begin{align*}
\norm{\theta^{w}}_{L^{2}(0,s;H^{2}) \cap H^{1}(0,s;(H^{2})^{*}) \cap C^{0}([0,s];L^{2})} \leq C_{21} \norm{w}_{L^{2}(0,s;L^{2})}^{2} \quad \forall s \in (0,T].
\end{align*}
Combining this with \eqref{Fdiff:1} and \eqref{Fdiff:2} yields \eqref{Fdiff:u:thm:estimate}.

\subsection{Fr\'{e}chet differentiability of the objective functional with respect to time}\label{sec:Fdiff:wrt:tau}

In this section, we assume that Assumption \ref{assump:regularityU} holds.  Using the relation 
\begin{align}\label{timechange}
\int_{\tau-r}^{\tau} \int_{\Omega} f(s) \dx \ds = \int_{0}^{\tau} \int_{\Omega} f(s) - f(s-r) \dx \ds + \int_{-r}^{0} \int_{\Omega} f(s) \ds
\end{align}  
for $f \in L^{1}(-r,T;L^{1})$ and $\tau \in (0,T)$, we can define
\begin{align*}
F(t,\varphi)& := \frac{1}{2} \int_{\Omega} \beta_{Q} \abs{(\varphi - \varphi_{Q})(t)}^{2} + \frac{\beta_{\Omega}}{r} \left ( \abs{(\varphi - \varphi_{\Omega})(t)}^{2} - \abs{(\varphi - \varphi_{\Omega})(t-r)}^{2} \right ) \dx \\
& + \frac{1}{2} \int_{\Omega} \frac{\beta_{S}}{r} \left ( \varphi(t) - \varphi(t-r) \right ) \dx,
\end{align*} 
and upon setting $\varphi(t) = \varphi_{0}$ for $t \leq 0$, we can express \eqref{eq:cost} as
\begin{align*}
J_{r}(\varphi, u, \tau) &= \frac{\beta_{u}}{2} \norm{u}_{L^{2}(Q)}^{2} + \int_{-r}^{0} \int_{\Omega} \frac{\beta_{\Omega}}{2r} \abs{\varphi_{0} - \varphi_{\Omega}}^{2} + \frac{\beta_{S}}{2r} (1 + \varphi_{0}) \dx \dt \\
& +  \int_{0}^{\tau} F(t,\varphi) \dt + \beta_{T} \tau.
\end{align*}
Note that only the last two terms on the right-hand side are dependent on $\tau$, and thus the first three terms on the right-hand side will vanish when we compute the Fr\'{e}chet derivative of $J_{r}$ with respect to $\tau$.  We now compute for any $f \in H^{1}(0,T) \subset L^{\infty}(0,T)$, and $\tau \in (0,T)$, $h > 0$ such that $\tau + h \in (0,T)$,
\begin{align*}
& \abs{\int_{0}^{\tau+h} \abs{f(t)}^{2} \dt - \int_{0}^{\tau} \abs{f(t)}^{2} \dt - h \abs{f(\tau)}^{2}} = \abs{\int_{\tau}^{\tau+h} \abs{f(t)}^{2} - \abs{f(\tau)}^{2} \dt} \\
& \quad \leq \abs{\int_{\tau}^{\tau+h} \abs{f(t) - f(\tau)} \abs{f(t) + f(\tau)} \dt} \leq 2 \norm{f}_{L^{\infty}(0,T)} \abs{\int_{\tau}^{\tau+h} \abs{\int_{\tau}^{t} \pd_{t}f(s) \ds} \dt} \\
& \quad \leq 2 \norm{f}_{L^{\infty}(0,T)} \int_{\tau}^{\tau+h} \norm{\pd_{t}f}_{L^{2}(\tau, t)} (t-\tau)^{\frac{1}{2}} \dt \leq 2 h^{\frac{3}{2}} \norm{f}_{L^{\infty}(0,T)} \norm{\pd_{t}f}_{L^{2}(0,T)}.
\end{align*}
This shows that 
\begin{align*}
\der_{\tau} \left ( \int_{0}^{\tau} \abs{f(t)}^{2} \dt \right ) = \abs{f(\tau)}^{2},
\end{align*}
and a similar argument also yields
\begin{align*}
\der_{\tau} \left ( \int_{0}^{\tau} f(t) \dt \right ) = f(\tau).
\end{align*}
Using the fact that $\varphi_{Q} \in H^{1}(0,T;L^{2})$, $\varphi_{*}, \varphi_{\Omega} \in H^{1}(-r,T;L^{2})$, we obtain that the optimal control $(u_{*}, \Optime)$ satisfies
\begin{align}\label{FONC:time:variational}
\der_{\tau} \mathcal{J}(u_{*}, \tau_{*})(s - \tau_{*}) \geq 0 \quad \forall s \in [0,T],
\end{align}
where
\begin{align*}
\der_{\tau} \mathcal{J}( u_{*}, \Optime) & = \beta_{T} + \frac{\beta_{Q}}{2} \norm{\varphi_{*}(\Optime) - \varphi_{Q}(\Optime)}_{L^{2}}^{2} + \frac{\beta_{S}}{2r} \int_{\Omega} \varphi_{*}(\Optime) - \varphi_{*}(\Optime - r) \dx \\
& + \frac{\beta_{\Omega}}{2r} \left ( \norm{(\varphi_{*} - \varphi_{\Omega})(\Optime)}_{L^{2}}^{2} - \norm{(\varphi_{*} - \varphi_{\Omega})(\Optime - r)}_{L^{2}}^{2} \right ).
\end{align*}
We can simplify \eqref{FONC:time:variational} with the following argument.  If $\Optime \in (0,T)$, choose $s = \Optime \pm h$ for $h > 0$ to deduce that $\der_{\tau} \mathcal{J}(u_{*}, \Optime) = 0$.  If $\Optime = 0$, then from \eqref{FONC:time:variational} we obtain $\der_{\tau} \mathcal{J}(u_{*}, \Optime) \geq 0$.  Meanwhile, if $\Optime = T$, then $s - \Optime \leq 0$ for any $s \in [0,T]$, and thus $\der_{\tau} \mathcal{J}(u_{*}, \Optime) \leq 0$.

\section{First order necessary optimality conditions}\label{sec:FONC}

\subsection{Unique solvability of the adjoint system}
We apply a Galerkin approximation and consider a basis $\{w_{i}\}_{i \in \N}$ of $H^{2}$ that is orthonormal in $L^{2}$, and we look for functions of the form
\begin{align*}
p_{n}(x,t) := \sum_{i=1}^{n} P_{n,i}(t) w_{i}(x), \quad q_{n}(x,t) := \sum_{i=1}^{n} Q_{n,i}(t)w_{i}(x), \quad r_{n}(x,t) := \sum_{i=1}^{n} R_{n,i}(t) w_{i}(x),
\end{align*}
which satisfy
\begin{subequations}\label{adjoint:galerkin}
\begin{align}
\label{adjoint:galerkin:p} 0 & = \int_{\Omega} - \pd_{t}p_{n} v  - B \nabla q_{n} \cdot \nabla v - A \Psi''(\varphi) q_{n} v +  h'(\varphi) \left ( \CC\sigma r_{n} - (\PP \sigma - \AA - \alpha u) p_{n}v \right ) \dx \\
\notag & - \int_{\Omega} \left ( \beta_{Q}(\varphi - \varphi_{Q}) + \tfrac{1}{2r} \chi_{(\Optime - r, \Optime)}(t) \left ( 2 \beta_{\Omega} (\varphi - \varphi_{\Omega}) + \beta_{S} \right ) \right ) v  \dx, \\
\label{adjoint:galerkin:q} 0 & = \int_{\Omega} q_{n}v + \nabla p_{n} \cdot \nabla v \dx, \\
\label{adjoint:galerkin:r} 0 & = \int_{\Omega} - \pd_{t}r_{n}v + \nabla r_{n} \cdot \nabla v + (\BB + \CC h(\varphi)) r_{n}v - \PP h(\varphi) p_{n} v \dx,
\end{align}
\end{subequations}
for all $v \in W_{n} := \mathrm{span}\{w_{1}, \dots, w_{n}\}$.  Substituting $v = w_{j}$ leads to
\begin{subequations}\label{adjoint:ODE}
\begin{align}
 \bm{P}_{n}'(t) & = -B \bm{S} \bm{Q}_{n}(t) - \bm{Z}_{n}(t) - \chi_{(\Optime - r, \Optime)}(t) \bm{G}_{n}(t), \quad \bm{Q}_{n}(t) = - \bm{S} \bm{P}_{n}(t), \\
\bm{R}_{n}'(t) & = \bm{S} \bm{R}_{n}(t) + \BB \bm{R}_{n}(t) + \bm{Y}_{n}(t),
\end{align}
\end{subequations}
where $\bm{S}$ is defined in \eqref{ODE:vectormatrix}, and
\begin{align*}
(\bm{G}_{n})_{j} & := \int_{\Omega} \frac{1}{2r} \left ( 2 \beta_{\Omega}(\varphi - \varphi_{\Omega}) + \beta_{S}) \right ) w_{j} \dx, \\
(\bm{Z}_{n})_{j} & := \int_{\Omega} \left ( A \Psi''(\varphi) q_{n} -  h'(\varphi) \left ( \CC \sigma r_{n}  - (\PP \sigma - \AA - \alpha u) p_{n} \right ) + \beta_{Q}(\varphi - \varphi_{Q}) \right )w_{j} \dx, \\
(\bm{Y}_{n})_{j} & := \int_{\Omega} \CC h(\varphi) r_{n} w_{j} - \PP h(\varphi) p_{n} w_{j} \dx,
\end{align*}
and we supplement the above backward-in-time system of ODEs with the condition
\begin{align*}
r_{n}(\Optime) = 0, \quad p_{n}(\Optime) = 0.
\end{align*}
Once again, we consider approximating sequences in $C^{0}([0,T];L^{2})$ for $u$, $\varphi_{Q}$ and $\varphi_{\Omega}$ and use the same variables to denote the approximating functions.  Note that the right-hand side of \eqref{adjoint:ODE} depends continuously on $(\bm{P}_{n}, \bm{Q}_{n}, \bm{R}_{n})$ but due to the term $\chi_{(\Optime-r,\Optime)}(t) \bm{G}_{n}$ in the equation for $\bm{P}_{n}'$, we cannot apply the Cauchy--Peano theorem directly.  But we can consider first solving \eqref{adjoint:ODE} on the interval $(\Optime-r,\Optime]$, that is, $\bm{P}_{n}$ and $\bm{R}_{n}$ satisfy
\begin{equation}\label{adjoint:ODE:part1}
\begin{alignedat}{4}
\bm{P}_{n}'(t) & = B \bm{S}^{2} \bm{P}_{n}(t) - \bm{Z}_{n}(t) - \bm{G}_{n}(t), && \quad \bm{P}_{n}(\Optime) = 0, \\
\bm{R}_{n}'(t) & = \bm{S} \bm{R}_{n}(t) + \BB \bm{R}_{n}(t) + \bm{Y}_{n}(t), && \quad \bm{R}_{n}(\Optime) = 0,
\end{alignedat}
\end{equation}
for $t \in (\Optime-r,\Optime]$, which would yield, via the Cauchy--Peano theorem, the existence of $t_{n} \in [\Optime-r,\Optime)$ and a local solution pair $(\bm{P}_{n}, \bm{R}_{n}) \in \left ( C^{1}((t_{n}, \Optime]; \R^{n}) \right )^{2}$ to \eqref{adjoint:ODE:part1}.  The a priori estimates derived below will allow us to deduce that $(\bm{P}_{n}, \bm{R}_{n})$ can be extended to $\Optime-r$, that is, $t_{n} = \Optime - r$ for all $n \in \N$.  Then, we then extend the solutions by solving the system
\begin{equation}\label{adjoint:ODE:part2}
\begin{alignedat}{4}
\bm{P}_{n}'(t) & = B \bm{S}^{2} \bm{P}_{n}(t) - \bm{Z}_{n}(t), \\
\bm{R}_{n}'(t) & = \bm{S} \bm{R}_{n}(t) + \BB \bm{R}_{n}(t) + \bm{Y}_{n}(t),
\end{alignedat}
\end{equation}
with terminal conditions at time $\Optime-r$.  Overall, this procedure yields functions $p_{n}, q_{n}, r_{n} \in C^{1}((t_{n}, \Optime];W_{n})$ satisfying \eqref{adjoint:galerkin} for some $t_{n} \in [0,\Optime)$.  We now derive the a priori estimates.

\paragraph{First estimate.}
Substituting $v = r_{n}$ in \eqref{adjoint:galerkin:r} and integrating over $[s, \Optime]$ for $s \in (0,\Optime)$ leads to
\begin{align}\label{adjoint:est:1}
\frac{1}{2} \norm{r_{n}(s)}_{L^{2}}^{2} + \norm{\nabla r_{n}}_{L^{2}(s,\Optime;L^{2})}^{2} \leq \PP \norm{p_{n}}_{L^{2}(s,\Optime;L^{2})} \norm{r_{n}}_{L^{2}(s,\Optime;L^{2})},
\end{align}
where we neglected the nonnegative term $\BB \abs{r_{n}}^{2} + \CC h(\varphi) \abs{r_{n}}^{2}$ and used the boundedness of $h$, and $r_{n}(\Optime) = 0$.  Then, substituting $v = p_{n}$ in \eqref{adjoint:galerkin:p} and $v = Bq_{n}$ in \eqref{adjoint:galerkin:q}, integrating over $[s,\Optime]$ for $s \in (0,\Optime)$ and summing leads to
\begin{equation}\label{adjoint:est:2}
\begin{aligned}
& \frac{1}{2} \norm{p_{n}(s)}_{L^{2}}^{2} + B \norm{q_{n}}_{L^{2}(s,\Optime;L^{2})}^{2} \\
& \quad \leq A C_{*} \norm{q_{n}}_{L^{2}(s,\Optime;L^{2})} \norm{p_{n}}_{L^{2}(s,\Optime;L^{2})} + \CC C_{*} \norm{r_{n}}_{L^{2}(s,\Optime;L^{2})} \norm{p_{n}}_{L^{2}(s,\Optime;L^{2})} \\
& \quad + (\PP + \AA + \alpha) C_{*} \norm{p_{n}}_{L^{2}(s,\Optime;L^{2})}^{2} \\
& \quad + \left ( \norm{ \beta_{Q}(\varphi - \varphi_{Q})}_{L^{2}(Q)} + \tfrac{1}{2r}  \norm{2 \beta_{\Omega}(\varphi - \varphi_{\Omega}) + \beta_{S}}_{L^{2}(Q)} \right ) \norm{p_{n}}_{L^{2}(s,\Optime;L^{2})},
\end{aligned}
\end{equation}
where we used that $h(\varphi) \leq 1$, $\sigma \leq 1$, $u \leq 1$ a.e. in $Q$, and \eqref{BoundingConst:hPsi}.  Combining \eqref{adjoint:est:1} and \eqref{adjoint:est:2}, and applying Young's inequality and then Gronwall's inequality, we see that
\begin{equation}\label{adjoint:first:apriori} 
\begin{aligned}
& \norm{p_{n}(s)}_{L^{2}}^{2} + \norm{r_{n}(s)}_{L^{2}}^{2} + \norm{q_{n}}_{L^{2}(s,\Optime;L^{2})}^{2} + \norm{\nabla r_{n}}_{L^{2}(s,\Optime;L^{2})}^{2} \\
& \quad \leq C  \left ( \norm{ \beta_{Q}(\varphi - \varphi_{Q})}_{L^{2}(Q)}^{2} + \tfrac{1}{2r}  \norm{2 \beta_{\Omega}(\varphi - \varphi_{\Omega}) + \beta_{S}}_{L^{2}(Q)}^{2} \right ) \text{ for } s \in (0,\Optime),
\end{aligned}
\end{equation}
for some positive constant $C$ depending only on $\CC$, $\PP$, $\AA$, $\alpha$, $C_{*}$, $A$, $B$, and $T$.  This implies that $(p_{n}, q_{n}, r_{n})$ can be extended to the interval $[0,\Optime]$, and thus $t_{n} = 0$ for each $n \in \N$. 

\paragraph{Second estimate.} Viewing \eqref{adjoint:galerkin:q} as the weak formulation of an elliptic problem for $p_{n}$, and using that $q_{n}$ is bounded uniformly in $L^{2}(0,\Optime;L^{2})$, we have by elliptic regularity that
\begin{align*}
\norm{p_{n}}_{L^{2}(0,\Optime;H^{2})} \leq C \left ( \norm{q_{n}}_{L^{2}(0,\Optime;L^{2})} + \norm{p_{n}}_{L^{2}(0,\Optime;L^{2})} \right ),
\end{align*}
for some positive constant $C$ not depending on $n$.

\paragraph{Third estimate.}  Substituting $v = -\pd_{t}r_{n}$ in \eqref{adjoint:galerkin:r}, integrating over $[s,\Optime]$ for $s \in (0,\Optime)$ leads to
\begin{align*}
& \frac{1}{2} \left (\norm{\nabla r_{n}(s)}_{L^{2}}^{2} +\BB \norm{r_{n}(s)}_{L^{2}}^{2} \right ) + \norm{\pd_{t}r_{n}}_{L^{2}(s,\Optime;L^{2})}^{2} \\
& \quad \leq \CC \norm{r_{n}}_{L^{2}(s,\Optime;L^{2})} \norm{\pd_{t}r_{n}}_{L^{2}(s,\Optime;L^{2})} + \PP \norm{p_{n}}_{L^{2}(s,\Optime;L^{2})} \norm{\pd_{t}r_{n}}_{L^{2}(s,\Optime;L^{2})}.
\end{align*}
Thus, by \eqref{adjoint:first:apriori} we have that
\begin{align*}
\norm{\nabla r_{n}(s)}_{L^{2}}^{2} + \norm{\pd_{t}r_{n}}_{L^{2}(s,\Optime;L^{2})}^{2} \leq C \text{ for } s \in (0,\Optime),
\end{align*}
for some positive constant $C$ not depending on $n$.  Furthermore, viewing \eqref{adjoint:galerkin:r} as a weak formulation of an elliptic problem for $r_{n}$ and elliptic regularity yields that
\begin{align*}
\norm{r_{n}}_{L^{2}(0,\Optime;H^{2})} \leq C \left ( \norm{\pd_{t}r_{n}}_{L^{2}(0,\Optime;L^{2})} + \norm{r_{n}}_{L^{2}(0,\Optime;L^{2})} + \norm{p_{n}}_{L^{2}(0,\Optime;L^{2})} \right ),
\end{align*}
for some positive constant $C$ not depending on $n$.

\paragraph{Fourth estimate.}
Integrating \eqref{adjoint:galerkin:p} over $[0,\Optime]$ and integrate by parts, by H\"{o}lder's inequality we obtain that
\begin{align*}
& \abs{\int_{0}^{\Optime} \int_{\Omega} \pd_{t}p_{n} v \dx \dt } \leq B\norm{q_{n}}_{L^{2}(0,\Optime;L^{2})} \norm{\Laplace v}_{L^{2}(0,\Optime;L^{2})}\\
& \quad  + \left ( AC_{*} \norm{q_{n}}_{L^{2}(0,\Optime;L^{2})}  + \CC C_{*} \norm{r_{n}}_{L^{2}(0,\Optime;L^{2})} +  C_{*} C_{u} \norm{p_{n}}_{L^{2}(0,\Optime;L^{2})}\right ) \norm{v}_{L^{2}(0,\Optime;L^{2})} \\
& \quad + \left ( \norm{ \beta_{Q}(\varphi - \varphi_{Q})}_{L^{2}(Q)} + \tfrac{1}{2r}  \norm{2 \beta_{\Omega}(\varphi - \varphi_{\Omega}) + \beta_{S}}_{L^{2}(Q)} \right ) \norm{v}_{L^{2}(0,\Optime;L^{2})},
\end{align*}
which yields that $\{\pd_{t}p_{n}\}_{n \in \N}$ is bounded uniformly in $L^{2}(0,\Optime;(H^{2})^{*})$.  

It follows from the a priori estimates that we obtain a relabelled subsequence $(p_{n}, q_{n}, r_{n})$ such that
\begin{alignat*}{3}
p_{n} & \to p \text{ weakly*} && \text{ in } L^{2}(0,\Optime;H^{2}) \cap H^{1}(0,\Optime;(H^{2})^{*}) \cap L^{\infty}(0,\Optime;L^{2}), \\
q_{n} & \to q \text{ weakly } && \text{ in }  L^{2}(0,\Optime;L^{2}), \\
r_{n} & \to r \text{ weakly* } && \text{ in }  L^{\infty}(0,\Optime;H^{1}) \cap H^{1}(0,\Optime;L^{2}) \cap L^{2}(0,\Optime;H^{2}),
\end{alignat*}
and by standard arguments the triplet $(p,q,r)$ satisfies \eqref{adjoint:weak} and is a solution to the adjoint system \eqref{eq:adjoint}.

\paragraph{Uniqueness.} Let $p := p_{1} - p_{2}$, $q := q_{1} - q_{2}$ and $r := r_{1} - r_{2}$ denote the difference between two solutions to the adjoint system \eqref{eq:adjoint}.  Then, it holds that
\begin{subequations}\label{adjoint:weak:difference}
\begin{align}
\label{adjoint:diff:p} 0 & = \inner{-\pd_{t}p}{\zeta}_{H^{2}} + \int_{\Omega} B q \Laplace \zeta - A \Psi''(\varphi) q \zeta +  h'(\varphi) \left ( \CC \sigma r - (\PP \sigma - \AA - \alpha u)p \right ) \zeta \dx, \\
\label{adjoint:diff:q} 0 & = \int_{\Omega} q \eta - \eta \Laplace p  \dx, \\
\label{adjoint:diff:r} 0 & = \int_{\Omega} - \pd_{t}r \eta + \nabla r \cdot \nabla \eta + \BB r \eta + \CC h(\varphi) r \eta - \PP h(\varphi) p \eta \dx
\end{align}
\end{subequations}
for a.e. $t \in (0,\Optime)$ and for all $\eta \in H^{1}$ and $\zeta \in H^{2}$, with $p(\Optime) = r(\Optime) = 0$.  Substituting $\zeta = p \in L^{2}(0,T;H^{2})$ in \eqref{adjoint:diff:p} and integrate by parts, substituting $\eta = Bq$ in \eqref{adjoint:diff:q} and $\eta = r$ in \eqref{adjoint:diff:r}, summing and then integrate over $[s,\Optime]$ for $s \in (0,\Optime)$ leads to
\begin{align*}
& \frac{1}{2} \left ( \norm{p(s)}_{L^{2}}^{2} + \norm{r(s)}_{L^{2}}^{2} \right ) + B \norm{q}_{L^{2}(s,\Optime;L^{2})}^{2} + \norm{\nabla r}_{L^{2}(s,\Optime;L^{2})}^{2} \\
& \quad \leq A C_{*} \norm{q}_{L^{2}(s,\Optime;L^{2})} \norm{p}_{L^{2}(s,\Optime;L^{2})} + \CC C_{*} \norm{r}_{L^{2}(s,\Optime;L^{2})} \norm{p}_{L^{2}(s,\Optime;L^{2})} \\
& \quad + C_{*} (\PP + \AA + \alpha) \norm{p}_{L^{2}(s,\Optime;L^{2})}^{2} + \PP \norm{p}_{L^{2}(s,\Optime;L^{2})} \norm{r}_{L^{2}(s,\Optime;L^{2})},
\end{align*}
where we neglected the nonnegative term $\BB \abs{r}^{2} + \CC h(\varphi) \abs{r}^{2}$ and used that $h(\varphi) \leq 1$, $\sigma \leq 1$, $u \leq 1$ a.e. in $Q$, and \eqref{BoundingConst:hPsi}.  Estimating the right-hand side with Young's inequality and a Gronwall argument shows that
\begin{align*}
\norm{p(s)}_{L^{2}}^{2} + \norm{r(s)}_{L^{2}}^{2} + \norm{q}_{L^{2}(s,\Optime;L^{2})}^{2} + \norm{\nabla r}_{L^{2}(s,\Optime;L^{2})}^{2} \leq 0 \text{ for all } s \in (0,\Optime),
\end{align*}
which yields that $p = q = r = 0$.

\subsection{Simplification of the first order necessary optimality condition for the control}
Let $(u_{*},\Optime)$ denote the minimizer of \eqref{OCProblem} from Theorem \ref{thm:minimizer}, with corresponding state variables $(\varphi_{*},\mu_{*}, \sigma_{*}) = \Soln(u_{*})$ and adjoint variables $(p,q,r)$ associated to $(\varphi_{*}, \mu_{*}, \sigma_{*})$.  For any $u \in \mathcal{U}_{\mathrm{ad}}$, let $w := u - u_{*} \in L^{2}(Q)$ and let $(\Phi, \Xi, \Sigma)$ denote the linearized state variables associated to $w$.  Then, from \eqref{Fdiff:J:wrt:u}, the optimal control $u_{*}$ satisfies the following first order necessary optimality condition,
\begin{equation}\label{FONC:wrt:u}
\begin{aligned}
& \left (\der_{u} \mathcal{J}(u_{*},\Optime) \right ) (u - u_{*}) = \left (\der_{u} \mathcal{J}(u_{*},\Optime) \right ) w \\
& \quad = \beta_{Q} \int_{0}^{\Optime} \int_{\Omega} (\varphi_{*} - \varphi_{Q}) \Phi \dx \dt + \frac{\beta_{\Omega}}{r} \int_{\Optime-r}^{\Optime} \int_{\Omega} (\varphi_{*} - \varphi_{\Omega}) \Phi \dx \dt\\
& \quad + \frac{\beta_{S}}{2r} \int_{\Optime-r}^{\Optime} \int_{\Omega} \Phi \dx \dt + \beta_{u} \int_{0}^{T} \int_{\Omega} u_{*} (u - u_{*}) \dx \dt \geq 0.
\end{aligned}
\end{equation}
Substituting $\zeta = \Phi$ in \eqref{adjoint:weak:p}, $\eta = \Xi$ in \eqref{adjoint:weak:q} and $\eta = \Sigma$ in \eqref{adjoint:weak:r}, integrate over $[0,\Optime]$ leads to
\begin{subequations}
\begin{align}
\label{adjoint:p:Phi} 0 & = \int_{0}^{\Optime} \left ( \inner{-\pd_{t}p}{\Phi}_{H^{2}} + \int_{\Omega} B q \Laplace \Phi - A \Psi''(\varphi_{*})q \Phi + \CC h'(\varphi_{*}) \sigma_{*} r \Phi \dx \right ) \dt \\
\notag & - \int_{0}^{\Optime} \int_{\Omega} h'(\varphi_{*}) (\PP \sigma_{*} - \AA - \alpha u_{*}) p \Phi + \beta_{Q}(\varphi_{*} - \varphi_{Q}) \Phi \dx \dt \\
\notag & - \int_{\Optime-r}^{\Optime} \int_{\Omega} \tfrac{1}{2r} \left ( 2 \beta_{\Omega} (\varphi_{*} - \varphi_{\Omega}) + \beta_{S} \right ) \Phi \dx \dt, \\
\label{adjoint:q:Xi}0 & = \int_{0}^{\Optime} \int_{\Omega} q \Xi + \nabla p \cdot \nabla \Xi \dx \dt, \\
\label{adjoint:r:Sigma} 0 & = \int_{0}^{\Optime} \int_{\Omega} - \pd_{t} r \Sigma + \nabla r \cdot \nabla \Sigma + \BB r \Sigma + \CC h(\varphi_{*}) r \Sigma - \PP h(\varphi_{*}) p \Sigma \dx \dt.
\end{align}
\end{subequations}
Meanwhile, substituting $\zeta = p$ in \eqref{weak:Linearized:varphi}, $\zeta = q$ in \eqref{weak:Linearized:mu} and $\zeta = r$ in \eqref{weak:Linearized:sigma}, integrate over $[0,\Optime]$ leads to
\begin{subequations}
\begin{align}
\label{Linearized:Phi:p} 0 & = \int_{0}^{\Optime} \left ( \inner{\pd_{t}\Phi}{p}_{H^{1}} + \int_{\Omega} \nabla \Xi \cdot \nabla p \dx \right ) \dt \\
\notag & - \int_{0}^{\Optime} \int_{\Omega} h(\varphi_{*})(\PP \Sigma - \alpha (u - u_{*}))p + h'(\varphi_{*})(\PP \sigma_{*} - \AA - \alpha u_{*}) \Phi) p \dx  \dt, \\
\label{Linearized:Xi:q} 0 & = \int_{0}^{\Optime} \int_{\Omega} q \Xi - A \Psi''(\varphi_{*}) q \Phi + B q \Laplace \Phi \dx \dt, \\
\label{Linearized:Sigma:r} 0 & = \int_{0}^{\Optime} \int_{\Omega} \pd_{t} \Sigma r + \nabla r \cdot \nabla \Sigma + \BB \Sigma r + \CC h(\varphi_{*}) \Sigma r + \CC h'(\varphi_{*}) \Phi \sigma_{*} r \dx \dt.
\end{align}
\end{subequations}
Using that $r(\Optime) = 0$, $p(\Optime) = 0$, $\Sigma(0) = 0$, $\Phi(0) = 0$, $\pd_{t} \Phi \in L^{2}(0,T;(H^{1})^{*})$ and $p \in L^{2}(0,T;H^{2})$, we have that 
\begin{subequations}
\begin{align}
\label{Timederivative:pPhi} \int_{0}^{\Optime} \inner{-\pd_{t}p}{\Phi}_{H^{2}} \dt & = \int_{0}^{\Optime} \inner{p}{\pd_{t}\Phi}_{H^{2}} \dt = \int_{0}^{\Optime} \inner{p}{\pd_{t}\Phi}_{H^{1}} \dt,  \\
\label{Timederivative:rSigma} \int_{0}^{\Optime} \int_{\Omega} \pd_{t}r \Sigma \dx \dt & = -\int_{0}^{\Optime} \int_{\Omega} \pd_{t}\Sigma r \dx \dt.
\end{align}
\end{subequations}
Substituting \eqref{Timederivative:rSigma} into \eqref{adjoint:r:Sigma} and comparing with \eqref{Linearized:Sigma:r} leads to
\begin{align}\label{FONC:sigma:r:relation}
\int_{0}^{\Optime} \int_{\Omega} \CC h'(\varphi_{*}) \sigma_{*} \Phi r + \PP h(\varphi_{*})p \Sigma \dx \dt = 0.
\end{align}
Meanwhile, substituting \eqref{Timederivative:pPhi} and \eqref{Linearized:Xi:q} into \eqref{adjoint:p:Phi}, and using \eqref{adjoint:q:Xi} and \eqref{FONC:sigma:r:relation} leads to
\begin{align*}
0 & = \int_{0}^{\Optime} \left ( 
\inner{p}{\pd_{t}\Phi}_{H^{1}} + \int_{\Omega} \nabla p \cdot \nabla \Xi - \PP h(\varphi_{*}) p \Sigma - h'(\varphi_{*}) (\PP \sigma_{*} - \AA - \alpha u_{*}) p \Phi \dx \right ) \dt \\
& - \int_{0}^{\Optime} \int_{\Omega}  \beta_{Q}(\varphi_{*} - \varphi_{Q})  \dx \dt - \int_{\Optime-r}^{\Optime} \int_{\Omega} \tfrac{1}{2r} \left ( 2 \beta_{\Omega} (\varphi_{*} - \varphi_{\Omega}) + \beta_{S} \right ) \Phi \dx \dt.
\end{align*}
Comparing the above equality with \eqref{Linearized:Phi:p} we obtain
\begin{align*}
& \int_{0}^{\Optime} \int_{\Omega} \beta_{Q}(\varphi_{*} - \varphi_{Q})\Phi \dx \dt + \int_{\Optime-r}^{\Optime} \int_{\Omega} \tfrac{1}{2r} \left ( 2 \beta_{\Omega} (\varphi_{*} - \varphi_{\Omega}) + \beta_{S} \right ) \Phi \dx \dt \\
& \quad  = - \int_{0}^{\Optime} \int_{\Omega} h(\varphi_{*}) \alpha p (u - u_{*}) \dx \dt, 
\end{align*}
and upon substituting this into \eqref{FONC:wrt:u} leads to \eqref{FONC:u}.

\bibliographystyle{plain}
\bibliography{OCtumour}

\end{document}